\documentclass[hidelinks,onefignum,onetabnum]{siamart220329}

\usepackage{lipsum}
\usepackage{amsfonts}
\usepackage{graphicx}
\usepackage{epstopdf}
\usepackage{algorithmic}
\ifpdf
  \DeclareGraphicsExtensions{.eps,.pdf,.png,.jpg}
\else
  \DeclareGraphicsExtensions{.eps}
\fi

 \usepackage{amsmath,amssymb}
 \usepackage{mathrsfs}
 \usepackage{dsfont} 
 \usepackage{mathtools} 
 \usepackage{esint} 
 \usepackage{xfrac} 
 \usepackage{defs} 
 \usepackage{fonts} 

 \usepackage[shortlabels]{enumitem} 

\usepackage[noadjust]{cite} 
\usepackage{cleveref}

\newsiamremark{remark}{Remark}
\newsiamremark{hypothesis}{Hypothesis}
\crefname{hypothesis}{Hypothesis}{Hypotheses}
\newsiamthm{claim}{Claim}

\headers{Nonlocal problems with local boundary conditions II}{James M. Scott and Qiang Du}

\title{Nonlocal problems with local boundary conditions II: 
Green's identities and regularity of solutions
}

\author{James M. Scott\thanks{Department of Applied Physics and Applied Mathematics, and the Data Science Institute, Columbia University, New York, NY 10027 
  (\email{jms2555@columbia.edu}, \email{qd2125@columbia.edu}).}
\and Qiang Du\footnotemark[1]}

\usepackage{amsopn}

\begin{document}

\maketitle

\begin{abstract}
We study nonlocal integral equations on bounded domains with finite-range nonlocal interactions that are localized at the boundary.
We establish a Green's identity for the nonlocal operator that recovers the classical boundary integral, which, along with the variational analysis established previously, leads to the well-posedness of these nonlocal problems with various types of classical local boundary conditions.
We continue our analysis via boundary-localized convolutions, 
using them to analyze the Euler-Lagrange equations, which permits us to establish global regularity properties and classical Sobolev convergence to their classical local counterparts. 
\end{abstract}



\section{Introduction}

There has been much recent interest in 
nonlocal problems 
defined on a bounded domain $\Omega\subset \bbR^d$ associated with  a nonlocal integral operator $\cL$ of the form
\begin{equation}\label{eq:Intro:Operator-a}
	\cL u (\bx) = \int_{\Omega} \big( \gamma(\bx,\by) + \gamma(\by,\bx) \big) \Phi_p'( u(\bx)-u(\by) ) \, \rmd \by\,,
\end{equation}
defined for measurable functions $u : \Omega \to \bbR$, a $C^2$ and even function $\Phi_p$ indexed by an exponent $p > 1$, and an often nonnegative kernel $\gamma$. 
These operators appear widely in both analysis and applications; see related studies in 
\cite{barles2014neumann,barlow2009non,braides2022compactness,caffarelli2007extension,caffarelli2011regularity,craig2016blob,carrillo2019blob,Du12sirev,Gilboa-Osher,grinfeld2005non,lovasz2012large,MeKl04,mogilner1999non,Nochetto2015,Coifman05,shi2017,Valdinoci:2009,Kassmann,garcia2020error,bellido2021restricted,Foghem2022,Grubb,d2021connections,bellido2015hyperelasticity,mengesha2015localization,Du2022nonlocal,d2020physically,foss2022convergence,Lipton-2016,Du-NonlocalCalculus,cortazar,DyKa19,Ros16,MengeshaDuElasticity,Du-Zhou2011,Lipton-2014,Silling2010,Valdinoci-peridynamic,Burch2014exit-time,Du-Tian2020,grube2023robust,foss2016differentiability,hepp2023divergence}
and additional references in books \cite{Du2019book,bucur2016nonlocal,andreu2010nonlocal} and review articles \cite{d2022review,d2020numerical} as well as our recent work \cite{scott2023nonlocal} that the present study is based on.
Because of the symmetric form of the kernel, we can identify $\cL$ with a variational form associated with a nonlocal energy of $u$
\begin{equation}
    \int_{\Omega} \int_{\Omega} \gamma (\bx,\by) \Phi_p( u(\bx)-u(\by) ) \, \rmd \by\, \rmd \bx\,.
\end{equation}
In this work we are interested in nonlocal operators $\cL$ whose kernel is of the form
\begin{equation*}   
    \gamma(\bx,\by) =
    \frac{C}{ |\by-\bx|^{\beta} }  \frac{ \mathds{1}_{ \{ |\by-\bx| < \delta \eta(\bx) \}} }{
    (\delta \eta(\bx))^{d+p-\beta} }, \;\; \bx, \by\in\Omega,
\end{equation*}
with exponents $p > 1$, $\beta\in [0, d+p) $, a normalization constant $C>0$, a scaling parameter $\delta>0$,  and a position-dependent weight  $\eta=\eta(\bx)$ to be specified later. 
At any point $\bx\in \Omega$, the parameter $\delta$ and function $\eta=\eta(\bx)$ together determine the   nonlocal interaction neighborhood $\{\by \in \Omega : |\bx-\bx| < \delta \eta(\bx)\}$.
For such a kernel, the nonlocal integral operator, its associated nonlocal energies, and the relevant nonlocal variational problems defined on a bounded domain, mirror the classical nonlinear Laplacian,  associated nonlinear Dirichlet energies, and boundary value problems.

The motivation to study
nonlocal problems, such as $\cL u = f $ in $\Omega$ with given data $f$, together
with classical boundary conditions prescribed only on $\partial \Omega$ is described in more detail in \cite{scott2023nonlocal}. 
Here, we recall two main factors behind such a study:
1) nonlocal constraints may raise unintended concerns about the regularity of solutions \cite{Du2022nonlocal}, and 2)  in practice, developers of simulation codes for applications of nonlocal models
have ample practical reasons to keep local boundary conditions in implementation, even though a nonlocal model might be derived and/or deemed a better modeling choice in the bulk domain of interest.

An important question for implementing a nonlocal problem with local boundary conditions is if the pointwise and variational forms of the problem can be placed in natural correspondence, which prompts a search for a nonlocal analog of the Green's identity for a nonlocal operator like $\cL$ in \eqref{eq:Intro:Operator-a}.
In this second part of a series of works on the analysis of nonlocal variational problems on bounded domains with local boundary conditions imposed via heterogeneous localization, we make progress on this question.
We adopt a general heterogeneous localization strategy for $\cL$, as elucidated later in this section and \Cref{sec:GreensIdentity}. A function $q$ is introduced to control the rate of localization at the boundary, i.e., 
specify precisely
how $\eta$ vanishes at $\partial\Omega$, which helps to characterize desired
boundary conditions. We rigorously establish a nonlocal Green's identity, and use it to connect various formulations of problems associated to the nonlocal operators with local boundary conditions to their variational forms. The latter are well-posed thanks in part to the variational analysis 
carried out in \cite{scott2023nonlocal}.
The study reveals how the nonlocal Green's identity is affected by the rate that the function $\eta(\bx)$ vanishes on the boundary, in comparison with $\dist(\bx, \partial \Omega)$, which helps to ensure consistency with the classical counterpart. This finding bears significant consequences in the application of localization strategies to nonlocal modeling.

In principle, methodologies like ours can work for local boundary conditions that are either homogeneous or inhomogeneous, and are of very general forms. Yet, for illustration, we consider boundary conditions of Dirichlet, Neumann, and Robin types. 
As an example,  we consider the case of a linear nonlocal operator related to a quadratic energy with  $\Phi_2(t)= \frac{t^2}{2}$. Let $\bsnu=\bsnu(\bx)$ denote the outward unit vector normal to $\p \Omega$ at $\bx \in \partial\Omega$.
Under suitable conditions on the exponent $\beta$ and the weight $\eta$ that vanishes on $\partial\Omega$, there is a constant $A_\delta$ such that
\begin{equation*}
    \begin{split}
    \int_{\Omega} &v(\bx) 
    \cL u(\bx)
    \rmd \bx 
    = \int_{\Omega} \int_{\Omega} \gamma(\bx,\by) (u(\bx) - u(\by))(v(\bx)-v(\by)) \, \rmd \by \, \rmd \bx - A_\delta \int_{\p \Omega} \frac{\p u}{\p \bsnu } v \, \rmd \sigma\,,
    \end{split}
\end{equation*}
for any pair of functions $u,v$ in suitable function spaces;
see \Cref{sec:GreensIdentity,sec:NonlocalOperator} for the precise scope of the result.
This gives a nonlocal analogue of the classical Green's identity for the Laplacian operator $-\Delta$, i.e.
\begin{equation*}
    \int_{\Omega} v (- \Delta u) \, \rmd \bx = \int_{\Omega} \grad u \cdot \grad v \, \rmd \bx - \int_{\p \Omega} \frac{\p u}{\p \bsnu } v \, \rmd \sigma\,.
\end{equation*}

Note that the nonlocal Green's identity recovers a classical boundary term, without encompassing information
on the complement of $\Omega$.
The value of $A_\delta$ depends on the nature of $\eta(\bx)$ as $\bx$ approaches the boundary $\partial\Omega$. The nonlocal and local Green's identities have consistent boundary terms, i.e. $A_\delta = 1$ for all $\delta$, if $\eta(\bx)$ is a genuinely superlinear
function of the distance (or generalized distance) to the boundary, while $A_\delta$ tends to $1$ as $\delta$ tends to $0$, ensuring consistency in the local limit.

We recover similar results in the case of more general nonlinear partial differential operators of the form $-\div[ \cA(\grad u)]$, where  $\cA : \bbR^d \to \bbR^d$ is a monotone function with $p$-growth properties, see the precise statements below in \Cref{subsec:Intro:NonlocalOperator}.

Returning to the linear case described above, when $\beta < d$, the integral used to define $\cL$ is absolutely convergent for each $\bx \in \Omega$.
Thus, the Green's identity reveals that weak solutions to the nonlocal problems are actually strong solutions. Moreover, the boundary-localizing property of the operator leads to a global $W^{1,2}$-regularity of weak solutions. Furthermore, we show the $W^{1,2}$-weak convergence of solutions to the solution of the corresponding local equation. In addition, the nonlocal Green's identity further permits us to show the nonlocal-to-local convergence of the solutions' corresponding normal derivatives.
This program is summarized in \Cref{subsec:Intro:LinearProblem}.

\subsection{Nonlocal function spaces}\label{sec:FunctionSpaces}
Throughout the paper, we assume that for $d\geq 1$, $\Omega \subset \bbR^d$ is an open connected set (a domain) that is bounded and Lipschitz.
To describe our main findings, we recall the function space introduced in \cite{scott2023nonlocal}
\begin{equation}
    \mathfrak{W}^{\beta,p}[\delta;q](\Omega) := \{ u \in L^p(\Omega) \, :\, [u]_{ \mathfrak{W}^{\beta,p}[\delta;q](\Omega) } < \infty \}\,, 
\end{equation}
equipped with the norm determined by $
\Vnorm{u}_{\mathfrak{W}^{\beta,p}[\delta;q](\Omega)}^p := \Vnorm{u}_{L^p(\Omega)}^p + [u]_{\mathfrak{W}^{\beta,p}[\delta;q](\Omega)}^p$ and the nonlocal seminorm defined by 
\begin{equation}\label{eq:Intro:NonlocalSeminorm}
[u]_{\mathfrak{W}^{\beta,p}[\delta;q](\Omega)}^p =   \int_{\Omega} \int_{\Omega} \gamma_{\beta,p}[\delta;q](\bx,\by) |u(\by)-u(\bx)|^p \, \rmd \by \, \rmd \bx\,.
\end{equation}
Here,
$p \in [1,\infty)$ and
\begin{equation}
\label{assump:beta}
 \beta \in [0,d+p)
\tag{\ensuremath{\rmA_{\beta}}}
\end{equation}
are taken as assumptions throughout the paper unless noted otherwise. The constant
$\delta > 0$ is the \textit{bulk horizon parameter} and the kernel in \eqref{eq:Intro:NonlocalSeminorm} is defined as
\begin{equation}\label{eq:Intro:kernelgamma}
	\gamma_{\beta,p}[\delta;q](\bx,\by) := \mathds{1}_{ \{ |\by-\bx| < \delta q(\dist(\bx,\p \Omega)) \} } \frac{ C_{d,\beta,p} }{ |\bx-\by|^{\beta} } \frac{1}{ (\delta q(\dist(\bx,\p \Omega)))^{d+p-\beta} }\,.
\end{equation}
For a Lebesgue measurable set $A \subset \bbR^d$, $\mathds{1}_A$ denotes its standard characteristic function.
$C_{d,\beta,p} > 0$ is a normalization constant so that for any $\bx \in \Omega$, 
\begin{equation}\label{eq:Intro:StdKernelNormalization}
    \int_{\bbR^d} \gamma_{\beta,p}[\delta;q](\bx,\by)|\bx-\by|^p \, \rmd \by = \int_{B(0,1)} \frac{ C_{d,\beta,p} }{ |\bsxi|^{\beta-p} } \, \rmd \bsxi =  \frac{ \sqrt{\pi} \Gamma ( \frac{d+p}{2} ) }{ \Gamma(\frac{p+1}{2} ) \Gamma(\frac{d}{2}) } := \overline{C}_{d,p}\,.
\end{equation}
Here, $B(0,1)$ denotes the unit ball centered at the origin in $\mathbb{R}^d$ and $\Gamma(z)$ the Euler gamma function,
and $C_{d,\beta,p} = \overline{C}_{d,p} \frac{d+p-\beta}{\sigma(\bbS^{d-1})}$ for the unit sphere
 $\bbS^{d-1} \subset \bbR^d$ with its surface measure $\sigma$. 
We recall from \cite{scott2023nonlocal} that these constants and exponents are taken 
to assure that
$[v]_{ \mathfrak{W}^{\beta,p}[\delta;q](\Omega) } = \Vnorm{\grad v}_{L^p(\Omega)}$ for any linear function $v=v(\bx)$.

The function $q: [0,\infty) \to [0,\infty)$ is used to characterize the dependence of the localization  on the distance function. It is assumed to satisfy the following:
\begin{equation}\label{assump:NonlinearLocalization}	\begin{aligned}
		i) &\, q \in C^k([0,\infty)) \text{ for some } k \in \bbN \cup \{\infty\},\; q(0)=0 \text{ and } 0 < q(r) \leq r,\; \forall r >  0;\\
		ii) &\, 0 \leq q'(r) \leq 1,\; \forall r \geq 0\,, \text{ and for a fixed } c_q >0,\;  q'(r) > 0,\; \forall r \in (0,c_q]; \\
        iii) &\, \text{there exists } C_q \geq 1 \text{ such that } q(2r) \leq C_q \, q(r)\,\; \forall r \in (0,\infty)\,. \\
	\end{aligned}
	\tag{\ensuremath{\rmA_{q}}}
\end{equation}
Thus the function $\delta q(\dist(\bx,\p \Omega))$ used in \eqref{eq:Intro:kernelgamma} does not exceed $\delta \diam(\Omega)$ for $\bx\in \Omega$, leading to the naming of $\delta$ as the bulk horizon parameter, but the function shrinks to $0$ as $\bx\to \p\Omega$, hence leading to boundary localization. 
The behavior of $q$ will affect the form of the nonlocal Green's identity obtained. For this purpose, we refine the assumptions on $q$ in the following way:
\begin{equation}\label{assump:MomentsOfNonlinLoc}
    \begin{gathered}
        \eqref{assump:NonlinearLocalization} \text{ is satisfied for some } k \geq 2, \text{ and there exists } N \in \{ 1, 2, \ldots, k - 1 \} \\ \text{ such that }
        q(0) = q'(0) = \ldots = q^{(N-1)}(0) = 0\,, \text{ with } \\
        q^{(N)}(0) > 0\,,\; 
        0 \leq q^{(N)}(r) \leq 1 \; \forall r \in [0,\infty)\,, \text{ and }
        q^{(N)} \in C^1_b([0,\infty))\,.
    \end{gathered}
    \tag{\ensuremath{\rmA_{q,N}}}
\end{equation}
Some examples of $q$ satisfying \eqref{assump:NonlinearLocalization} with $k = \infty$ are $q(r) = r$ and $q(r) = \frac{2}{\pi} \arctan(r)$,
which also satisfy \eqref{assump:MomentsOfNonlinLoc} for $N = 1$. 
Another class of examples is $q(r) \approx \frac{1}{j!} \min \{ r^j, 1 \}$ for $j \in \bbN$, mollified in a neighborhood of $r=1$ so that $q \in C^{k}$ for any desired $k$. If $k \geq j + 1$, then these examples satisfy \eqref{assump:MomentsOfNonlinLoc} for $N = j$.
Different positive values of $\delta$ result in the same equivalent space, as recalled from \cite{scott2023nonlocal} in \Cref{thm:FxnSpaceProp} below.

Following \cite{scott2023nonlocal}, we work with 
a generalized distance function $\lambda : \overline{\Omega} \to [0,\infty)$ that allows us to treat nonlocal models with a more general smooth heterogeneous localization function $\eta[\lambda,q](\bx) := q(\lambda(\bx))$, and it is assumed to satisfy
\begin{equation}\label{assump:Localization}
	\begin{aligned}
		i) & \, \text{there exists a constant } \kappa_0 \geq 1 \text{ such that } \\
		&\quad \frac{1}{\kappa_0} \dist(\bx,\p \Omega) \leq \lambda(\bx) \leq \kappa_0 \dist(\bx,\p \Omega),\; \forall  \bx \in \overline{\Omega}\,;\\
        ii) & \, \text{there exists a constant } \kappa_1 > 0 \text{ such that } \\
        &\quad |\lambda(\bx) - \lambda(\by)| \leq \kappa_1 |\bx-\by|, \; \forall \bx,\by \in \Omega; \\
		iii) & \, \lambda \in C^0(\overline{\Omega}) \cap C^{k}(\Omega) \text{ for some } k \in \bbN_0 \cup \{\infty\}; \text{ and } \\
        iv) & \, \text{for each multi-index } \alpha \in \bbN^d_0 \text{ with } |\alpha| \leq k\,, \\
		&\quad \exists \kappa_{\alpha} > 0 \text{ such that } |D^\alpha \lambda(\bx)| \leq \kappa_{\alpha} |\dist(\bx,\p \Omega)|^{1-|\alpha|}, \;  \forall \bx \in \Omega\,.
	\end{aligned} \tag{\ensuremath{\rmA_{\lambda}}}
\end{equation}

A function $\lambda$ with $k = \infty$ and all $\kappa_\alpha$ depending only on $d$ is guaranteed to exist \cite{Stein}.
The distance function itself satisfies \eqref{assump:Localization} for $k = 0$ and $\kappa_1 = 1$.

The rescaled heterogeneous localization function is given by 
\begin{equation}\label{eq:localizationfunction}
	\eta_\delta[\lambda,q](\bx) := \delta \eta[\lambda,q](\bx) := \delta 
 q(\lambda(\bx))\,, \quad\forall \bx \in \Omega\,,
\end{equation}
where $\eta_{1}[\lambda,q] = \eta[\lambda,q]$. 
As in \cite{scott2023nonlocal}, we will write $\eta_\delta[\lambda,q]$ simply as $\eta_\delta$ (with $\eta_1 = \eta$) whenever the dependence is clear from context. 

The maximum admissible value of the bulk horizon parameter $\delta$ is chosen to depend on $\eta(\bx)$ as follows:
\begin{equation}\label{eq:HorizonThreshold}
\begin{gathered} 
\delta \in (0, \underline{\delta}_0) \;\text{ where }\;
\underline{\delta}_0 := \frac{1}{3 \max \{ 1, \kappa_1, C_q \kappa_0^{\log_2(C_q)} \} }\,.
\tag{\ensuremath{\rmA_{\delta}}}
\end{gathered}
\end{equation}
Recall from \cite{scott2023nonlocal}, \eqref{assump:NonlinearLocalization} and \eqref{assump:Localization} guarantee that, for all $\delta < \underline{\delta}_0$ and $\bx$, $\by \in \Omega$, we have $|\eta_\delta[\lambda,q](\bx)| < \frac{1}{3}$ and $|\eta_\delta[\lambda,q](\bx) - \eta_\delta[\lambda,q](\by)| < \frac{1}{3}|\bx-\by|$.
These bounds simplify many of the estimates and coordinate changes used throughout the paper.
Thus, our analysis encompasses the case that $\eta(\bx) = q(\lambda(\bx))$ is a smooth function that allows for specific forms of heterogeneous localization on the boundary $\p \Omega$, i.e., it is constant away from $\p \Omega$ and vanishes as $\bx$ approaches $\p \Omega$; see further discussion in \cite{scott2023nonlocal}.

\subsection{The Nonlocal Operator}\label{subsec:Intro:NonlocalOperator}
To define the nonlocal operator,
we introduce a nonlocal kernel $\rho: \bbR \to [0,\infty)$, a nonnegative even function satisfying the assumption
\begin{equation}\label{assump:VarProb:Kernel}
    \begin{gathered}
    \rho \in L^{\infty}(\bbR)\,, \quad 
    [-c_\rho,c_{\rho}] \subset \supp \rho \Subset (-1,1) \text{ for a fixed constant } c_{\rho} > 0\,.
    \end{gathered}
    \tag{\ensuremath{\rmA_{\rho}}}
\end{equation}
For a given exponent $p > 1$, we let $\Phi_p : \bbR \to \bbR$ be a $C^2$, even, and convex function that satisfies the following $p$-growth conditions for positive constants $c$ and $C$,
\begin{equation}\label{eq:assump:Phi}
 \begin{gathered}
    \forall \, t \in \bbR\,, \; \Phi_p(t) = \Phi_p(-t)\,,\; c|t|^p \leq \Phi_p(t) \leq C |t|^p  \,, \\  
    |\Phi_p'(t)| \leq C |t|^{p-1}\,, \; 
    \Phi_p'(t)t \geq |t|^p\,, \; \text{ and } \; c |t|^{p-2} \leq |\Phi_p''(t)| \leq C|t|^{p-2}\,.
    \end{gathered}
 \tag{\ensuremath{A_\Phi}}
\end{equation}
An example to keep in mind is $\Phi_p(t) = \frac{1}{p}|t|^p$ for $p \geq 2$.

Under the additional assumptions \eqref{assump:beta}, \eqref{assump:NonlinearLocalization}, and \eqref{eq:HorizonThreshold}, the
operators we are interested in are formally induced by the first variation of the nonlocal energy

\begin{equation}\label{eq:Intro:Varprob:Energy}
  \cE_{p,\delta}(u) := \int_{\Omega} \int_{\Omega} \frac{ \rho ( |\by-\bx|/\eta_\delta(\bx))}{ \eta_\delta(\bx)^{d+p-\beta}
  |\bx-\by|^{\beta-p}
  } \Phi_p \Big( \frac{ |u(\bx)-u(\by)|}{|\bx-\by|} \Big) \, \rmd \by \, \rmd \bx\,.
\end{equation}

The first variation of $\cE_{p,\delta}$ defines the following form
\begin{equation}
\label{eq:bilinear}
 \cB_{p,\delta}(u,v) := \int_{\Omega} \int_{\Omega}
    \frac{ \rho \left(  \frac{|\by-\bx|}{\eta_\delta(\bx) } \right) }{ \eta_\delta(\bx)^{d+p-\beta} } 
    \Phi_p' \Big( \frac{u(\bx)-u(\by)}{|\bx-\by|} \Big)
    \frac{v(\bx)-v(\by)}{|\bx-\by|^{\beta-p+1} } \, \rmd \by \, \rmd \bx\,.
\end{equation}
The nonlocal function space $\mathfrak{W}^{\beta,p}[\delta;q](\Omega)$ is the natural choice of energy space for problems associated to $\cB_{p,\delta}$, since the nonlocal seminorm remains the same under perturbations of the heterogeneous localization $\lambda$ and kernel $\rho$; moreover, the space yields a well-defined trace map $T$ on $\partial\Omega$ (see \Cref{thm:FxnSpaceProp}).

Using nonlocal integration by parts in $\cB_{p,\delta}$, we formally get the nonlinear operator
\begin{equation}\label{eq:Intro:Operator}
    \cL_{p,\delta} u(\bx) := \int_\Omega 
    \left(  
    \frac{ \rho \left(  \frac{ |\by-\bx| }{ \eta_\delta(\bx) } \right) }{ \eta_\delta(\bx)^{d+p-\beta} }
    +     \frac{ \rho \left(  \frac{ |\by-\bx| }{ \eta_\delta(\by) } \right) }{ \eta_\delta(\by)^{d+p-\beta} }
    \right)
    \frac{  \Phi_p' \left( \frac{u(\bx)-u(\by)}{|\bx-\by|} \right) }{ |\bx-\by|^{\beta-p+1} } \, \rmd \by\,. 
\end{equation}

Our primary interest in this work is to develop natural variational formulations of nonlocal boundary-value problems that mirror the classical.
As $\delta \to 0$, the variational form $\cB_{p,\delta}$ and the operator $\cL_{p,\delta}$ formally converge to local counterparts, i.e.
    \begin{equation}\label{eq:LocalizedObjectsDef}
\begin{gathered}
        \cB_{p,\delta}(u,v) \to \cB_{p,0}(u,v) := \int_{\Omega} \cA_{p}(\grad u(\bx) ) \cdot \grad v(\bx) \, \rmd \bx\,, \\
          \cL_{p,\delta} u \to \cL_{p,0}u := - \div( \cA_p(\grad u))\,,
\;
    \text{ where } \\
    \cA_{p}(\ba) := 
    \bar{\rho}_{p-\beta} \fint_{\bbS^{d-1}} 
    \Phi_p'( \ba \cdot \bsomega ) \bsomega \, \rmd \sigma(\bsomega)\,,\;
 \bar{\rho}_{p-\beta} := \int_{B(0,1)} |\bz|^{p-\beta} \rho(|\bz|) \, \rmd \bz\,.
    \end{gathered}
\end{equation}
These local objects satisfy a Green's identity, which is a ubiquitous tool in the variational analysis of classical boundary-value problems:
\begin{equation}\label{eq:GreensIdentity:Local}
    \int_{\Omega} v \cL_{p,0} u \, \rmd \bx = \cB_{p,0}(u,v) - \int_{\p \Omega} v \cA_p(\grad u) \cdot \bsnu \, \rmd \sigma\,,
\end{equation}
valid for sufficiently smooth $u$ and $v$, and Lipschitz domains $\Omega$.
Our first main result is the rigorous development of a counterpart of this identity for the nonlocal objects.

Before we state it, we note that due to the singularity on $\p \Omega$, as well as the potential singularity of the kernel on the diagonal $\bx = \by$, the operator $\cL_{p,\delta} u$ may not be a well-defined object in general, even if $u \in C^\infty$. To get around this, we define a truncated form of the operator; for $\veps > 0$ we define
\begin{equation*}
    \cL_{p,\delta}^\veps u(\bx) := \mathds{1}_{ \{ \eta(\bx) > \veps \} } \int_{\Omega} 
    \left( 
    \frac{ \rho^\veps \left( \frac{|\by-\bx|}{ \eta_\delta(\bx) } \right)  }{ \eta_\delta(\bx)^{d+p-\beta}  }  
    + \frac{ \rho^\veps \left( \frac{|\by-\bx|}{ \eta_\delta(\by) } \right)  }{ \eta_\delta(\by)^{d+p-\beta}  } 
    \right)
    \frac{ \Phi_p' \left( \frac{ u(\bx)-u(\by) }{ |\bx-\by| } \right) }{ |\bx-\by|^{\beta-p+1} }
   \, \rmd \by\,,
\end{equation*}
where $\rho^\veps(t) := \mathds{1}_{ \{ \veps \leq |t| < 1\} } \rho(t)$.
We will first understand $\cL_{p,\delta}$ via the dual action of $\cL_{p,\delta}^\veps$ in the limit $\veps \to 0$. The following is a nonlocal Green's identity, which also realizes $\cL_{p,\delta}u$ for smooth functions $u \in C^2(\overline{\Omega})$, and is our first main result. 

\begin{theorem}[Nonlocal Green's identity]\label{thm:Intro:GreensIdentity}
Assume 
$p \geq 2$,
\eqref{assump:beta}, 
\eqref{assump:NonlinearLocalization},
\eqref{assump:MomentsOfNonlinLoc},
\eqref{assump:Localization}, \eqref{assump:VarProb:Kernel}, 
\eqref{eq:assump:Phi},  
\eqref{eq:HorizonThreshold}, and 
\eqref{assump:Localization:NormalDeriv}, given in \Cref{sec:pfGreen} below.
    Let $u \in C^2(\overline{\Omega})$. Then for each $0 < \veps \ll \delta$ there exists an open set $\wt{U}_\veps \Subset \Omega$ depending only on $\Omega$ such that $\lim\limits_{\veps \to 0} |\wt{U}_\veps| = 0$
    and $\vnorm{ \mathds{1}_{\Omega \setminus \wt{U}_\veps} \cL_{p,\delta}^\veps u }_{ [\mathfrak{W}^{\beta,p}[\delta;q](\Omega)]^* } $ is bounded independent of $\veps$. Hence, the sequence
    $\mathds{1}_{\Omega \setminus \wt{U}_\veps} \cL_{p,\delta}^\veps u$ converges as $\veps \to 0$ in the weak-$*$ topology on $ [\mathfrak{W}^{\beta,p}[\delta;q](\Omega)]^* $ to a limit defined as $(\mathrm{d}\text{-}\cL)_{p,\delta} u$. Further, the action of $(\mathrm{d}\text{-}\cL)_{p,\delta}$ is given by
    \begin{equation}\label{eq:GreensIdentity}
        \begin{split}
            \vint{ (\mathrm{d}\text{-}\cL)_{p,\delta} u, v } :=& \lim\limits_{\veps \to 0} \int_{\Omega \setminus \wt{U}_\veps } \cL_{p,\delta}^\veps u(\bx) v(\bx) \, \rmd \bx \\
            =&\, \cB_{p,\delta}(u,v) - \int_{\p \Omega} BF_{p,\delta}^N(\grad u,\bsnu) v \, \rmd \sigma\,, \quad \forall v \in \mathfrak{W}^{\beta,p}[\delta;q](\Omega)\,,
        \end{split}
    \end{equation}
    The function $BF_{p,\delta}^N : \bbR^d \times \bbS^{d-1} \to \bbR$ is defined for $N>1$ and $N=1$ respectively as
    \begin{equation*}
        \begin{split}
        BF_{p,\delta}^N(\ba,\bstheta) &:= 
            \int_{B(0,1) } \frac{\bz \cdot \bstheta}{|\bz|} \frac{\rho(|\bz|)}{ |\bz|^{\beta-p} } \Phi_p' \left( \ba \cdot \frac{\bz}{|\bz|} \right)  \rmd \bz 
            = \cA_p(\ba) \cdot \bstheta\,, \quad \text{ for } N > 1\,,\\
        BF_{p,\delta}^1(\ba,\bstheta) & := \int_{B(0,1)} \ln \left( \frac{ 1+ \delta q'(0) \bstheta \cdot \bz }{ 1 - \delta q'(0) \bstheta \cdot \bz } \right)  \frac{1}{ 2 \delta q'(0)|\bz|} \frac{\rho(|\bz|)}{ |\bz|^{\beta-p} } \Phi_p' \left( \ba \cdot \frac{\bz}{|\bz|} \right)  \rmd \bz\,.
 \end{split}
    \end{equation*}
\end{theorem}
The proof of \Cref{thm:Intro:GreensIdentity} is in \Cref{sec:pfGreen}.
Observe that the normal boundary flux (BF) terms in \eqref{eq:GreensIdentity:Local} and \eqref{eq:GreensIdentity} coincide when $N > 1$, i.e., 
if a perfect match of the boundary flux between local and nonlocal problems is desired, one may let $q$ be genuinely nonlinear.
Also observe that for $N = 1$,  $BF_{p,\delta}^1(\ba,\bstheta) \to \cA_p(\ba) \cdot \bstheta$ as $\delta \to 0$, hence in all situations the boundary flux in the nonlocal Green's identity is formally consistent with the local Green's identity in the vanishing horizon limit.

Now we seek to extend \eqref{eq:GreensIdentity} to more general contexts. First, there is the definition of the nonlocal operator; the formal expression \eqref{eq:Intro:Operator} may not define an absolutely convergent integral when $\beta$ is close to $d+p$. Further, due to the translation variance of the kernel as well as its singularity at the boundary, the integral may not converge in the principal value sense, even when $u \in C^{\infty}(\overline{\Omega})$.
In order to guarantee that $\cL_{p,\delta} u$ can be defined pointwise by the integral \eqref{eq:Intro:Operator} for $\beta$ small enough,
we introduce the sufficient condition
\begin{equation}\label{assump:FullLocalization:C11}
    \eta \in W^{2,\infty}(\Omega)\,.
    \tag{\ensuremath{\rmA_{\eta}}}
\end{equation}
\begin{theorem}\label{thm:Intro:PointwiseOperator}
    Take all assumptions of \Cref{thm:Intro:GreensIdentity}, and additionally assume \eqref{assump:FullLocalization:C11} and  $\beta < d$.
    Let $u \in C^2(\overline{\Omega})$ and
    $r \in [1,\infty)$. Then $\cL_{p,\delta}^\veps u(\bx) \to \cL_{p,\delta} u(\bx)$ as $\veps \to 0$
    in $L^r(\Omega)$, where $\cL_{p,\delta} u \in L^1_{loc}(\Omega)$ is defined via the absolutely convergent integral \eqref{eq:Intro:Operator} for all $\bx \in \Omega$.
    There exists a constant $C = C(\rho,\beta,p,r) > 0$ such that
    $$
    \vnorm{\cL_{p,\delta} u}_{L^r(\Omega)} \leq C \big(\vnorm{\grad u}_{L^{r(p-1)}(\Omega)}^{p-1} 
        + \vnorm{\grad^2 u}_{L^{r(p-1)}(\Omega)}^{p-1}
        \big)\,,
    $$
    and $\cL_{p,\delta} u \to \cL_{p,0} u$ as $\delta \to 0$
    strongly in $L^r(\Omega)$.
    Moreover, 
    $\cL_{p,\delta} u = (\mathrm{d}\text{-}\cL)_{p,\delta} u$ in the space $[\mathfrak{W}^{\beta,p}[\delta;q](\Omega)]^*$, i.e. the nonlocal Green's identity \eqref{eq:GreensIdentity} holds for all $u \in C^2(\overline{\Omega})$ and all $v \in \mathfrak{W}^{\beta,p}[\delta;q](\Omega)$ with $(\mathrm{d}\text{-}\cL)_{p,\delta} u$ replaced with $\cL_{p,\delta} u$. 
\end{theorem}
Each of the statements in \Cref{thm:Intro:PointwiseOperator} holds in a more general context; see \Cref{subsec:PointwiseOperator} for details.
See also this section for a list of different possible sufficient conditions on $\eta$ that guarantee \eqref{assump:FullLocalization:C11}.
We note here that the proof of the localization result simplifies and generalizes the same calculations for the fractional $p$-Laplacian carried out in \cite{ishii2010class}.

To describe a second generalization of \eqref{eq:GreensIdentity}, in the linear case of $\Phi_2(t) = \frac{t^2}{2}$ we extend the domain of the operator $\cL_{2,\delta}$ from $C^2(\overline{\Omega})$ to $W^{2,r}(\Omega)$ for any $r \in [1,\infty)$ by defining a distributional form. This definition allows the treatment of singularity on the diagonal, and also implies a Green's second identity; see \Cref{thm:WeakOperator}.

\subsection{Boundary-localized convolutions}\label{sec:LocalizedConvolution}

An essential tool in the analysis of related function spaces and minimization problems in \cite{scott2023nonlocal}, and thus the analysis of the nonlocal operators and regularity of solutions in this work, is the \textit{boundary-localized convolution} operator
\begin{equation}\label{eq:ConvolutionOperator}
	K_{\delta}u (\bx) = K_\delta[\lambda,q,\psi](\bx) := \int_{\Omega} \frac{1}{(\eta_\delta[\lambda,q](\bx))^d} \psi \left( \frac{|\by-\bx|}{ \eta_\delta[\lambda,q](\bx) } \right) u(\by) \,\rmd \by , \; \bx \in \Omega\,.
\end{equation}
Here, $\psi:\bbR \to [0,\infty)$ is a standard mollifier with $\eta_\delta[\lambda,q]$ defined as before.
Specifically, $\psi$ is assumed to be a nonnegative even function that satisfies
for some $ k \in \bbN_0 \cup \{\infty\}$
\begin{equation}\label{Assump:Kernel}
    \begin{gathered}
    \psi \in C^{k}(\bbR), \;
    [-c_\psi,c_{\psi}] \subset \supp \psi \Subset (-1,1) \text{ for fixed } c_{\psi} > 0, \; 
    \int_{\bbR^d} \psi(|\bx|) \, \rmd \bx = 1.
    \end{gathered}
    \tag{\ensuremath{\rmA_{\psi}}}
\end{equation}
The boundary-localized convolution $K_\delta[\lambda,q,\psi]$ is named as such in \cite{scott2023nonlocal} because it has all of the smoothing properties of classical convolution operators, and additionally recovers the boundary values of a function. 
Here, we also use its formal adjoint
\begin{equation}\label{eq:ConvolutionOperator:Adj}
    K_\delta^*u(\bx) = K_\delta^* [\lambda,q,\psi](\bx) := \int_{\Omega} \frac{1}{(\eta_\delta[\lambda,q](\by))^d} \psi \left( \frac{|\by-\bx|}{ \eta_\delta[\lambda,q](\by) } \right) u(\by) \,\rmd \by , \; \bx \in \Omega.
\end{equation}
The same convention used to abbreviate the heterogeneous localization $\eta_\delta[\lambda,q] = \eta_\delta$ is applied here to abbreviate 
$K_\delta[\lambda,q,\psi]$ and $K_\delta^*[\lambda,q,\psi]$
as $K_\delta$ and $K_\delta^*$ respectively.

We build on the analysis of $K_\delta$ in \cite{scott2023nonlocal}
(see a summary of the relevant properties in \Cref{thm:KnownConvResults})
to study the adjoints $K_\delta^*$, as well as the more general convolution-type operators $J_{\delta,\beta}$ introduced later in \Cref{sec:BdyLocalizedConv}.

\subsection{Nonlocal boundary-value problems}\label{sec:varprobs}
The Green's identity allows us to formally identify boundary-value problems with corresponding weak formulations. This identification can be made rigorous in some contexts by using \Cref{thm:Intro:PointwiseOperator},
but nevertheless it can additionally be shown that these weak formulations are well-posed in a more general context, without requiring the operator $\cL_{p,\delta}$ to be a defined via an integral.
Thus, we continue our discussion formally, keeping in mind the prior results. The nonlocal equation we consider has principal part $\cL_{p,\delta}$, and is defined as
\begin{equation}\label{eq:Fxnal}
    \cL_{p,\delta} u + \mu K_\delta^* [\bar{\lambda},q,\psi] \big( \ell'(K_\delta[\bar{\lambda},q,\psi] u) \big) = f\,,
\end{equation}
where $\mu \geq 0$ is a constant, and we assume that
$\psi$ satisfies \eqref{Assump:Kernel} for $k_\psi \geq 1$ and $\bar{\lambda}$ satisfies \eqref{assump:Localization}.
The latter may not necessarily be equal to the $\lambda$ used for the nonlocal functional given in 
\eqref{eq:Intro:Varprob:Energy}.
The  map $\ell \in C^1(\bbR)$ and its derivative 
$\ell'$ are assumed to satisfy continuity and growth assumptions given in \Cref{subsec:BVPReview}.
By introducing the convolution $K_\delta$, this semilinear term allows us to consider lower-order terms that, without mollification, may not be continuous in the nonlocal function space. 

We treat boundary-value problems associated to \eqref{eq:Fxnal}, using the Green's identity to place each of them into correspondence with a variational form. There exists a unique variational, or weak, solution of each equation, since the variational form can be realized as an Euler-Lagrange equation associated to a functional studied in \cite{scott2023nonlocal}. In the same way, in the vanishing horizon limit $\delta \to 0$, the $\Gamma$-convergence results of \cite{scott2023nonlocal} assure
the strong $L^p(\Omega)$-convergence of weak solutions to a weak solution of the corresponding local boundary-value problem, as summarized in \Cref{subsec:BVPReview}.

The first nonlocal problem we treat is one with an inhomogeneous Dirichlet-type constraint on $\p \Omega$, or more generally $\p \Omega_D$, a $\sigma$-measurable subset of $\p \Omega$ with a positive measure $\sigma(\p \Omega_D) > 0$.
Let $g \in W^{1-1/p,p}(\p \Omega_D)$, and define the admissible set 
$$
\mathfrak{W}^{\beta,p}_{g,\p \Omega_D}[\delta;q](\Omega) := \{ u \in \mathfrak{W}^{\beta,p}[\delta;q](\Omega) \, : \, u = g \text{ on } \p \Omega_D \text{ in the trace sense } \}\,.
$$
Given $f \in [\mathfrak{W}^{\beta,p}[\delta;q](\Omega)]^*$, we say that $u \in \mathfrak{W}^{\beta,p}_{g,\p \Omega_D}[\delta;q](\Omega)$ is a weak solution of 
\begin{equation}\label{eq:BVP:Dirichlet}
    \cL_{p,\delta} u + \mu K_\delta^*[ \ell'(K_\delta u)] = f \text{ in } \Omega\,, \qquad 
    \begin{aligned}
        Tu &= g \text{ on } \p \Omega_D\,, \\
        BF_{p,\delta}^N(\grad u,\bsnu) &= 0 \text{ on } \p \Omega \setminus \p \Omega_D\,.
    \end{aligned}
\end{equation}
if $u$ is in the admissible set $\mathfrak{W}^{\beta,p}_{g,\p \Omega_D}[\delta;q](\Omega)$ and satisfies
\begin{equation}\label{eq:BVP:Dirichlet:Weak}
\begin{gathered}
    \cB_{p,\delta}(u,v) + \mu \int_{\Omega} \ell'(K_\delta u) K_\delta v \, \rmd \bx  = \vint{f,v}\,, \quad \forall \, v \in \mathfrak{W}^{\beta,p}_{0,\p \Omega_D}[\delta;q](\Omega)\,, 
\end{gathered}
\tag{\ensuremath{\rmD_\delta}}
\end{equation}
where $\vint{\cdot,\cdot}$ denotes the natural duality pairing.
A special case is when $g \equiv 0$ on $\p \Omega_D$, for which we consider solutions in the Banach space
\begin{equation}\label{eq:HomNonlocSpDef}
	\begin{split}
		\mathfrak{W}^{\beta,p}_{0,\p \Omega_D}[\delta;q](\Omega) := \{ \text{closure of } C^{1}_c(\overline{\Omega} \setminus \p \Omega_D) \text{ with respect to } \Vnorm{\cdot}_{\mathfrak{W}^{\beta,p}[\delta;q](\Omega)} \}\,.
	\end{split}
\end{equation}
Then we say that $u$ is a weak solution in the admissible space $\mathfrak{W}^{\beta,p}_{0,\p \Omega_D}(\Omega)$
of \eqref{eq:BVP:Dirichlet} with Poisson data $f \in [\mathfrak{W}^{\beta,p}_{0,\p \Omega_D}[\delta;q](\Omega)]^*$ and $g=0$ if $u\in \mathfrak{W}^{\beta,p}_{0,\p \Omega_D}(\Omega)$  satisfies \eqref{eq:BVP:Dirichlet:Weak}.

We also treat problems with a pure Neumann boundary condition. 
Given $f \in [W^{1,p}(\Omega)]^*$ and $g \in [W^{1-1/p,p}(\p \Omega)]^*$, consider the problem
\begin{equation}\label{eq:BVP:Neumann}
    \cL_{p,\delta} u = f \text{ in } \Omega\,, \qquad BF_{p,\delta}^N(\grad u,\bsnu) = g \text{ on } \p \Omega\,.
\end{equation}
It is apparent that constants satisfy the homogeneous equation, and from the Green's identity that the compatibility condition $\vint{f,1} + \vint{g,1} = 0$ is required for the existence of a solution. Here, $\vint{g,v}$ also denotes the natural duality pairing for objects defined on $\p \Omega$.
Thus, we take the class of admissible solutions to be the nonlocal space
$$
\mathring{\mathfrak{W}}^{\beta,p}[\delta;q](\Omega) := \{ u \in \mathfrak{W}^{\beta,p}[\delta;q](\Omega) \, : \, (u)_\Omega = 0 \}\,,
$$
where $(u)_\Omega = \frac{1}{|\Omega|} \int_{\Omega} u(\bx) \, \rmd \bx = \fint_{\Omega} u(\bx) \, \rmd \bx $ denotes the integral average of $u$ over $\Omega$.
We say that $u$ is a weak solution in the admissible space $\mathring{\mathfrak{W}}^{\beta,p}[\delta;q](\Omega)$ of \eqref{eq:BVP:Neumann} if 
\begin{equation}\label{eq:BVP:Neumann:Weak}
\begin{gathered}
    \cB_{p,\delta}(u,v) = \vint{f,v} + \vint{g,v}\,, \quad \forall \, v \in \mathring{\mathfrak{W}}^{\beta,p}[\delta;q](\Omega)\,.
\end{gathered}
\tag{\ensuremath{\rmN_\delta}}
\end{equation}

The final type of nonlocal problem is one with a Robin-type constraint.
Given $b \in L^{\infty}(\p \Omega)$, $f \in [W^{1,p}(\Omega)]^*$, and $g \in [W^{1-1/p,p}(\p \Omega)]^*$, we say that $u \in \mathfrak{W}^{\beta,p}[\delta;q](\Omega)$ is a weak solution of
\begin{equation}\label{eq:BVP:Robin}
    \cL_{p,\delta} u + \mu K_\delta^*[ \ell'(K_\delta u)] = f \text{ in } \Omega\,, \qquad BF_{p,\delta}^N(\grad u,\bsnu) + b \Phi_p'(u) = g \text{ on } \p \Omega\,,
\end{equation}
if $u$ in the admissible space $\mathfrak{W}^{\beta,p}[\delta;q](\Omega)$ satisfies
\begin{equation}\label{eq:BVP:Robin:Weak}
\begin{gathered}
    \cB_{p,\delta}(u,v) + \mu \int_{\Omega} \ell'(K_\delta u) K_\delta v \, \rmd \bx + \int_{\p \Omega} b \Phi_p'(Tu) Tv \, \rmd \sigma = \vint{f,v} + \vint{g,v}\,, \\ 
    \forall \, v \in \mathfrak{W}^{\beta,p}[\delta;q](\Omega)\,.
\end{gathered}
\tag{\ensuremath{\rmR_\delta}}
\end{equation}

As $\delta \to 0$, solutions to these nonlocal problems converge to the solution of the corresponding local boundary value problem with principal part $\cB_{p,0}$. 
This is a consequence of the variational analysis performed in \cite{scott2023nonlocal}.
Note that much more general lower-order terms could be treated using the analysis of \cite{scott2023nonlocal}, but to simplify the presentation in this work we consider the simplified form above, which still captures the essence of the analysis.
For instance, some subtleties remain in the convergence results for solutions to problems with nonconvex semilinearities. We will cover such questions in a future work, and be content here with analyzing boundary-value problems for which the solution is unique, and hence is a minimizer of an energy functional.
This result is in the same spirit as the program carried out in \cite{ponce2004new, mengesha2015VariationalLimit}, in which nonlocal models are shown to be consistent with appropriate classical counterparts.

\subsection{Examples}\label{subsec:Examples}

To demonstrate the scope of our analysis, we present some examples related to $p$-Laplace-type operators.
First, we note that the identity 
\begin{equation}\label{eq:GammaSphereInt}
        \fint_{\bbS^{d-1}} |\ba \cdot \bsomega|^{p} \, \rmd \sigma(\bsomega) = \frac{1}{ \overline{C}_{d,p}} |\ba|^p\,, \quad \ba \in \bbR^d\,,
\end{equation}
can be seen by transforming the right-hand side via spherical coordinates into a beta integral.
Then, we set $\Phi_p(t) = \frac{|t|^p}{p}$
and let $\rho$ be a kernel satisfying $\bar{\rho}_{p-\beta} = \overline{C}_{d,p}$. By
performing a variation with respect to $\ba$ in  \eqref{eq:GammaSphereInt}, we get
\begin{equation}\label{eq:MonotoneFormDerivation}
    \cA_p(\ba) =  \bar{\rho}_{p-\beta} \fint_{\bbS^{d-1}} |\ba \cdot \bsomega|^{p-2}(\ba \cdot \bsomega) \bsomega \, \rmd \sigma(\bsomega) = \frac{ \bar{\rho}_{p-\beta} }{ \overline{C}_{d,p} } |\ba|^{p-2}  \ba = |\ba|^{p-2}  \ba
    \,;
\end{equation}
these are used implicitly in the following examples.

\textit{Example 1: Dirichlet constraints for a fractional $p$-Laplace equation.}
Under the assumptions of \Cref{thm:Intro:GreensIdentity}, take $p \geq 2$, 
set $\beta = d+sp$ for some $s \in (0,1)$ and assume that $N > 1$.
Then the Green's identity in this case is
\begin{equation*}
    \lim\limits_{\veps \to 0} \vint{\cL_{p,\delta}^\veps u,v} = \cB_{p,\delta}(u,v) - \int_{\p \Omega} |\grad u|^{p-2} \frac{\p u}{\p \bsnu} v \, \rmd \sigma\,.
\end{equation*}
In the context of \eqref{eq:BVP:Dirichlet:Weak}, let $\mu = 0$, $\p \Omega_D = \p \Omega$, let $g = 0$ and let $f \in [\mathfrak{W}^{d+sp,p}_{0, \p \Omega}[\delta;q] (\Omega)]^*$. 
The Dirichlet problem obtained from the Green's identity for $u \in \mathfrak{W}^{d+sp,p}_{0,\p \Omega}[\delta;q](\Omega)$ is 
\begin{equation*}
    \begin{split}
        \int_{\Omega} \int_{\Omega} \frac{ \rho_\delta(\bx,\by) }{ \eta_\delta(\bx)^{(1-s)p} } \frac{ |u(\bx)-u(\by)|^{p-2} (u(\bx)-u(\by)) (v(\bx) - v(\by)) }{ |\bx-\by|^{d+sp} } \, \rmd \by \, \rmd \bx = \vint{ f, v }\,,
    \end{split}
\end{equation*}
for all $v \in \mathfrak{W}^{d+sp,p}_{0,\p \Omega}[\delta;q](\Omega)$. Solutions $u_\delta$ to this problem converge as $\delta \to 0$ in $L^p(\Omega)$ to a weak solution $u \in W^{1,p}_{0,\p \Omega}(\Omega)$ of
\begin{equation*}
    \int_{\Omega} |\grad u|^{p-2} \grad u \cdot \grad v \, \rmd \bx = \vint{f,v}\,, \qquad \forall \, v \in W^{1,p}_{0,\p \Omega}(\Omega)\,.
\end{equation*}

\textit{Example 2: A linear Neumann problem.}
Under the assumptions of \Cref{thm:Intro:GreensIdentity}, 
we take $p = 2$, and $\beta < d$, and assume additionally that \eqref{assump:FullLocalization:C11} holds. 
We abbreviate, in this special case, the operator $\cL_{2,\delta}$ as $\cL_\delta$, i.e., for $u:\Omega\to \mathbb{R}$,
\begin{equation}\label{eq:LinearOperatorDefinition}
    \cL_{\delta} u(\bx) := \int_\Omega 
    \left(  
    \frac{ \rho \left(  \frac{ |\by-\bx| }{ \eta_\delta(\bx) } \right) }{ \eta_\delta(\bx)^{d+2-\beta} }
    +     \frac{ \rho \left(  \frac{ |\by-\bx| }{ \eta_\delta(\by) } \right) }{ \eta_\delta(\by)^{d+2-\beta} }
    \right)
    \frac{ u(\bx)-u(\by) }{ |\bx-\by|^{\beta} } \, \rmd \by\,,
\end{equation}
which is an absolutely convergent integral for $u \in C^2(\overline{\Omega})$ by \Cref{thm:Intro:PointwiseOperator}, and $\cL_\delta u \to - \frac{\bar{\rho}_{2-\beta}}{ \overline{C}_{d,2} } \Delta u = - \Delta u$ in $L^r(\Omega)$.
In this case $\cB_{2,\delta}$ becomes a bilinear form, which we also abbreviate as $\cB_{2,\delta}(u,v) = \cB_\delta(u,v)$. 
Then the Green's identity \eqref{eq:GreensIdentity} in this case is
\begin{equation}\label{eq:Intro:GreensIdentity:LinearForm}
    \int_{\Omega} \cL_{\delta }u \, v \, \rmd \bx = \cB_{\delta}(u,v) - A_\delta^N \int_{\p \Omega} \frac{\p u}{\p \bsnu} v \, \rmd \sigma\,,
\end{equation}
where the constant $A_{\delta}^N$ is defined as
    \begin{equation}
    \begin{split}
        A_\delta^N := \frac{\bar{\rho}_{2-\beta}}{ \overline{C}_{d,2} } =1 \text{ for } N > 1, \text{ and } A_{\delta}^1 := 
        \int_{B(0,1)} \frac{ \ln \left( \frac{ 1+\delta q'(0) z_d }{  1-\delta q'(0) z_d } \right)  }{ 2 \delta q'(0) |\bz| }  \frac{ \rho(|\bz|) }{ |\bz|^{\beta-2} } \frac{z_d}{|\bz|}  \, \rmd \bz \,.
    \end{split}
    \end{equation}
The constant $A_{\delta}^1$ can be obtained from $BF_{2,\delta}^1$ via a rotational change of coordinates.
Note that when $q'(0) > 0$, $A_\delta^1 > 1$ for  all $\delta>0$, but as $\delta \to 0$, 
 $A_\delta^1 \to \frac{\bar{\rho}_{2-\beta}}{\overline{C}_{d,2}} = 1$, which is consistent.

Consider \eqref{eq:BVP:Neumann:Weak} with $p=2$; let $g \in [W^{1/2,2}(\p \Omega)]^*$ and let $f \in [\mathfrak{W}^{\beta,2}[\delta;q](\Omega)]^*$. 
Then the Neumann problem we obtain from the Green's identity is
\begin{equation*}
    \begin{split}
        \int_{\Omega} \int_{\Omega} \frac{ \rho \left(  \frac{ |\by-\bx| }{ \eta_\delta(\bx) } \right) }{ \eta_\delta(\bx)^{d+2-\beta} } \frac{ (u(\bx)-u(\by)) (v(\bx) - v(\by)) }{ |\bx-\by|^\beta } \, \rmd \by \, \rmd \bx = \vint{ f, v } +  \vint{g,v}\,,
    \end{split}
\end{equation*}
for all $v \in \mathring{\mathfrak{W}}^{\beta,2}[\delta;q](\Omega)$. Solutions $u_\delta \in \mathring{\mathfrak{W}}^{\beta,2}[\delta;q](\Omega)$ to this problem converge as $\delta \to 0$ in $L^2(\Omega)$ to a function $u \in \mathring{W}^{1,2}(\Omega)$ that satisfies
\begin{equation*}
    \cB_0(u,v) := \int_{\Omega} \grad u \cdot \grad v \, \rmd \bx = \vint{f,v} + \vint{g,v}\,, \qquad \forall \, v \in \mathring{W}^{1,2}(\Omega)\,.
\end{equation*}

\subsection{Regularity and improved convergence for a linear problem}\label{subsec:Intro:LinearProblem}

In the previous example, the operator $\cL_\delta u$ converges strongly, and moreover can be separated since the diagonal singularity is integrable.
Therefore as a consequence of the Green's identity, if $u_\delta$ is a weak solution of $\cL_\delta u_\delta = f$, then it satisfies $\cL_\delta u_\delta = f$ in the strong sense, and therefore satisfies, for $K_\delta = K_\delta[\lambda,q,|\cdot|^{-\beta} \rho(\cdot)]$ and $K_\delta^* = K_\delta^*[\lambda,q,|\cdot|^{-\beta} \rho(\cdot)]$,
\begin{equation*}
    u(\bx) = \frac{ K_\delta u(\bx) + K_\delta^*(\eta_\delta^{-2} u)(\bx) + f(\bx) }{ \bar{\rho}_{-\beta} + K_\delta^*(\eta_\delta^{-2})(\bx) }\,.
\end{equation*}
The right-hand side is regularized thanks to the smoothness and weighted integrability properties of the boundary-localized convolutions, their adjoints, and related operators established in \Cref{sec:BdyLocalizedConv}, so long as the kernel $\rho$ is smooth.
In fact, for a more general class of Poisson data $f$ possibly varying in $\delta$, solutions $u_\delta$ actually belong to $W^{1,2}(\Omega)$, and converge weakly in $W^{1,2}(\Omega)$ to the solution $u$ of the local problem.

\begin{theorem}
    Take all assumptions of \Cref{thm:Intro:GreensIdentity}, with $p = 2$, $0 \leq \beta < d - 1$, and let $\cL_\delta$ be defined as in \eqref{eq:LinearOperatorDefinition}.
    Assume additionally that $k_\psi \geq 2$ in \eqref{eq:Fxnal}, and
    the kernel $\rho$ in $\cL_\delta u$ belongs to $C^1(\bbR)$ and satisfies $\bar{\rho}_{2-\beta} = \overline{C}_{d,2}$. 
    Let $f \in [W^{1,2}(\Omega)]^*$, and let $\{f_\delta\}_\delta$ be a sequence of regularizations of $f$ (made precise in \Cref{sec:Convergence} below) that converge to $f$ weakly in $[W^{1,2}(\Omega)]^*$ and satisfy certain weighted Sobolev estimates in \eqref{eq:H1Convergence:RHSEstimate}. 
    Let $\{u_\delta\}$ be the sequence of solutions in the respective admissible set to either \eqref{eq:BVP:Dirichlet:Weak}, \eqref{eq:BVP:Neumann:Weak}, or \eqref{eq:BVP:Robin:Weak} with Poisson data $f_\delta$ and boundary data $g$.
    Then there exists a constant $C$ independent of $\delta$ and the solutions $\{ u_\delta \}_\delta$ such that
    \begin{equation*}
        \Vnorm{u_\delta}_{W^{1,2}(\Omega)} \leq C\,.
    \end{equation*}
    Consequently, there exists a subsequence (not relabeled) $\{u_\delta \}_\delta$ that converges weakly in $W^{1,2}(\Omega)$ to a function $u$, and $u$ is the unique weak solution in the respective admissible set to either \eqref{eq:BVP:Dirichlet:Local}, \eqref{eq:BVP:Neumann:Local}, or \eqref{eq:BVP:Robin:Local}, with Poisson data $f$ and boundary data $g$. 
\end{theorem}

The precise statements and their proofs are in \Cref{sec:Regularity,sec:Convergence}.
Given any $f \in [W^{1,2}(\Omega)]^*$, the existence of such a regularizing sequence $\{f_\delta\}_\delta$ is guaranteed, and is given via the adjoint localized convolution $f_\delta = K_\delta^* f$; see \Cref{sec:BdyLocalizedConv} and \Cref{thm:DataMollification}.

With this convergence result for the linear problem in hand, we turn to the same question for another quantity of interest -- the normal derivative, or boundary flux. 
We can use the nonlocal Green's identity \eqref{eq:Intro:GreensIdentity:LinearForm} to formally define a distributional normal derivative on the  domain boundary, which can be made rigorous when the data are regular enough. To illustrate, we present the following convergence theorem for solutions to the Dirichlet problem:
\begin{theorem}
    In the setting of the previous theorem, assume additionally that $f$, $f_\delta$ all belong to $L^2(\Omega)$, and that the Dirichlet data $g \in W^{3/2,2}(\p \Omega)$. Then the sequence of distributions $A_\delta^N \frac{\p u_\delta}{\p \bsnu}$, defined (formally) via 
\begin{equation*}
    \Vint{ A_\delta^N \frac{\p u_\delta}{\p \bsnu} , v } = \cB_\delta(u_{\delta},\bar{v}) - \int_{\Omega} \bar{v} \cL_\delta u_\delta \, \rmd \bx\,, \quad v \in W^{1/2,2}(\p \Omega)\,,
\end{equation*}
where $\bar{v}$ is any $W^{1,2}(\Omega)$-extension of $v$ to $\Omega$, converges to $\frac{\p u}{\p \bsnu}$ weakly in $[W^{1/2,2}(\p \Omega)]^*$.
\end{theorem}

The precise convergence statements for both Dirichlet and Robin problems are in \Cref{sec:ConvOfNormals}.
Note that our analysis in this work applies to an array of models associated with nonlinear nonlocal equations in the interior of $\Omega$, see more discussions in \cite{scott2023nonlocal}.

Let us remark on a few generalizations. First, all of the results in this work can be obtained for the operator 
$\cL_{p,\delta} + c I$ in place of $\cL_{p,\delta}$, where $c > 0$ is a fixed constant and $I$ is the identity map; the proofs can be modified in a straightforward way.
Second,
in the case $1 < p < 2$ many of the results in this work can be obtained -- in some form -- for the case of the nonlinearity $\Phi_p(t) = \frac{1}{p} |t|^p$. However, due to the degeneracy of $|t|^{p-2}$, often the additional assumption $\grad u \neq {\bf 0}$ is required. 
We will not go into further detail here, but only say that the necessary modifications for results such as \Cref{thm:Intro:GreensIdentity} are similar to analogous modifications in \cite{ishii2010class}.

This paper is organized as follows. The next two subsections contain results obtained directly from \cite{scott2023nonlocal} that we will need; this consists of the relevant properties of the nonlocal function space, and the well-posedness and nonlocal-to-local convergence of the boundary-value problems. \Cref{sec:buildingblockest} contains some estimates of quantities involving the heterogeneous localization that we reference throughout the paper.
The nonlocal Green's identity is proved in \Cref{sec:GreensIdentity}, and its generalizations to other classes of functions and other notions of the nonlocal operator are investigated in \Cref{sec:NonlocalOperator}.
Properties of the boundary-localized convolution and its adjoint are collected in \Cref{sec:BdyLocalizedConv}.
The regularity result for the linear problem is studied in \Cref{sec:Regularity}. \Cref{sec:Convergence} contains the proof of Sobolev convergence of the weak solutions to the solution of the corresponding local problem, and the analogous proofs of convergence for their boundary normal derivatives are in \Cref{sec:ConvOfNormals}.

\section{Background on the nonlocal problems}

\subsection{Properties of the nonlocal function space}
In this section, we present a few important properties on the nonlocal function space $\mathfrak{W}^{\beta,p}[\delta;q](\Omega)$, such as the density of smooth functions and trace theorems.
The nonlocal function space was studied in \cite{scott2023nonlocal}, in which properties used for studying associated variational problems were shown. 
To continue the study, we summarize the known properties needed here.

\begin{theorem}[Properties of the nonlocal function space, \cite{scott2023nonlocal}]\label{thm:FxnSpaceProp}
Let $p\in [1,\infty)$, and assume 
\eqref{assump:beta}, 
\eqref{assump:NonlinearLocalization}, and \eqref{eq:HorizonThreshold}. Then the following hold:
\begin{enumerate}
    \item[1)] For any $u \in W^{1,p}(\Omega)$,
    \begin{equation}\label{eq:Embedding}
        [u]_{\mathfrak{W}^{\beta,p}[\delta;q](\Omega)} \leq \frac{1}{(1-\delta)^{1/p}} [u]_{W^{1,p}(\Omega)}\,.
    \end{equation}
    \item[2)] $C^{k}(\overline{\Omega})$ is dense in $\mathfrak{W}^{\beta,p}[\delta;q](\Omega)$ for any $k \leq k_q$.

    \item[3)] Let $p \in (1,\infty)$, and let $T$ denote the trace operator, i.e. for $u \in C^{1}(\overline{\Omega})$,
	  $ T u = u \big|_{\p \Omega}$.
	Then the trace operator extends to a bounded linear operator $T : \mathfrak{W}^{\beta,p}[\delta;q](\Omega) \to W^{1-1/p,p}(\p \Omega)$. That is,
    there exists $C = C(d,p,\beta,q,\Omega)$ such that
	$$
	\Vnorm{Tu}_{W^{1-1/p,p}(\p \Omega)} \leq C \Vnorm{u}_{\mathfrak{W}^{\beta,p}[\delta;q](\Omega)} ,\qquad\forall u \in \mathfrak{W}^{\beta,p}[\delta;q](\Omega)\,.
	$$

    \item[4)] For $1 < p < \infty$, a function $u$ belongs to $\mathfrak{W}^{\beta,p}_{0, \p \Omega_D}[\delta;q](\Omega)$ if and only if $u \in \mathfrak{W}^{\beta,p}[\delta;q](\Omega)$ and $T u = 0$ on $\p \Omega_D$.

    \item[5)] For constants $0 < \delta_1 \leq \delta_2 < \underline{\delta}_0$,
	\begin{equation*}
 \begin{aligned}
	&	\left( \frac{1-\delta_2}{2(1+\delta_2)} \right)^{\frac{d+p-\beta}{p}}[u]_{\mathfrak{W}^{\beta,p}[\delta_2;q](\Omega)} 
    \leq [u]_{\mathfrak{W}^{\beta,p}[\delta_1;q](\Omega)}\\
    &\qquad \leq \left( \frac{\delta_2}{\delta_1} \right)^{1+(d-\beta)/p} [u]_{\mathfrak{W}^{\beta,p}[\delta_2;q](\Omega)}, \quad
    \forall 
u \in \mathfrak{W}^{\beta,p}[\delta_2;q](\Omega)\,.
     \end{aligned}
	\end{equation*}

    \item[6)] For $\rho$ satisfying \eqref{assump:VarProb:Kernel} and $\lambda$ satisfying \eqref{assump:Localization}, define the seminorm
    \begin{equation*}
        [u]_{\mathfrak{V}^{\beta,p}[\delta;q;\rho,\lambda](\Omega)}^p := \int_{\Omega} \int_{\Omega} \gamma_{\beta,p}[\delta;q;\rho,\lambda](\bx,\by) |u(\by)-u(\bx)|^p \, \rmd \by \, \rmd \bx\,,
    \end{equation*}
    where $\gamma_{\beta,p}[\delta;q;\rho,\lambda](\bx,\by)$ is a kernel defined as
    \begin{equation*}
        \gamma_{\beta,p}[\delta;q;\rho,\lambda](\bx,\by) := \frac{ \overline{C}_{d,p} }{ \bar{\rho}_{p-\beta} } \frac{ \rho( |\bx-\by| / \eta_\delta(\bx) )}{ |\bx-\by|^{\beta} \eta_\delta(\bx)^{d+p-\beta} }\,.
    \end{equation*}
   Then there exist positive constants $c$ and $C$ depending only on $d$, $\beta$, $p$, $\rho$, $q$, and $\kappa_0$ such that for any $u \in \mathfrak{W}^{\beta,p}[\delta;q](\Omega)$,
    \begin{equation}\label{thm:EnergySpaceIndepOfKernel}
        c [u]_{\mathfrak{W}^{\beta,p}[\delta;q](\Omega)} \leq
        [u]_{\mathfrak{V}^{\beta,p}[\delta;q;\rho,\lambda](\Omega)} \leq C  [u]_{\mathfrak{W}^{\beta,p}[\delta;q](\Omega)}\,.
    \end{equation}

\end{enumerate}
\end{theorem}

The trace theorems ensure that proper local boundary conditions can be imposed for the associated nonlocal problems.

\subsection{Weak formulation of boundary-value problems: Well-posedness and local limits}\label{subsec:BVPReview}
Let us begin the study of boundary-value problems
associated with \eqref{eq:Fxnal}
with some discussions on the semilinear term. 
The function $\ell \in C^1(\bbR)$ is assumed to be a nonnegative convex function, and acts as 
a lower-order term.
For $p \in [1,\infty)$ and the exponent $p^*$ defined as
\begin{equation}\label{eq:SobolevExponent}
    p^* :=
    \begin{cases}
        \frac{dp}{d-p}\,, &\text{ if } p < d\,, \\
        \text{any exponent } < \infty\,, &\text{ if } p \geq d\,,
    \end{cases}
\end{equation}
we assume that for some $m \in (1,p^*)$, $\ell$ satisfies,
for some constants $c >0 $ and $C >0$,
\begin{equation}\label{eq:LowerOrderTerm}
\begin{gathered}
    c|t|^m \leq \ell(t) \leq C |t|^m \,, \, \forall t \in \bbR\,, \\
    c|t|^{m} \leq \ell'(t) t\,, \text{ and } |\ell'(t)| \leq C|t|^{m-1}\,, \\ 
    |\ell'(t) - \ell'(\tau)| \leq
    \begin{cases}
        C (|t|^{m-2} + |\tau|^{m-2}) |t-\tau|, & \text{ if } m \geq 2\,, \\
            C |t-\tau|^{m-1}\,, & \text{ if } 1 < m \leq 2\,,
    \end{cases}
    \quad  \forall t, \tau \in \bbR\,.
\end{gathered}
\end{equation}
The functional $\int_{\Omega} \ell(K_\delta u)$ is well-defined for $u \in \mathfrak{W}^{\beta,p}[\delta;q](\Omega)$ thanks to the estimate \eqref{eq:Intro:ConvEst:Deriv} below, along with the classical Sobolev embedding.

We note that the results of this subsection are under the assumptions 
\eqref{assump:beta}, \eqref{assump:Localization}, \eqref{assump:NonlinearLocalization}, \eqref{assump:VarProb:Kernel}, \eqref{eq:assump:Phi}, 
and \eqref{eq:HorizonThreshold}.
Note that the assumptions \eqref{assump:Localization:NormalDeriv} and \eqref{assump:MomentsOfNonlinLoc} were taken in order to apply Green's identity and obtain the weak formulations; the assumptions are not required for their well-posedeness.

For the Dirichlet problem \eqref{eq:BVP:Dirichlet:Weak}, we have the following:
\begin{theorem}\label{thm:WellPosedness:Dirichlet}
    Given $f \in [\mathfrak{W}^{\beta,p}[\delta;q](\Omega)]^*$ and $g\in
    W^{1-1/p,p}(\p \Omega_D)$, there exists a unique solution $u_\delta\in \mathfrak{W}^{\beta,p}_{g,\p \Omega_D}[\delta;q](\Omega)$ of \eqref{eq:BVP:Dirichlet:Weak} that satisfies 
    \begin{equation}\label{eq:EnergyEstimate:Dirichlet}
        \Vnorm{u_\delta}_{\mathfrak{W}^{\beta,p}[\delta;q](\Omega)}^{p-1}
        \leq C \big( \vnorm{f}_{[\mathfrak{W}^{\beta,p}[\delta;q](\Omega)]^*} + \Vnorm{g}_{W^{1-1/p,p}(\p \Omega_D)} \big)\,.
    \end{equation}
    Moreover, under the further assumption that $\rho$ is nonincreasing on $[0,\infty)$ and $f\in [W^{1,p}(\Omega)]^*$, for a sequence $\delta \to 0$ let $u_\delta \in \mathfrak{W}^{\beta,p}_{g, \p \Omega_D}[\delta;q](\Omega)$ be a sequence in the admissible set satisfying \eqref{eq:BVP:Dirichlet:Weak} with Poisson data $K_\delta^* f$ and boundary data $g$. Then $\{u_\delta\}_\delta$ is precompact in the strong topology on $L^p(\Omega)$. Furthermore, any limit point $u$ satisfies $u \in W^{1,p}_{g,\p \Omega_D}(\Omega)$, where 
    $W^{1,p}_{g, \p \Omega_D}(\Omega) := \{ v \in W^{1,p}(\Omega) \, : \, Tv = g \text{ on } \p \Omega_D \},
    $
    and $u$ satisfies
    \begin{equation}\label{eq:BVP:Dirichlet:Local}
        \cB_{p,0}(u,v) + \mu \int_{\Omega} \ell'(u) v \, \rmd \bx = \vint{f,v}\,, \forall v \, \in W^{1,p}_{0,\p \Omega_D}(\Omega)\,, \text{ with } Tu = g \text{ on } \p \Omega_D\,.
    \tag{\ensuremath{\rmD_0}}
    \end{equation}
    In the case that $g \equiv 0$, the same well-posedness result, energy estimate, and convergence result holds in the case $f \in [W^{1,p}_{0,\p \Omega_D}(\Omega)]^*$, with solutions $u_\delta \in \mathfrak{W}^{\beta,p}_{0,\p \Omega_D}[\delta;q](\Omega)$, and any limit point $u$ belongs to the Banach space $W^{1,p}_{0,\p \Omega_D}(\Omega)$, which denotes the closure of $C^1_c(\overline{\Omega}\setminus \p \Omega_D)$ with respect to $\Vnorm{\cdot}_{W^{1,p}(\Omega)}$.
\end{theorem}
\begin{proof}
    The well-posedness results and the limit behavior have been shown in \cite{scott2023nonlocal} using the energy minimization formulation. The estimate \eqref{eq:EnergyEstimate:Dirichlet} follows easily from the weak formulation,  assumptions on the semilinear term, and the coercivity of the form established in \cite{scott2023nonlocal}.
\end{proof}

Likewise, we have similar results on the Neumann and Robin problems. The results are stated with proofs skipped.

\begin{theorem}\label{thm:WellPosedness:Neumann}
  For $
    f \in [\mathfrak{W}^{\beta,p}[\delta;q](\Omega)]^*$ and $g\in
    [W^{1-1/p,p}(\p \Omega)]^*$ with
 $\vint{f,1} + \vint{g,1} = 0$,
    there exists a unique solution $u_\delta \in \mathring{\mathfrak{W}}^{\beta,p}[\delta;q](\Omega)$ of \eqref{eq:BVP:Neumann:Weak} that satisfies
    \begin{equation}\label{eq:EnergyEstimate:Neumann}
        \Vnorm{u_\delta}_{\mathfrak{W}^{\beta,p}[\delta;q](\Omega)}^{p-1} \leq C \big( \vnorm{f}_{[\mathfrak{W}^{\beta,p}[\delta;q](\Omega)]^*} + \Vnorm{g}_{[W^{1-1/p,p}(\partial \Omega)]^*} \big)\,.
    \end{equation}
    Moreover, under the further assumption that $\rho$ is nonincreasing on $[0,\infty)$ and $f \in [W^{1,p}(\Omega)]^*$, 
    for a sequence $\delta\to 0$ 
    let $u_\delta \in \mathring{\mathfrak{W}}^{\beta,p}[\delta;q](\Omega)$ be a sequence satisfying \eqref{eq:BVP:Neumann:Weak} with Poisson data $K_\delta^* f$ and boundary data $g\in
    [W^{1-1/p,p}(\p \Omega)]^*$. Then $\{u_\delta\}_\delta$ is precompact in the strong topology on $L^p(\Omega)$. Furthermore, any limit point $u$ satisfies $u \in \mathring{W}^{1,p}(\Omega)$, where $\mathring{W}^{1,p}(\Omega) := \{ v \in W^{1,p}(\Omega) \, : \, (v)_\Omega = 0 \}$,
    and $u$ satisfies
    \begin{equation}\label{eq:BVP:Neumann:Local}
        \cB_{p,0}(u,v) = \vint{f,v} + \vint{g,v}\,, \quad \forall v \, \in \mathring{W}^{1,p}(\Omega)\,.
    \tag{\ensuremath{\rmN_0}}
    \end{equation}
\end{theorem}

\begin{theorem}\label{thm:WellPosedness:Robin}
Assume that at least one of the following hold
for \eqref{eq:BVP:Robin:Weak}:
\begin{enumerate}
    \item[1)] $\mu > 0$,
    \item[2)] there exists a $\sigma$-measurable set $\p \Omega_R$ of $\p \Omega$ satisfying $\sigma(\p \Omega_R) > 0$ and $b(\bx) \geq b_0 > 0$ for a constant $b_0$ and $\sigma$-almost every $\bx \in \p \Omega_R$.
\end{enumerate}
Then, given $
    f \in [\mathfrak{W}^{\beta,p}[\delta;q](\Omega)]^*$ and $g \in [W^{1-1/p,p}(\p \Omega)]^*$,
    there exists a unique solution $u_\delta\in \mathfrak{W}^{\beta,p}[\delta;q](\Omega)$
to \eqref{eq:BVP:Robin:Weak} that satisfies
\begin{equation}\label{eq:EnergyEstimate:Robin}
        \Vnorm{u_\delta}_{\mathfrak{W}^{\beta,p}[\delta;q](\Omega)}^{p-1}
        \leq C \big( \vnorm{f}_{[\mathfrak{W}^{\beta,p}[\delta;q](\Omega)]^*} + \Vnorm{g}_{[W^{1-1/p,p}(\p \Omega)]^*} \big)\,.
    \end{equation}
Moreover, under the further assumption that $\rho$ is nonincreasing on $[0,\infty)$ and $f \in [W^{1,p}(\Omega)]^*$, for a sequence $\delta\to 0$, let $u_\delta \in \mathfrak{W}^{\beta,p}[\delta;q](\Omega)$ be a sequence satisfying \eqref{eq:BVP:Robin:Weak} with Poisson data $K_\delta^* f$ and boundary data $g$. Then $\{u_\delta\}_\delta$ is precompact in the strong topology on $L^p(\Omega)$. Furthermore, any limit point $u$ satisfies $u \in W^{1,p}(\Omega)$, and $u$ satisfies
    \begin{equation}\label{eq:BVP:Robin:Local}
        \cB_{p,0}(u,v) + \mu \int_{\Omega} \ell'(u) v \, \rmd \bx + \int_{\p \Omega} b \Phi_p'(Tu) Tv \, \rmd \sigma = \vint{f,v} + \vint{g,v}, \;\; \forall v \in W^{1,p}(\Omega).
        \tag{\ensuremath{\rmR_0}}
    \end{equation}
\end{theorem}

\section{Properties of heterogeneous localization functions and the associated kernels}\label{sec:buildingblockest}

We now present some properties related to the function $\eta$ and various kernels used in this work.
All the discussions are 
under the assumptions
\eqref{assump:NonlinearLocalization}, \eqref{assump:Localization},
\eqref{Assump:Kernel}, and \eqref{eq:HorizonThreshold}.
Additional assumptions on $\psi$ are made for some of the results presented in 
\Cref{sec:AuxMolls}.

\subsection{Spatial variations of the heterogeneous localization function}\label{sec:LocalizationFunc}
For ease of access, we record the following comparisons of 
the heterogeneous localization function $\eta_\delta$ that are frequently referred to in later discussions. The first two lemmas were proved in \cite{scott2023nonlocal}, and we write them here for ease of reference.

\begin{lemma}[\cite{scott2023nonlocal}]\label{lma:ComparabilityOfXandY}
For all $\bx, \by \in \Omega$,
\begin{eqnarray}\label{eq:ComparabilityOfDistanceFxn1}
	&	(1- \kappa_1 \delta ) \eta_\delta(\bx) \leq \eta_\delta(\by) \leq (1+\kappa_1 \delta ) \eta_\delta(\bx), \; \text{ if } |\bx-\by| \leq \eta_\delta(\bx),\\
\label{eq:ComparabilityOfDistanceFxn2}
&		(1-\kappa_1 \delta ) \eta_\delta(\by) \leq \eta_\delta(\bx) \leq (1+\kappa_1 \delta ) \eta_\delta(\by), \;  \text{ if } |\bx-\by| \leq \eta_\delta(\by)\,.
	\end{eqnarray}
\end{lemma}

The next lemma is used later to facilitate a change of coordinates.

\begin{lemma}[\cite{scott2023nonlocal}]\label{lma:CoordChange2}
For a fixed $\bz \in B(0,1)$, 
    define the function $\bszeta_\bz^\veps : \Omega \to \bbR^d$ by 
    \begin{equation*}
    \bszeta_\bz^\veps(\bx) := \bx + \eta_\veps(\bx) \bz\,, \quad
    \forall \veps \in (0,\underline{\delta}_0)\,.
    \end{equation*}
    Then for all $\bx$, $\by \in \Omega$, for all $\delta \in (0,\underline{\delta}_0)$, and for all $\veps \in (0,\underline{\delta}_0)$,
    we have 
    \begin{equation}\label{eq:bszetaproperties}
        \begin{gathered}
            \det \grad \bszeta_\bz^\veps(\bx) = 1 + \grad \eta_\veps(\bx) \cdot \bz > 1-\kappa_1 \veps > \frac{2}{3}\,, \\
            \bszeta_\bz^\veps(\bx) \in \Omega \text{ and } 0 < (1- \kappa_1 \veps) \eta(\bx) \leq \eta(\bszeta_\bz^\veps(\bx)) \leq (1+\kappa_1\veps) \eta(\bx) \,, \\
            0 < (1- \kappa_1 \veps) |\bx-\by| \leq | \bszeta_\bz^\veps(\bx) - \bszeta_\bz^\veps(\by)| \leq (1+\kappa_1 \veps) |\bx-\by|\,, \text{ and } \\
            |\bx-\by| \leq \eta_\delta(\bx) \quad \Rightarrow \quad |\bszeta_{\bz}^{\veps}(\bx) - \bszeta_{\bz}^{\veps}(\by)| \leq \frac{1+ \kappa_1 \veps}{1-\kappa_1 \veps} \eta_\delta( \bszeta_{\bz}^{\veps}(\bx) )\,.
        \end{gathered}
    \end{equation} 
\end{lemma}

\begin{lemma}\label{lma:CoordChange1}
    Fix $\bx \in \Omega$, and define the function $\bsupsilon_\bx^\delta : \Omega \to \bbR^d$ by 
    \begin{equation*}
        \bsupsilon_\bx^\delta(\by) := \frac{\by-\bx}{\eta_\delta(\by)}\,.
    \end{equation*}
    Then $\bsupsilon_\bx$ is injective on the set $B_{\bx}^\delta := \{ \by \in \bbR^d \, : \, |\bsupsilon_\bx^\delta(\by)| \leq 1 \} \subset \Omega$, and satisfies   
    \begin{equation}\label{eq:CoordChange1:Det}
        0 < \frac{1-\kappa_1 \delta}{ \eta_\delta(\by)^d } \leq \det \grad \bsupsilon_\bx^\delta(\by) = \frac{\eta_\delta(\by) + \grad \eta_\delta(\by) \cdot (\bx-\by)}{\eta_\delta (\by)^{d+1}} 
        \leq \frac{1 + \kappa_1 \delta}{ \eta_\delta(\by)^d }
    \end{equation}
    for all $\by \in B_{\bx}^\delta$. 
    That is, $\bsupsilon_{\bx}^\delta(\by)$ is globally invertible on $B_{\bx}^\delta$.
\end{lemma}

\begin{proof}
    First, by \eqref{eq:ComparabilityOfDistanceFxn2} we have 
    $
    |\by-\bx| \leq \eta_\delta(\by) \leq \frac{\delta}{ 1-\kappa_1 \delta } \eta (\bx) < \frac{1}{2} \eta(\bx)
    $,
    for all $\by \in B_\bx^\delta$, so $B_\bx^\delta \subset \Omega$.
    Second, if $\bsupsilon_{\bx}^\delta(\by_1) = \bsupsilon_{\bx}^\delta(\by_2)$ for $\by_1$, $\by_2 \in B_{\bx}^\delta$, then 
    \begin{equation*}
    \begin{split}
        |\by_1 -\by_2|^2 &= |\by_1 - \bx|^2 + |\by_2 - \bx|^2 - 2 \Vint{ \by_1 - \bx, \by_2 - \bx } \\
        &= |\bsupsilon_{\bx}^\delta(\by_1)|^2 \eta_\delta(\by_1)^2 +  |\bsupsilon_{\bx}^\delta(\by_2)|^2 \eta_\delta(\by_2)^2 -  2 \vint{ \bsupsilon_\bx^\delta(\by_1) \cdot \bsupsilon_\bx^\delta(\by_2) } \eta_\delta(\by_1) \eta_\delta(\by_2) \\
        &=  |\bsupsilon_\bx^\delta(\by_1)|^2 ( \eta_\delta(\by_1)^2 + \eta_\delta(\by_2)^2 - 2 \eta_\delta(\by_1) \eta_\delta(\by_2) ) \\
        &\leq |\eta_\delta(\by_1) - \eta_\delta(\by_2)|^2 \leq \frac{1}{9}|\by_1 - \by_2|^2\,,
    \end{split}
    \end{equation*}
    implying that $\by_1 = \by_2$. Third, on $B_\bx^\delta$ we have $|\grad \eta_\delta(\by) \cdot (\frac{\bx-\by}{\eta_\delta(\by)})| \leq \kappa_1 \delta$, so therefore
    \begin{equation*}
        1 - \kappa_1 \delta \leq \frac{ \eta_\delta(\by) + \grad \eta_\delta(\by) \cdot (\bx-\by) }{ \eta_\delta(\by) } \leq 1 + \kappa_1 \delta\,,
    \end{equation*}
    and the bounds on $\det \grad \bsupsilon_\bx^\delta(\by)$ follow.
\end{proof}

\subsection{Mollifier kernels}\label{sec:AuxMolls}
For any  function $\psi: [0,\infty) \to \bbR$, 
and $\alpha\in (-\infty,d)$, we introduce a \textit{boundary-localizing} mollifier corresponding to a standard mollifier $\psi$ described in \eqref{Assump:Kernel}, for any $\bx,\by\in \Omega$,
\begin{equation}\label{eq:KernelWeightedDef}
    \psi_{\delta,\alpha}[\lambda,q](\bx,\by) = \left( \frac{\psi(\cdot)}{ |\cdot|^\alpha } \right)_{\delta}[\lambda,q](\bx,\by) := \frac{1}{\eta_\delta[\lambda,q](\bx)^{d}} \frac{ \psi \left( \frac{|\by-\bx|}{\eta_\delta[\lambda,q](\bx)} \right) }{ \left( \frac{|\by-\bx|}{\eta_\delta[\lambda,q](\bx)} \right)^{\alpha} } \,,
\end{equation}
which extends the following special case defined in \cite{scott2023nonlocal}
\begin{equation}\label{eq:OperatorKernelDef}
	\begin{split}
		\psi_{\delta}[\lambda,q](\bx,\by) &:= 
  \psi_{\delta,0}[\lambda,q](\bx,\by)=
  \frac{1}{\eta_\delta[\lambda,q](\bx)^{d}} {\psi} \left( \frac{|\by-\bx|}{\eta_\delta[\lambda,q](\bx)} \right) \,.
	\end{split}
\end{equation}
As in \cite{scott2023nonlocal}, we recall that for $\delta<\underline{\delta}_0$, 
\[
\int_{\Omega} \psi_{\delta}[\lambda,q](\bx,\by) \, \rmd \by =
\int_{\Omega} \psi_{\delta,\alpha}[\lambda,q](\bx,\by) \, \rmd \by =1, \; \forall \bx\in\Omega.
\]
Meanwhile, we let
\begin{equation}
    \Psi_{\delta,\alpha}[\lambda,q,\psi](\bx):= \int_{\Omega} \psi_{\delta,\alpha}[\lambda,q](\by,\bx) \, \rmd \by\,,
    \;\;
        \Psi_{\delta}[\lambda,q,\psi](\bx):= \Psi_{\delta,0}[\lambda,q,\psi](\bx),\;
\end{equation}
together with abbreviation conventions
\begin{equation}
    \label{eq:abbrev}
    \psi_{\delta}[\lambda,q] = \psi_{\delta},\; \psi_{\delta,\alpha}[\lambda,q] = \psi_{\delta,\alpha},\;
\Psi_{\delta}[\lambda,q,\psi]=\Psi_{\delta},
\,\text{ and }\, \Psi_{\delta,\alpha}[\lambda,q,\psi]=\Psi_{\delta,\alpha}.
\end{equation}
We note that, when the context is clear, such abbreviations are adopted in this paper for other notations where $\psi$ and $\Psi$ are replaced by other functions.
We present some properties of $\Psi_{\delta,\alpha}$, similar to those derived for 
$\Psi_\delta$ in \cite{scott2023nonlocal}.

\begin{lemma}\label{lma:KernelIntegral}
    Let $\psi$ be a nonnegative even function in $L^\infty(\bbR)$ with nonempty support contained in $[-1,1]$. Then 
\begin{equation}\label{eq:KernelIntFunction:Bounds}
        \frac{1}{1+\kappa_1 \delta} \int_{B(0,1)} \frac{\psi(|\bz|)}{|\bz|^\alpha} \, \rmd \bz 
        \leq \Psi_{\delta,\alpha}(\bx) 
        \leq 
        \frac{1}{1-\kappa_1 \delta} \int_{B(0,1)} \frac{\psi(|\bz|)}{|\bz|^\alpha} \, \rmd \bz \,, \quad \forall \bx \in \Omega.
\end{equation}
\end{lemma}

\begin{proof}
    We use the bound \eqref{eq:CoordChange1:Det} and \Cref{lma:CoordChange1} 
    to get
    \begin{equation*}
        \Psi_{\delta,\alpha}(\bx) \leq \int_{ \{ |\bsupsilon_\bx^\delta(\by)| \leq 1 \} } \frac{\psi( \bsupsilon_\bx^\delta(\by) ) }{ |\bsupsilon_\bx^\delta(\by)|^\alpha } \frac{ \det \grad \bsupsilon_{\bx}^\delta(\by) }{ 1- \kappa_1 \delta } \, \rmd \by = \frac{1}{1-\kappa_1 \delta} \int_{B(0,1)} \frac{\psi(|\bz|)}{|\bz|^\alpha} \, \rmd \bz
    \end{equation*}
    and
    \begin{equation*}
        \Psi_{\delta,\alpha}(\bx) \geq \int_{ \{ |\bsupsilon_\bx^\delta(\by)| \leq 1 \} } \frac{\psi( \bsupsilon_\bx^\delta(\by) ) }{ |\bsupsilon_\bx^\delta(\by)|^\alpha } \frac{ \det \grad \bsupsilon_{\bx}^\delta(\by) }{ 1+ \kappa_1 \delta } \, \rmd \by = \frac{1}{1+\kappa_1 \delta} \int_{B(0,1)} \frac{\psi(|\bz|)}{|\bz|^\alpha} \, \rmd \bz\,.
    \end{equation*}
\end{proof}

We now turn to the derivatives of $\psi_\delta$.
It is clear that $\psi_\delta
\in C^k(\Omega \times \Omega)$ whenever \eqref{assump:Localization}, \eqref{assump:NonlinearLocalization}, and \eqref{Assump:Kernel} are satisfied for the same $k \in \bbN \cup \{\infty\}$. We record several estimates on the derivative of the kernel that we will use.
\begin{theorem}[\cite{scott2023nonlocal}]\label{thm:gradientpsi}
Let $\psi$  satisfy \eqref{Assump:Kernel} for some $k \geq 1$.
Then there exists $C = C(d,\psi)$ such that
    \begin{equation}\label{eq:KernelDerivativeEstimate}
		|\grad _{\bx} \psi_{\delta}(\bx,\by)| \leq \frac{C}{\eta_\delta(\bx) } \Big( \psi_{\delta}(\bx,\by) 
		+ (|\psi'|)_{\delta}(\bx,\by) \Big)\,,  \quad \forall \bx, \by \in \Omega.
	\end{equation}
Moreover, for any $\alpha \in \bbR$ there exists $C = C(d,\psi,\kappa_1,\alpha)$ such that\begin{equation}\label{eq:KernelIntegralDerivativeEstimate}
		\int_{\Omega} |\eta_\delta(\by)|^{\alpha} |\grad_{\bx} \psi_{\delta}(\bx,\by)| \, \rmd \by
		\leq \frac{C}{\eta_\delta(\bx)^{1-\alpha}}\,, \quad \forall \, \bx \in \Omega.
	\end{equation}
\end{theorem}

For the study of the adjoint operator $K_\delta^*$, we also record corresponding estimates on its kernel.
\begin{theorem}
Let $\psi$  satisfy \eqref{Assump:Kernel} for some $k \geq 1$. Then 
\begin{equation}\label{eq:KernelDerivativeEstimate:Adj}
		|\grad _{\bx} \psi_{\delta}(\by,\bx)| \leq \frac{1}{\eta_\delta(\by) }  (|\psi'|)_{\delta}(\by,\bx)\,, \quad \forall \, \bx, \by \in \Omega\,.
	\end{equation}
\end{theorem}

\begin{proof} 
 By direct computation,
$\grad_{\bx} {\psi}_{\delta}(\by,\bx) = \psi'_\delta(\by,\bx) \frac{\bx-\by}{|\bx-\by|} \frac{1}{\eta_\delta(\by)}$.
Thus, using the support of $\psi$, we see the result.
\end{proof}
\begin{corollary}
    Let $\psi$ satisfy \eqref{Assump:Kernel} for some $k \geq 1$, and let $\alpha \in \bbR$. Then there exists $C = C(d,\psi,\kappa_1,\alpha)$ such that
\begin{equation}\label{eq:KernelIntegralDerivativeEstimate:Adj}
		\int_{\Omega} |\eta_\delta(\by)|^{\alpha} |\grad_{\bx} \psi_{\delta}(\by,\bx)| \, \rmd \by
		\leq \frac{C}{\eta_\delta(\bx)^{1-\alpha}}\,, \quad \forall \, \bx \in \Omega.
	\end{equation}
\end{corollary}

\begin{proof}
    We first apply  \eqref{eq:KernelDerivativeEstimate:Adj} and then use  \eqref{eq:ComparabilityOfDistanceFxn2}. 
\end{proof}
Last, the following hat notation is used whenever we symmetrize any kernel function of two arguments:
\begin{equation}\label{eq:SymmetricKerneldDef}
    \hat{\gamma}(\bx,\by) = \gamma(\bx,\by) + \gamma(\by,\bx)\,,
\end{equation}
for example, $\hat{\psi}_\delta(\bx,\by) = \psi_\delta(\bx,\by) + \psi_\delta(\by,\bx)$.

\section{Nonlocal Green's Identity}\label{sec:GreensIdentity}

The proof of the nonlocal Green's identity hinges on precise computations near $\p \Omega$. These are proved in the next subsection, and so here we collect first the notation and known results we will need.

To this end, we denote $(d-1)$-dimensional Hausdorff measure as $\scH^{d-1}$; see for instance \cite{E15} for the definition. We denote the $\ell^\infty$-norm on $\bbR^{d-1}$ by $\vnorm{\bx}_\infty = \vnorm{\bx} = \sup \{ |x_i| \, : \, 1 \leq i \leq d-1 \}$. Here and in the rest of the paper, we denote $\dist(\bx,\p \Omega) = d_{\p \Omega}(\bx)$.

The following lemma is comprised of consequences of the main results of \cite{shapiro1987differentiability}, adapted for our purposes.

\begin{lemma}\label{lma:DiffOfDistatBdy}
For a  Lipschitz map $\varphi : \bbR^{d-1} \to \bbR$, let $\Omega_\phi \subset \bbR^d$ be a special Lipschitz domain    given by
    \begin{equation*}
        \Omega_\phi := \{ (\bx',x_d) \, : \, x_d > \varphi(\bx')\,, \bx' \in \bbR^{d-1} \}
    \end{equation*}
    with boundary $\p \Omega_\phi = \{ (\bx',\varphi(\bx')) \, : \, \bx' \in \bbR^{d-1} \}$.
    Then there exists a set $N_0 \subset \bbR^{d-1}$ such that $\scH^{d-1}(N_0) = 0$ and for all $\bx' \in \bbR^{d-1} \setminus N_0$ the following holds:

    \begin{enumerate}
        \item[i)] $\grad \varphi(\bx')$ exists (by Rademacher's theorem).
        \item[ii)] Define $\bx = (\bx',\varphi(\bx')) \in \p \Omega_\phi$. Then the
        outward
        unit normal vector $\bsnu(\bx)$ exists and is given by
        \begin{equation}\label{eq:NormalVectorDef}
            \bsnu(\bx) = \frac{(\grad \varphi(\bx'),-1)}{ \sqrt{1+ |\grad \varphi(\bx')|^2 } }\,.
        \end{equation}
        \item[iii)] Define the hyperplane $\Vint{\bsnu(\bx)}^\perp := \{ \by \in \bbR^d \, : \, \by \cdot \bsnu(\bx) = 0 \}$. For any $\bz \in \bbR^d$, the $\bz$-directional derivative of $d_{\p \Omega_\phi}$ exists at $\bx = (\bx',\varphi(\bx'))$ and is given by 
        \begin{equation*}
            \frac{ \p d_{\p \Omega_\phi} }{ \p \bz }(\bx) := \lim\limits_{t \to 0^+} \frac{d_{\p \Omega_\phi}(\bx+t\bz) }{t} = \dist(\bz, \Vint{\bsnu(\bx)}^\perp ) = - \bsnu(\bx) \cdot \bz \,.  
        \end{equation*}
        \item[iv)]
        Define for any $\bx \in \bbR^d$ a nearest point projection operator $\bsxi(\bx)$ as a point on $\p \Omega_\phi$ such that $|\bx - \bsxi(\bx)| = \dist(\bx,\p \Omega_\phi)$.
        (Note that the definition of $\bsxi$ is not unique).
        
        For any $\bx = (\bx',\varphi(\bx')) \in \p \Omega_\phi$, and for any $\bz \in \bbR^d$, let $\bar{\bz}$ be the unique point in $\vint{\bsnu(\bx)}^\perp$ that satisfies $|\bz-\bar{\bz}| = \dist(\bz,\vint{\bsnu(\bx)}^\perp)$. Then 
        \begin{equation*}
            \lim\limits_{t \to 0^+} \frac{|\bsxi(\bx+t \bz) - (\bx + t \bar{\bz})| }{ t } = 0\,.
        \end{equation*}
    \end{enumerate}
\end{lemma}

For this section, in addition to \eqref{assump:beta}, \eqref{assump:Localization}, \eqref{assump:NonlinearLocalization}, \eqref{assump:VarProb:Kernel}, \eqref{eq:assump:Phi}, 
and \eqref{eq:HorizonThreshold}, we assume that
\begin{equation}\label{assump:Localization:NormalDeriv}
    \begin{gathered}
    \text{there exists a set } N_0 \subset \p \Omega \text{ such that } \scH^{d-1}(N_0) = 0 \text{ and for all } \\ 
    \bx \in \p \Omega \setminus N_0 \text{ the following holds: }
    \text{For any } \bz \in \bbR^d\,, \\
    \text{ the } \bz\text{-directional derivative of } \lambda \text{ exists at } 
    \bx \text{ with } \\
    \frac{\p \lambda}{\p \bz}(\bx) = \lim\limits_{t \to 0} \frac{ \lambda(\bx + t \bz) }{ t } = - \bsnu(\bx) \cdot \bz\,.
    \end{gathered}
    \tag{\ensuremath{\rmA_{\lambda,nor}}}
\end{equation}
By \Cref{lma:DiffOfDistatBdy}, if $\lambda = d_{\p \Omega}$ then \eqref{assump:Localization:NormalDeriv} is satisfied. In fact, any Lipschitz domain $\Omega$ has a generalized distance function $\lambda$ that satisfies both \eqref{assump:Localization} and \eqref{assump:Localization:NormalDeriv}, with $k_\lambda = \infty$, see \Cref{lma:BdyLocalizedDistFxn} below.

For $\veps > 0$, we define the sets
\begin{equation*}
    \Omega_{\veps;\lambda,q} := \{ \bx \in \Omega \, : \, q(\lambda(\bx)) < \veps \} \,, \quad \text{ and } \quad \Omega^{\veps;\lambda,q} := \{ \bx \in \Omega \, : \, q(\lambda(\bx)) \geq \veps \}\,,
\end{equation*}
and recall the truncated operator defined in the introduction as
\begin{equation*}
    \cL_{p,\delta}^\veps u(\bx) = \mathds{1}_{\Omega^{\veps;\lambda,q}}(\bx)
    \int_{\Omega}
    \frac{\hat{\rho}_{\delta,\beta-p}^\veps[\lambda,q](\bx,\by) }{ |\bx-\by| }
    \Phi_p' \left( \frac{ u(\bx)-u(\by) }{ |\bx-\by| } \right)
   \, \rmd \by\,,
\end{equation*}
where the kernel  $\rho_{\delta,\beta-p}$
is defined using the notation of \eqref{eq:KernelWeightedDef} (with $\rho$ in place of $\psi$), and $\hat{\rho}_{\delta,\beta-p}^\veps$ is defined  from $\rho_{\delta,\beta-p}^\veps$
as in \eqref{eq:SymmetricKerneldDef}. 
It is clear from the form of the truncations that $\cL_{p,\delta}^\veps u \in L^1(\Omega)$ whenever $u \in L^1_{loc}(\Omega)$.

\subsection{Local boundary computations}

\begin{theorem}\label{thm:LocGreens:1}
For $\varrho > 0$, define the open cubes $Q = (-\varrho,\varrho)^{d-1} \subset \bbR^{d-1}$ and $2Q =(-2\varrho,2\varrho)^{d-1}$.
Let $\varphi : \bbR^{d-1} \to [0,\infty)$ be a Lipschitz function with Lipschitz constant $M$. For $\veps \ll \min\{\delta, \varrho \}$, define respectively the sets
    \begin{equation*}
    \begin{aligned}
   &     U_\veps := \{ (\bx',x_d) \in \bbR^d \, : \, \bx' \in Q \,, \varphi(\bx') < x_d < \varphi(\bx') + \veps \},\\
  & U^\veps := \{ (\bx',x_d) \in \bbR^d \, : \, \bx' \in 2Q \,, x_d > \varphi(\bx') + \veps \}\,.
  \end{aligned}
    \end{equation*}
Assume that $q$ satisfies \eqref{assump:MomentsOfNonlinLoc} for some $N \in \bbN$. 
Then for any $u,v \in C^2_b(\bbR^d)$
    \begin{equation}\label{eq:LocGreens:1}
        \begin{split}
        \lim\limits_{\veps \to 0} 
        &\int_{U_\veps} \int_{U^\veps} \rho_{\delta,\beta-p}^\veps(\bx,\by) \Phi_p' \left( \grad u(\bx) \cdot \frac{\bx-\by}{|\bx-\by|} \right)  \frac{ v(\bx) }{|\bx-\by|} \, \rmd \by  \, \rmd \bx \\
        &= \int_{Q} BF_{p,\delta}^{N,+}( \grad u(\bx',\varphi(\bx')),\bsnu(\bx')) v(\bx,\varphi(\bx')) \sqrt{1+ |\grad \varphi(\bx')|^2} \, \rmd \bx'\,,
        \end{split}
    \end{equation}
    where $\bsnu(\bx')$ is defined via \eqref{eq:NormalVectorDef},
    $BF_{p,\delta}^{N,+} : \bbR^d \times \bbS^{d-1} \to \bbR$ is defined as $BF_{p,\delta}^{N,+}(\ba,\bstheta) = \frac{1}{2} BF_{p,\delta}^N(\ba,\bstheta)$ for $N > 1$, while
    \begin{equation}
    \begin{split}
        BF_{p,\delta}^{1,+}(\ba,\bstheta) &:= 
            \int_{B(0,1)\cap  \{ \bstheta \cdot \bz > 0 \} }
            \frac{ \ln(1+\delta q'(0) (\bstheta \cdot \bz)) }{ \delta q'(0) |\bz| } \frac{ \rho(|\bz|) }{ |\bz|^{\beta-p} } \Phi_p' \left( \ba \cdot \frac{\bz}{|\bz|} \right) \, \rmd \bz \,.
    \end{split}
    \end{equation}
\end{theorem}

\begin{proof}
    First we prove \eqref{eq:LocGreens:1} under the assumption that $N=1$, i.e. $q'(0) > 0$. Noting that $\supp \rho \subset B(0,1)$, we have by the change of variables $\bz = \frac{\by-\bx}{\eta_\delta(\bx)}$ in the $\by$-integral, and then $\bar{x}_d = \frac{x_d - \varphi(\bx')}{\veps}$ in the $x_d$-integral, 
\begin{equation}\label{eq:LocGreens:PfNeg1}
        \begin{split}
         I_\veps &:= \int_{U_\veps} \int_{U^\veps} \rho_{\delta,\beta-p}^\veps(\bx,\by) \Phi_p' \left( \grad u(\bx) \cdot \frac{\bx-\by}{|\bx-\by|} \right)  \frac{ v(\bx) }{|\bx-\by|} \, \rmd \by  \, \rmd \bx  \\
            &= \int_Q \int_0^1 \int_{U^\veps_*(\bx)} \frac{-\veps \mathds{1}_{ \{ \veps < |\bz| < 1 \} } }{\eta_\delta(\bx',\varphi(\bx') + \veps x_d) |\bz| } \frac{ \rho(|\bz|) }{ |\bz|^{\beta-p} } \\
            &\qquad \Phi_p'
            \left( \grad u(\bx',\varphi(\bx')+\veps x_d) \cdot \frac{ \bz }{|\bz| } \right)
            v(\bx',\varphi(\bx')+\veps x_d) \, \rmd \bz \, \rmd x_d \, \rmd \bx'\,,
        \end{split}
    \end{equation}
    where we still write $\bar{x}_d$ as $x_d$, and where 
    \begin{equation*}
        U^\veps_*(\bx) := \left\{
        \bz =(\bz',z_d)  \in \bbR^d \, : \,
        \begin{gathered}
        \vnorm{ \bx'+\eta_\delta(\bx',\varphi(\bx')+\veps x_d)\bz' }_\infty < 2 \varrho \,,\qquad \\
       \;\; \veps x_d + \varphi(\bx') + \eta_\delta(\bx',\varphi(\bx')+ \veps x_d)z_d >\quad\\
        \quad\qquad\varphi(\bx'+\eta_\delta(\bx',\varphi(\bx')+\veps x_d)\bz') + \veps.
        \end{gathered}
        \right\}\,.
    \end{equation*}
    On one hand, clearly
        $\lambda(\bx',\varphi(\bx') +\veps x_d) \leq \kappa_0 \veps x_d$,
    so we have
    \begin{equation}\label{eq:LocGreens:Pf0}
        \eta_\delta(\bx',\varphi(\bx') +\veps x_d) \leq \kappa_0 \delta \veps x_d\,.
    \end{equation}
    On the other hand, letting $\bsxi$ denote the nearest point projection map, we have $\bsxi = (\by',\varphi(\by'))$ for some $\by' \in \p \Omega$, so
    \begin{equation}\label{eq:ComparisonOfGraphDistance}
    \begin{aligned}
            \veps x_d &= |(\bx',\varphi(\bx') +\veps x_d) - (\bx',\varphi(\bx'))|
            \leq |\bsxi(\bx',\varphi(\bx') +\veps x_d) - (\bx',\varphi(\bx'))|\\
            &\qquad         
            + |(\bx',\varphi(\bx') +\veps x_d) - \bsxi(\bx',\varphi(\bx') +\veps x_d)|\\
            &\qquad  
  \leq 
  \sqrt{1+M^2} |\by'-\bx'| + d_{\p \Omega}(\bx',\varphi(\bx')+\veps x_d)  \\
            &\leq (1+ \sqrt{1 + M^2}) d_{\p \Omega}(\bx',\varphi(\bx') +\veps x_d) \leq \frac{2(1+M)}{\kappa_0} \lambda(\bx',\varphi(\bx')+\veps x_d)\,,
        \end{aligned}
    \end{equation}
    and therefore
    \begin{equation*}
    \begin{split}
        \delta q'(0) \veps x_d & \leq \frac{2 (1+M)}{\kappa_0} ( \eta_\delta(\bx',\varphi(\bx') +\veps x_d) + C \delta \lambda(\bx',\varphi(\bx') +\veps x_d)^{2}) \\
        &\leq R_1 \eta_\delta(\bx',\varphi(\bx') +\veps x_d) + \delta R_2 (\veps x_d)^{2} \,,
    \end{split}
    \end{equation*}
    for positive constants $R_1$ and $R_2$ depending only on $q$, $M$, and $\kappa_0$, and in particular independent of $\bx$. So there exists $\veps_0 > 0$ sufficiently small (depending only on $q$, $\varphi$, $\lambda$, and $\delta_0$, but not on $\bx$ or $\bz$) such that for all $\veps < \veps_0$ and for all $x_d \in (0,1)$
    \begin{equation}\label{eq:LocGreens:Pf1}
        0 < R \veps x_d := \frac{\delta (q'(0) - R_2 \veps_0)}{R_1} \veps x_d \leq \eta_\delta(\bx',\varphi(\bx') +\veps x_d)\,.
    \end{equation}
    Now, we see using \eqref{eq:LocGreens:Pf0} and the inequality
    \begin{equation*}
        \begin{split}
      \varphi(\bx'+\eta_\delta(\bx',\varphi(\bx')+\veps x_d)\bz') - \varphi(\bx') - \eta_\delta(\bx',\varphi(\bx')+\veps x_d)z_d >\qquad\quad\\
      \qquad\qquad - (1+M)|\bz| \eta_\delta(\bx',\varphi(\bx')+\veps x_d)      
        \end{split}
    \end{equation*}
    that $U^\veps_*(\bx)$ is contained in the set
        $\{ \bz \in \bbR^d \, : \, x_d (1 + \delta \kappa_0 (1+M) |\bz| ) > 1 \}$.
    Hence, by \eqref{eq:LocGreens:Pf1} 
    we have for all $\veps < \veps_0$
    \begin{equation*}
        \begin{split}
  &      \int_0^1 \frac{\veps \mathds{1}_{U^\veps_*(\bx)} \cdot \mathds{1}_{ \{ \veps < |\bz| < 1\} } }{ \eta_\delta(\bx',\varphi(\bx') + \veps x_d) |\bz| } \, \rmd x_d
            \leq \int_{ \frac{1}{1+\delta \kappa_0 (1+M) |\bz| } }^1 \frac{ \mathds{1}_{ \{ \veps < |\bz| < 1\} } }{ R x_d |\bz| } \, \rmd x_d \\
            &\qquad\qquad = \frac{ \mathds{1}_{ \{ \veps < |\bz| < 1\} } }{ R |\bz| } \ln( 1 + \delta \kappa_0 (1+M) |\bz| ) \leq \frac{1}{R \kappa_0 (1+M)\delta}\,.
        \end{split}
    \end{equation*}
    Since the right-hand side is finite, it is now clear that the dominated convergence theorem (as well as Fubini's theorem) can be applied to the right-hand side of \eqref{eq:LocGreens:PfNeg1} thanks to the differentiability of $\grad u$ and $v$: 
    \begin{equation}\label{eq:LocGreens:Pf2}
    \begin{split}
        \lim\limits_{\veps \to 0} I_\veps
    & =\int_Q  \int_{B(0,1)} \int_\bbR \lim\limits_{\veps \to 0} \left[ \frac{-\veps \mathds{1}_{ \{ |\bz| > \veps \} } \mathds{1}_{\{ 0 < x_d < 1 \}} \cdot  \mathds{1}_{U^\veps_*(\bx)} }{\eta_\delta(\bx',\varphi(\bx') + \veps x_d) |\bz| } \right] \frac{ \rho(|\bz|) }{ |\bz|^{\beta-p} } \\
            &\qquad \qquad \Phi_p'
            \left( \grad u(\bx',\varphi(\bx')) \cdot \frac{ \bz }{|\bz| } \right)
            v(\bx',\varphi(\bx')) \, \rmd x_d \, \rmd \bz \, \rmd \bx' \\
    \end{split}
    \end{equation}
    To compute this limit we use \Cref{lma:DiffOfDistatBdy}.
    Writing $(\bx',\varphi(\bx')+\veps x_d) = (\bx',\varphi(\bx')) + \veps x_d \be_d$, we have from \eqref{assump:Localization:NormalDeriv},
    item ii), item iii) 
    and item iv) that
    \begin{equation*}
    \begin{split}
        \lambda(\bx',\varphi(\bx')+\veps x_d) 
        &= \lambda((\bx',\varphi(\bx')) + \veps x_d \be_d) 
        = \frac{\veps x_d}{\sqrt{1+|\grad \varphi(\bx')|^2}} + o(\veps) \,,
    \end{split}
    \end{equation*}
    for $\scH^{d-1}$-almost every $\bx' \in Q$. Thus
    \begin{equation}\label{eq:LocGreens:Pf3}
        \eta(\bx',\varphi(\bx')+\veps x_d) =  \frac{\veps q'(0) x_d}{\sqrt{1+|\grad \varphi(\bx')|^2}} + o(\veps)\,.
    \end{equation}
    Thanks to this identity and Rademacher's theorem, the second relation describing $U^\veps_*(\bx)$ can be written as
    \begin{equation*}
        \veps x_d - \veps > \grad \varphi(\bx') \cdot \frac{\veps \delta q'(0) x_d \bz'}{\sqrt{1+|\grad \varphi(\bx')|^2}} - \frac{\veps \delta q'(0) x_d}{\sqrt{1+|\grad \varphi(\bx')|^2}} z_d + o(\veps)\,,
    \end{equation*}
    which is 
    \begin{equation*}
        x_d > \frac{1}{1 - \delta q'(0) \bsnu(\bx') \cdot \bz } + o(1)\,,
    \end{equation*}
    for almost every $\bx' \in Q$. Note that the first term on the right-hand side of the inequality is well-defined since $ \delta q'(0) \bsnu(\bx') \cdot \bz \in [-\frac{1}{3},\frac{1}{3}]$.
    Therefore we finally have
    \begin{equation*}
        \begin{split}
            \lim\limits_{\veps \to 0} I_\veps  &=\int_Q  \int_{B(0,1)} \int_\bbR \lim\limits_{\veps \to 0} \left[ \frac{-\veps \mathds{1}_{ \{ |\bz| > \veps \} } \mathds{1}_{\{ 0 < x_d < 1 \}} \cdot  \mathds{1}_{U^\veps_*(\bx)} }{\eta_\delta(\bx',\varphi(\bx') + \veps x_d) |\bz| } \right] \frac{ \rho(|\bz|) }{ |\bz|^{\beta-p} } \\
            &\quad \qquad \Phi_p'
            \left( \grad u(\bx',\varphi(\bx')) \cdot \frac{ \bz }{|\bz| } \right)
            v(\bx',\varphi(\bx')) \, \rmd x_d \, \rmd \bz \, \rmd \bx' \\
        &= -\int_Q \int_{B(0,1)}
        \left(\int_{ \min \{1, \frac{1}{1-\delta q'(0) \bsnu(\bx') \cdot \bz } \} }^1   \frac{1}{x_d} \rmd x_d \,\right)
        \frac{\sqrt{1+|\grad \varphi(\bx')|^2}}{ \delta q'(0) |\bz| } \frac{\rho(|\bz|)}{|\bz|^{\beta-p}}  \\
            &\quad \qquad \Phi_p'
            \left( \grad u(\bx',\varphi(\bx')) \cdot \frac{ \bz }{|\bz| } \right)
            v(\bx',\varphi(\bx')) \,  \rmd \bz \, \rmd \bx' \,.
        \end{split}
    \end{equation*}
   Since  
 \begin{equation}
\label{eq:LogIntegral:1}
   \int_{\frac{1}{1-a}}^1 \frac{1}{x} \, \rmd x = \ln(1-a) \quad\text{ for }
        a \in (-1/3,1/3),
    \end{equation}
by the change of variables $\bz \to -\bz$, 
    we obtain \eqref{eq:LocGreens:1} in the case $q'(0) > 0$.
    
    The proof of \eqref{eq:LocGreens:1} in the case $N > 1$ 
    follows a similar line of reasoning. Analogously to \eqref{eq:LocGreens:Pf0} and \eqref{eq:LocGreens:Pf1}
    \begin{equation}\label{eq:LocGreens:nl:Pf0}
        \eta_\delta(\bx',\varphi(\bx') +\veps x_d) \leq \frac{\delta }{N!} (\kappa_0 \veps x_d)^N
    \end{equation}
    and
    \begin{equation*}
    \begin{split}
        \delta \frac{q^{(N)}(0)}{N!} (\veps x_d)^N & \leq R_1  \eta_\delta(\bx',\varphi(\bx') +\veps x_d) + \delta R_2 \lambda(\bx',\varphi(\bx') +\veps x_d)^{N+1} \\
        &\leq R_1 \eta_\delta(\bx',\varphi(\bx') +\veps x_d) + \delta R_2 (\veps x_d)^{N+1}\,,
    \end{split}
    \end{equation*}
    for positive constants $R_1$ and $R_2$ depending only on $q$, $\kappa_0$, $\varphi$, and in particular independent of $\bx$. 
    So there exists $\veps_0 > 0$ sufficiently small (depending only on $q$, $\kappa_0$, $\varphi$ and $\delta_0$, but not on $\bx$ or $\bz$) such that for all $\veps < \veps_0$ and for all $x_d \in (0,1)$
    \begin{equation}\label{eq:LocGreens:nl:Pf1}
        0 < R (\veps x_d)^N := \frac{\delta (\frac{q^{(N)}(0)}{N!} - R_2 \veps_0)}{R_1} (\veps x_d)^N \leq \eta_\delta(\bx',\varphi(\bx') +\veps x_d)\,.
    \end{equation}
    In order to apply the dominated convergence theorem on the $\bx'$- and $\bz$-integrals in \eqref{eq:LocGreens:PfNeg1} in this case, it suffices to check that 
    for all $\veps < \veps_0$
    \begin{equation}\label{eq:LocGreens:nl:Pf2}
        \begin{split}
        \int_0^1 & \frac{\veps \mathds{1}_{U^\veps_*(\bx)} \cdot \mathds{1}_{ \{ \veps < |\bz| < 1\} } }{\eta_\delta(\bx',\varphi(\bx') + \veps x_d) |\bz|} \, \rmd x_d
            \leq C\,,
        \end{split}
    \end{equation}
    where $C$ is independent of $\veps$, $\bx'$ and $\bz$.
    Indeed, by \eqref{eq:LocGreens:nl:Pf0} and \eqref{eq:LocGreens:nl:Pf1}
    \begin{equation*}
        \begin{split}
        \int_0^1 & \frac{\veps \mathds{1}_{U^\veps_*(\bx)} \cdot \mathds{1}_{ \{\veps < |\bz| < 1\} } }{\eta_\delta(\bx',\varphi(\bx') + \veps x_d) |\bz|} \, \rmd x_d \\
        &\leq \int_{ \{ \veps x_d - \veps > - \frac{ \delta \kappa_0^N}{N!}(1+M) |\bz| (\veps x_d)^N \} \cap \{x_d < 1 \} } \frac{ \mathds{1}_{ \{ \veps <  |\bz| < 1\} }}{ R \veps^{N-1} x_d^N |\bz| } \, \rmd x_d \\
        &:= \frac{\mathds{1}_{ \{ \veps <  |\bz| < 1\} } }{R \veps^{N-1} |\bz| } \int_{ \{ A |\bz| \veps^{N-1} x_d^N +  x_d > 1 \} \cap \{ x_d < 1 \}  } \frac{ 1 }{ x_d^N } \, \rmd x_d\\ 
        &    =  \frac{\mathds{1}_{ \{ \veps < |\bz| < 1\} } }{R \veps^{N-1} |\bz| } \int_{ \{ 0 < y < A |\bz|  \veps^{N-1} (1-y)^N \}  } \frac{ 1 }{ (1-y)^N } \, \rmd y\,
  \end{split}
    \end{equation*}
where a change of variable  $y = 1-x_d$ is used in the last step.
    The polynomial $P(y) = A|\bz| \veps^{N-1}(1-y)^N - y$ is monotone, with $P(0) > 0$ and $P( A |\bz| \veps^{N-1}) < 0$ for all $\veps > 0$ sufficiently small. Thus the set $\{ y : 0 < y <  A |\bz| \veps^{N-1}(1-y)^N \}$ is contained in the interval $(0,A |\bz|  \veps^{N-1})$ and so the integral is bounded from above by
    \begin{equation*}
        \frac{ A \mathds{1}_{ \{ |\bz| < 1\} } }{ R (1- A \veps_0)^N } < \infty\,,
    \end{equation*}
    choosing $\veps_0 > 0$ smaller if necessary. Thus \eqref{eq:LocGreens:nl:Pf2} is established. After Taylor expansion of $\grad u$ and $v$, we can use similar estimates to establish that the limit of the right-hand side of \eqref{eq:LocGreens:PfNeg1} is unchanged if we replace the arguments of $\grad u$ and $v$ with $(\bx',\varphi(\bx'))$.
    Hence, we have
    \begin{equation}\label{eq:LocGreens:nl:Pf3}
    \begin{split}
        \lim\limits_{\veps \to 0} 
        I_\veps
        &=\int_Q  \int_{B(0,1)} \lim\limits_{\veps \to 0} \left[ \int_0^1 \frac{-\veps \mathds{1}_{ \{ |\bz| > \veps \} } \mathds{1}_{U^\veps_*(\bx)} }{\eta_\delta(\bx',\varphi(\bx') + \veps x_d) |\bz| } \, \rmd x_d \right] \frac{ \rho(|\bz|) }{ |\bz|^{\beta-p} } \\
            &\quad \qquad \Phi_p'
            \left( \grad u(\bx',\varphi(\bx')) \cdot \frac{ \bz }{|\bz| } \right)
            v(\bx',\varphi(\bx')) \, \rmd \bz \, \rmd \bx'\,.
    \end{split}
    \end{equation}
    We again use \Cref{lma:DiffOfDistatBdy} to get
  $      \lambda(\bx',\varphi(\bx')+\veps x_d) 
        = \frac{\veps x_d}{\sqrt{1+|\grad \varphi(\bx')|^2}} + o(\veps x_d)$
    for almost every $\bx' \in Q$. Thus
    \begin{equation}\label{eq:LocGreens:nl:Pf4}
        \eta(\bx',\varphi(\bx')+\veps x_d) =  \frac{ q^{(N)}(0) }{N! (1+|\grad \varphi(\bx')|^2)^{N/2}} (\veps x_d)^N + o((\veps x_d)^N)\,.
    \end{equation}
    Thanks to this identity and Rademacher's theorem, the second relation describing $U^\veps_*(\bx)$ can be written as
    \begin{equation*}
        \veps x_d - \veps > \frac{ q^{(N)}(0) }{N!} \frac{\delta( \bsnu(\bx') \cdot \bz) }{ (1+|\grad \varphi(\bx')|^2)^{\frac{N-1}{2}} } (\veps x_d)^N + o((\veps x_d)^N)
    \end{equation*}
    for almost every $\bx' \in Q$. Defining $a := \frac{ q^{(N)}(0) }{N!} \frac{\delta( \bsnu(\bx') \cdot \bz) }{ (1+|\grad \varphi(\bx')|^2)^{\frac{N-1}{2}} } \in (-\frac{1}{3N!}, \frac{1}{3N!})$,
    we have
    \begin{equation}\label{eq:LocGreens:nl:Pf5}
        \begin{split}
            \int_0^1 \frac{-\veps \mathds{1}_{U^\veps_*(\bx)} \mathds{1}_{ \{ |\bz| > \veps \} } }{\eta_\delta(\bx',\varphi(\bx') + \veps x_d)} \, \rmd x_d
            = \mathds{1}_{ \{ |\bz| > \veps \} } \frac{ (\bsnu(\bx') \cdot \bz) }{|\bz|} \sqrt{ 1+|\grad \varphi(\bx')|^2 } \; \tilde{I}_\veps(a)\,,
        \end{split}
    \end{equation}
    where 
    \begin{equation}
        \tilde{I}_\veps(a) := \frac{1}{-a \veps^{N-1}}
                \int_{ \{ 1 + a \veps^{N-1} x_d^N +o( \veps^{N-1} x_d^{N} ) < x_d < 1 \} } \frac{1}{ x_d^N + o( \veps x_d^N) } \, \rmd x_d\,.
    \end{equation}
    To find $\lim\limits_{\veps \to 0} \tilde{I}_\veps(a)$, we note that the limit will be nonzero only if $a < 0$, and then perform the change of variables $y = \frac{x_d - 1}{ a \veps^{N-1} }$ to obtain
    \begin{equation*}
        \begin{split}
        \lim\limits_{\veps \to 0} \tilde{I}_\veps(a)
        = \mathds{1}_{ \{ a < 0 \} } \int_0^1  
        \frac{ \mathds{1}_{ \{ 0 < y <  (1 + a \veps^{N-1} y)^N + o( (1 + a \veps^{N-1}y)^N ) \} } }
        { (1+a \veps^{N-1} y)^N + o( \veps (1 + a \veps^{N-1} y)^N) } 
        \, \rmd y
        = \mathds{1}_{ \{ a < 0 \} }\,,
        \end{split}
    \end{equation*}
    where the limit follows from the dominated convergence theorem (note that the dominating constant may still depend on $\bx'$ or $\bz$, but is independent of $y$).
    Inserting this into \eqref{eq:LocGreens:nl:Pf5} and in turn into \eqref{eq:LocGreens:nl:Pf3} we have
    \begin{equation*}
        \begin{split}
           \lim\limits_{\veps \to 0} I_\veps(a) &=
            \int_Q   \int_{B(0,1)} \lim\limits_{\veps \to 0} \left[ \int_0^1 \frac{-\veps \mathds{1}_{ \{ |\bz| > \veps \} } \mathds{1}_{U^\veps_*(\bx)} }{\eta_\delta(\bx',\varphi(\bx') + \veps x_d) |\bz| } \, \rmd x_d \right] \frac{ \rho(|\bz|) }{ |\bz|^{\beta-p} } \\
            &\quad \qquad \Phi_p'
            \left( \grad u(\bx',\varphi(\bx')) \cdot \frac{ \bz }{|\bz| } \right)
            v(\bx',\varphi(\bx')) \, \rmd \bz \, \rmd \bx' \\
            &= \int_Q  \int_{B(0,1)} \mathds{1}_{ \{ \bsnu(\bx') \cdot \bz < 0 \} } \frac{ (\bsnu(\bx') \cdot \bz) }{ |\bz| } \sqrt{ 1+|\grad \varphi(\bx')|^2 }   \frac{ \rho(|\bz|) }{ |\bz|^{\beta-p} } \\
            &\quad \qquad \Phi_p'
            \left( \grad u(\bx',\varphi(\bx')) \cdot \frac{ \bz }{|\bz| } \right)
            v(\bx',\varphi(\bx')) \, \rmd \bz \, \rmd \bx'\,.
        \end{split}
    \end{equation*}
    Changing to polar coordinates in the $\bz$-integral, and noting that
    \begin{equation*}
        \begin{split}
            \fint_{\bbS^{d-1}} \mathds{1}_{ \{ \bstheta \cdot \bsomega < 0 \} }  \Phi_p'
            \left( \ba \cdot \bsomega \right) (\bsomega \cdot \bstheta) \, \rmd \sigma(\bsomega) 
            = \fint_{\bbS^{d-1}} \mathds{1}_{ \{ \bstheta \cdot \bsomega > 0 \} }  \Phi_p'
            \left( \ba \cdot \bsomega \right) (\bsomega \cdot \bstheta) \, \rmd \sigma(\bsomega)\,,
        \end{split}
    \end{equation*}
    we obtain \eqref{eq:LocGreens:1} in the case $N > 1$.
\end{proof}
Likewise, we can derive a result similar to 
\eqref{eq:LocGreens:1} for the kernel $\rho_{\delta,\beta-p}^\veps(\by,\bx)$ in place of $\rho_{\delta,\beta-p}^\veps(\bx,\by)$. For this, we first give an upper bound to the corresponding integral.

\begin{lemma}\label{lma:LocGreens:FiniteInt}
    Under the same assumptions as \Cref{thm:LocGreens:1},
    \begin{equation*}
        \int_{U_\veps} \int_{U^\veps} \rho_{\delta,\beta-p}^\veps(\by,\bx) \Phi_p' \left( \grad u(\bx) \cdot \frac{\bx-\by}{|\bx-\by|} \right)  \frac{ v(\bx) }{|\bx-\by|} \, \rmd \by  \, \rmd \bx \leq \frac{ C \Vnorm{ \grad u}_{L^\infty}^{p-1} \Vnorm{ v }_{L^{\infty}} }{ \veps^2 }\,.
    \end{equation*}
\end{lemma}

\begin{proof}
    Similarly to \eqref{eq:ComparisonOfGraphDistance}, we have $|y_d - \varphi(\by')| \leq \frac{2(1+M)}{\kappa_0} \eta_\delta(\by)$, so it follows that $\frac{\veps}{2(1+\Vnorm{\grad \varphi}_{L^\infty})} < \eta_\delta(\by)$, $\forall \by \in U^\veps$. Thus $|\bx-\by| \geq \veps \eta_\delta(\by) \geq C \veps^2$ in the integrand.
\end{proof}

\begin{theorem}\label{thm:LocGreens:2}
    Under the same assumptions as \Cref{thm:LocGreens:1},
    \begin{equation}\label{eq:LocGreens:2}
        \begin{split}
        \lim\limits_{\veps \to 0} 
        &\int_{U_\veps} \int_{U^\veps} \rho_{\delta,\beta-p}^\veps(\by,\bx) \Phi_p' \left( \grad u(\bx) \cdot \frac{\bx-\by}{|\bx-\by|} \right)  \frac{ v(\bx) }{|\bx-\by|} \, \rmd \by  \, \rmd \bx \\
        &= \int_{Q} BF_{p,\delta}^{N,-}( \grad u(\bx',\varphi(\bx')),\bsnu(\bx')) v(\bx,\varphi(\bx')) \sqrt{1+ |\grad \varphi(\bx')|^2} \, \rmd \bx'\,,
        \end{split}
    \end{equation}
    with $BF_{p,\delta}^{N,-} : \bbR^d \times \bbS^{d-1} \to \bbR$ given by $BF_{p,\delta}^{N,-}(\ba,\bstheta) = \frac{1}{2} BF_{p,\delta}^N(\ba,\bstheta)$ for $N > 1$, while
    \begin{equation*}
    \begin{split}
        BF_{p,\delta}^{1,-}(\ba,\bstheta) &:= 
            -\int_{B(0,1) \cap  \{ \bstheta \cdot \bz > 0 \} }
            \frac{ \ln(1-\delta q'(0) (\bstheta \cdot \bz)) }{ \delta q'(0) |\bz|} \frac{ \rho(|\bz|) }{ |\bz|^{\beta-p} } \Phi_p' \left( \ba \cdot \frac{\bz}{|\bz|} \right) \, \rmd \bz \,.
    \end{split}
    \end{equation*}
\end{theorem}

\begin{proof}
    First we prove \eqref{eq:LocGreens:2} under the assumption that $q'(0) > 0$. The proof is similar to that of \Cref{thm:LocGreens:1}. By \Cref{lma:LocGreens:FiniteInt} for any $\veps > 0$, we can use Fubini's theorem to interchange the $\bx$- and $\by$-integrals, and use similar coordinate changes as in the proof of \eqref{eq:LocGreens:1} and come to
    \begin{equation*}
    \begin{split}
        &\int_{U_\veps} \int_{U^\veps} \rho_{\delta,\beta-p}^\veps(\by,\bx) \Phi_p' \left( \grad u(\bx) \cdot \frac{\bx-\by}{|\bx-\by|} \right)  \frac{ v(\bx) }{|\bx-\by|} \, \rmd \by  \, \rmd \bx \\
        &\qquad = \int_{ \{ \vnorm{\by'}_\infty < 2 \varrho \} } \int_1^\infty \int_{U_\veps^*(\by)} \frac{\veps \mathds{1}_{ \{ \veps < |\bz| < 1 \} } }{\eta_\delta(\by',\varphi(\by') + \veps y_d) |\bz|} \frac{ \rho(|\bz|) }{ |\bz|^{\beta-p} } \\
            &\qquad \quad\qquad \Phi_p' \left( \grad u(\by^*_{\bz,\veps})  \cdot \frac{\bz}{|\bz|} \right) v(\by^*_{\bz,\veps}) \, \rmd \bz \, \rmd y_d \, \rmd \by'\,,
        \end{split}
    \end{equation*}
    where $\by^*_{\bz,\veps} := (\by' + \eta_\delta(\by',\varphi(\by')+\veps y_d)\bz', \varphi(\by') + \veps y_d + \eta_\delta(\by',\varphi(\by')+\veps y_d)z_d)$ and 
    \begin{equation*}
    \begin{split}
        & U_\veps^*(\by) := 
        \left\{
        \bz =(\bz',z_d)\in \bbR^d \, : \,
        \begin{gathered}
        \vnorm{ \by'+\eta_\delta(\by',\varphi(\by')+\veps x_d)\bz' }_\infty < \varrho \,, \qquad \qquad \qquad  \\
       \veps y_d + \varphi(\by') + \eta_\delta(\by',\varphi(\by')+ \veps y_d)z_d -  \qquad  \qquad \quad \\
       \qquad  \qquad \varphi(\by'+\eta_\delta(\by',\varphi(\by')+\veps y_d)\bz') \in (0, \veps)
        \end{gathered}
        \right\}.
    \end{split}
    \end{equation*}
    Since $\by^*_{\bz,\veps} \to (\by',\varphi(\by'))$ as $\veps \to 0$ uniformly in $\by$ and $\bz$, the dominated convergence theorem can used the same way as in the proof of \Cref{thm:LocGreens:1}, and so
    \begin{equation*}
        \begin{split}
            &\int_{U_\veps} \int_{U^\veps} \int_{U_\veps} \int_{U^\veps} \rho_{\delta,\beta-p}^\veps(\by,\bx) \Phi_p' \left( \grad u(\bx) \cdot \frac{\bx-\by}{|\bx-\by|} \right)  \frac{ v(\bx) }{|\bx-\by|} \, \rmd \by  \, \rmd \bx \\
            &=\int_{  \{ \vnorm{\by'}_\infty < 2 \varrho \} } \int_{B(0,1)} \int_\bbR \lim\limits_{\veps \to 0}  \left[  \frac{\veps \mathds{1}_{\{ 1 < y_d < \infty \}} \mathds{1}_{ \{ |\bz| > \veps \} }  \mathds{1}_{U_\veps^*(\by)} }{\eta_\delta(\by',\varphi(\by') + \veps y_d) |\bz| } \right] \frac{ \rho(|\bz|) }{|\bz|^{\beta-p}} \\
            &\qquad\qquad  \qquad \Phi_p' \left( \grad u(\by',\varphi(\by')) \cdot \frac{\bz}{|\bz|} \right) v(\by',\varphi(\by')) \, \rmd y_d \, \rmd \bz  \, \rmd \by' \\
            &= \int_{ \{ \vnorm{\by'}_\infty < \varrho \} } \int_{B(0,1)} \int_1^{ \max \{ 1, \frac{1}{1-\delta q'(0) \bsnu(\by') \cdot \bz} \} }  \frac{\sqrt{1+|\grad \varphi(\by')|^2}}{ \delta q'(0) y_d |\bz| } \frac{ \rho(|\bz|) }{ |\bz|^{\beta-p} } \\
            &\qquad \qquad \qquad \Phi_p' \left( \grad u(\by',\varphi(\by')) \cdot \frac{\bz}{|\bz|} \right) v(\by',\varphi(\by')) \, \rmd y_d  \, \rmd \bz \, \rmd \by' \,.
        \end{split}
    \end{equation*}
    Applying \eqref{eq:LogIntegral:1}, we obtain \eqref{eq:LocGreens:2} in the case $q'(0) > 0$.

    The proof of \eqref{eq:LocGreens:2} in the case $N > 1$ follows in a similar way to that of \eqref{eq:LocGreens:1}, again using Fubini's theorem at the beginning of the proof.
\end{proof}

\begin{corollary}\label{cor:LocGreens}
    Under the same assumptions as \Cref{thm:LocGreens:1},
    \begin{equation}\label{eq:LocGreens}
        \begin{split}
        \lim\limits_{\veps \to 0} 
        &\int_{U_\veps} \int_{U^\veps} \hat{\rho}_{\delta,\beta-p}^\veps(\bx,\by) \Phi_p' \left( \grad u(\bx) \cdot \frac{\bx-\by}{|\bx-\by|} \right)  \frac{ v(\bx) }{|\bx-\by|} \, \rmd \by  \, \rmd \bx \\
        &= \int_{Q} BF_{p,\delta}^{N}( \grad u(\bx',\varphi(\bx')),\bsnu(\bx')) v(\bx,\varphi(\bx')) \sqrt{1+ |\grad \varphi(\bx')|^2} \, \rmd \bx'\,.
        \end{split}
    \end{equation}
\end{corollary}

\begin{proof}
    We only need to check that $BF_{p,\delta}^{1,+}(\ba,\bstheta) + BF_{p,\delta}^{1,-}(\ba,\bstheta) = BF_{p,\delta}^{1}(\ba,\bstheta)$,
    which follows from the identity
   \begin{equation*}
    \begin{split}
        \int_{B(0,1) \cap  \{ \bstheta \cdot \bz > 0 \} } 
       &
        \frac{ \ln \left( \frac{ 1+ \delta q'(0) (\bstheta \cdot \bz)
        }{ 1 - \delta q'(0) (\bstheta \cdot \bz) } \right) }{ \delta q'(0) |\bz| } \frac{\rho(|\bz|)}{ |\bz|^{\beta-p} } \Phi_p' \left( \ba \cdot \frac{\bz}{|\bz|} \right) 
        \, \rmd \bz \\
        &=\int_{B(0,1)\cap  \{ \bstheta \cdot \bz < 0 \} } 
        \frac{ \ln \left( \frac{ 1+ \delta q'(0) (\bstheta \cdot \bz) }{ 1 - \delta q'(0) (\bstheta \cdot \bz) } \right) }{ \delta q'(0) |\bz| 
        } \frac{\rho(|\bz|)}{ |\bz|^{\beta-p} } \Phi_p' \left( \ba \cdot \frac{\bz}{|\bz|} \right) 
        \, \rmd \bz
    \end{split}
    \end{equation*}
    obtained by the change of variables $\bz \to -\bz$.
\end{proof}

\subsection{Proof of the Nonlocal Green's identity}
\label{sec:pfGreen}
\begin{proof}[Proof of \Cref{thm:Intro:GreensIdentity}]
    Thanks to the trace property of the nonlocal function space $\mathfrak{W}^{\beta,p}[\delta;q](\Omega)$, the right-hand side of \eqref{eq:GreensIdentity} is continuous in the $\mathfrak{W}^{\beta,p}[\delta;q](\Omega)$-norm, and so by density it suffices to show \eqref{eq:GreensIdentity} for $v \in C^1(\overline{\Omega})$ (see \Cref{thm:FxnSpaceProp}).
    Let $\veps > 0$ with $\veps \ll \delta$.
    For $\veps$ small enough, cover the set $\Omega_{\veps;\lambda,q}$ with sets of the form (up to a rigid motion)
    \begin{equation*}
        U_{\veps}^i := \{ \bx = (\bx',x_d) \in \Omega \, : \, \vnorm{\bx'}_\infty \leq \varrho_i \,, \varphi_i(\bx') < x_d < \varphi_i(\bx') + \veps, \; 1 \leq i \leq N,
    \end{equation*}
    where $\varrho_i \gg \veps$ and $\varphi_i :\bbR^{d-1} \to [0,\infty)$ are Lipschitz functions.
    Define
    \begin{equation*}
        \wt{U}_{\veps}^i := \{ \bx = (\bx',x_d) \in \Omega \, : \, \vnorm{\bx'}_\infty \leq 2\varrho_i \,, \varphi_i(\bx') < x_d < \varphi_i(\bx') + \veps \},  \; \wt{U}_\veps := \cup_{i = 1}^N \wt{U}_\veps^i.
    \end{equation*}
 Since
    \begin{equation*}
    \begin{split}
        \left| \Phi_p' \left( \frac{u(\bx)-u(\by)}{|\bx-\by|} \right) \right| \frac{ |v(\bx)-v(\by)| }{|\bx-\by|} 
        \leq 
        \Vnorm{\grad u}_{L^{\infty} (\Omega)}^{p-1} \Vnorm{\grad v}_{L^{\infty}(\Omega)},
    \end{split}
    \end{equation*}
    the right-hand side quantity multiplied by $\rho_{\delta,\beta-p}^\veps(\bx,\by)$ belongs to $L^1(\Omega \times \Omega)$ by \eqref{eq:KernelIntFunction:Bounds}, and so we get by symmetry and the dominated convergence theorem
	\begin{equation*}
		\begin{split}
			&\cB_{p,\delta}(u,v) =\iintdm{\Omega}{\Omega}{ \rho_{\delta,\beta-p}(\bx,\by) \Phi_p' \left( \frac{u(\bx)-u(\by)}{|\bx-\by|} \right) \left( \frac{v(\bx)-v(\by)}{|\bx-\by|} \right) }{\by}{\bx} \\
			&\quad= \lim\limits_{\veps \to 0} \iintdm{\Omega \setminus \wt{U}_\veps }{\Omega \setminus \wt{U}_\veps }{  \frac{ \hat{\rho}_{\delta,\beta-p}^\veps(\bx,\by) }{2} \Phi_p' \left( \frac{u(\bx)-u(\by)}{|\bx-\by|} \right) \left( \frac{v(\bx)-v(\by)}{|\bx-\by|} \right)  }{\by}{\bx}\,.
		\end{split}
	\end{equation*}
	Now we can use nonlocal integration by parts on the truncated form to get
	\begin{equation*}
		\begin{split}
			&\iintdm{\Omega \setminus \wt{U}_\veps }{\Omega \setminus \wt{U}_\veps }{  \frac{ \hat{\rho}_{\delta,\beta-p}^\veps(\bx,\by) }{2} \Phi_p' \left( \frac{u(\bx)-u(\by)}{|\bx-\by|} \right) \left( \frac{v(\bx)-v(\by)}{|\bx-\by|} \right)  }{\by}{\bx} \\
			&= - \int_{\Omega \setminus \wt{U}_\veps} \int_{\wt{U}_\veps} \hat{\rho}_{\delta,\beta-p}^\veps(\bx,\by) \Phi_p' \left( \frac{u(\bx)-u(\by)}{|\bx-\by|} \right)  \frac{ v(\bx) }{|\bx-\by|}  \, \rmd \by  \, \rmd \bx \\
                &\qquad + \int_{\Omega \setminus \wt{U}_\veps } \mathds{1}_{ \Omega^{\veps;\lambda,q} }(\bx) \int_{\Omega} \hat{\rho}_{\delta,\beta-p}^\veps(\bx,\by) \Phi_p' \left( \frac{u(\bx)-u(\by)}{|\bx-\by|} \right)  \frac{ v(\bx) }{|\bx-\by|} \, \rmd \by  \, \rmd \bx \\
		  &= - \int_{\Omega \setminus \wt{U}_\veps} \int_{\wt{U}_\veps} \hat{\rho}_{\delta,\beta-p}^\veps(\bx,\by) \Phi_p' \left( \frac{u(\bx)-u(\by)}{|\bx-\by|} \right)  \frac{ v(\by) }{|\bx-\by|}  \, \rmd \by  \, \rmd \bx \\
			&\qquad - \int_{\Omega \setminus \wt{U}_\veps} \int_{\wt{U}_\veps} \hat{\rho}_{\delta,\beta-p}^\veps(\bx,\by) \Phi_p' \left( \frac{u(\bx)-u(\by)}{|\bx-\by|} \right)  \frac{ v(\bx)-v(\by) }{|\bx-\by|}  \, \rmd \by  \, \rmd \bx \\
			&\qquad + \int_{\Omega \setminus \wt{U}_\veps}\hspace{-2pt} \cL_{p,\delta}^\veps u(\bx) v(\bx) \, \rmd \bx\,.
		\end{split}
	\end{equation*}
        By the dominated convergence theorem, the second integral vanishes. 
    So it remains to show 
	\begin{equation}\label{eq:GreensIdPf:Limit1}
        \begin{split}
            \lim\limits_{\veps \to 0} & \int_{\wt{U}_\veps} \int_{\Omega \setminus \wt{U}_\veps} \hat{\rho}_{\delta,\beta-p}^\veps(\bx,\by) \Phi_p' \left( \frac{u(\bx)-u(\by)}{|\bx-\by|} \right)  \frac{ v(\bx) }{|\bx-\by|} \, \rmd \by  \, \rmd \bx \\
            &= \int_{\p \Omega} BF_{p,\delta}^N(\grad u,\bsnu) v \, \rmd \sigma\,.
        \end{split}
    \end{equation}
    Now, by \eqref{eq:KernelIntFunction:Bounds} again
    \begin{equation*}
		\lim\limits_{\veps \to 0} \int_{\wt{U}_\veps } \int_{\Omega \setminus \wt{U}_\veps } \hat{\rho}_{\delta,\beta-p}^\veps(\bx,\by) \, \rmd \by  \, \rmd \bx \leq \lim\limits_{\veps \to 0}  C |\wt{U}_\veps| = 0\,.
	\end{equation*}
    Thus, on the left-hand side integrand of \eqref{eq:GreensIdPf:Limit1}, we 
        use the Taylor expansion
        \begin{equation}\label{eq:Taylor}
             \frac{u(\bx)-u(\by)}{|\bx-\by|} = \grad u(\bx) \cdot  \frac{\bx-\by}{|\bx-\by|} - \int_0^1 \grad^2 u(\bx + t(\by-\bx))(\bx-\by) \cdot \frac{\bx-\by}{|\bx-\by|} (1-t) \, \rmd t,
        \end{equation}
        as well as the local Lipschitz condition on $\Phi_p'$ to see that \eqref{eq:GreensIdPf:Limit1} is equivalent to	
    \begin{equation*}
        \begin{split}
        \lim\limits_{\veps \to 0}  & \int_{\wt{U}_\veps} \int_{\Omega \setminus \wt{U}_\veps} \hat{\rho}_{\delta,\beta-p}^\veps(\bx,\by) \Phi_p' \left( \grad u(\bx) \cdot \frac{\bx-\by}{|\bx-\by|} \right)  \frac{ v(\bx) }{|\bx-\by|} \, \rmd \by  \, \rmd \bx \\ &
        = \int_{\p \Omega} BF_{p,\delta}^N(\grad u,\bsnu) v \, \rmd \sigma\,.
        \end{split}
    \end{equation*}

    Now we localize the problem. Let $\zeta_i^\veps$ be a $C^\infty$ partition of unity subordinate to the cover $\{ U^i_\veps \}_{i=1}^N$ (enlarging the cover slightly if necessary so that $\sum_{i} \zeta_i(\bx) \equiv 1$ for all $\bx \in \overline{\Omega}$). We can choose $\veps > 0$ small enough so that, for any $\bx \in U^i_\veps$ and 
    $i \in \{1, \ldots, N \}$,  any $\by \in \Omega \setminus \wt{U}_\veps$ with either $|\by-\bx| \leq \eta_\delta(\bx)$ or $|\by-\bx| \leq \eta_\delta(\by)$ actually satisfies (after a rigid motion) $\vnorm{\by'}_\infty < 2 \varrho_i$ and $y_d > \varphi_i(\by') + \veps$.
    Therefore, in the notation of \Cref{thm:LocGreens:1},
    \begin{equation*}
        \begin{split}
            &\int_{\wt{U}_\veps} \int_{\Omega \setminus \wt{U}_\veps} \hat{\rho}_{\delta,\beta-p}^\veps(\bx,\by) \Phi_p' \left( \grad u(\bx) \cdot \frac{\bx-\by}{|\bx-\by|} \right)  \frac{ v(\bx) }{|\bx-\by|} \, \rmd \by  \, \rmd \bx \\
            &= \sum_{i = 1}^N \int_{U_\veps^i} \int_{U^{\veps,i} } \hat{\rho}_{\delta,\beta-p}^\veps(\bx,\by) \Phi_p' \left( \grad u(\bx) \cdot \frac{\bx-\by}{|\bx-\by|} \right)  \frac{ \zeta_i(\bx) v(\bx) }{|\bx-\by|} \, \rmd \by  \, \rmd \bx\,.
        \end{split}
    \end{equation*}
    By \Cref{cor:LocGreens} each of the terms in the sum converges to
    \begin{equation*}
        \int_{Q_i} BF_{p,\delta}^N(\grad u(\bx',\varphi_i(\bx')),\bsnu_i(\bx')) (\zeta_i v)(\bx',\varphi_i(\bx')) \sqrt{ 1 + |\grad \varphi_i(\bx')|^2 } \, \rmd \bx'\,,
    \end{equation*}
    and the result follows.
\end{proof}

\section{Realizations of the nonlocal operator}\label{sec:NonlocalOperator}

In this section, we show that in special cases, $(\mathrm{d}\text{-}\cL)_{p,\delta}$ agrees with a pointwise integral operator, defined in a principal-value sense. We will also define a ``weak derivative'' form of the operator, so that the class of functions $u$ on which $\cL_{p,\delta}$ acts can be relaxed.

For this section, we continue to assume $p \geq 2$, \eqref{assump:beta}, \eqref{assump:Localization}, \eqref{assump:NonlinearLocalization}, \eqref{assump:VarProb:Kernel}, \eqref{eq:assump:Phi}, and \eqref{eq:HorizonThreshold}.
At times we will also assume \eqref{assump:FullLocalization:C11}; this condition will make all the difference on whether the operator can be extended past its distributional definition. 

We describe several configurations of $\eta$ with conditions sufficient to guarantee \eqref{assump:FullLocalization:C11}; of course there are additional possibilities.
First, if it happens that $\Omega$ is a $C^2$ domain, then then one can choose $\lambda = d_{\p \Omega}$. However, some care must be taken, as $d_{\p \Omega}$ does not belong to $C^{2}(\overline{\Omega})$, but rather there exists $\veps_\Omega > 0$ such that $d_{\p \Omega}$ is $C^2$ on the set $\{ \bx \in \overline{\Omega} : \, d_{\p \Omega}(\bx) \leq \veps_\Omega \}$; see \cite{foote1984regularity}. However, if $q$ is chosen to satisfy \eqref{assump:NonlinearLocalization} for $k_q = 2$ with $q(r)$ constant for $r \geq \veps_\Omega$, it follows that the resulting heterogeneous localization $\eta[d_{\p \Omega},q]$ satisfies \eqref{assump:FullLocalization:C11}.
Second, if it is not the case that $\Omega$ is $C^2$, then one can consider, in place of $d_{\p \Omega}$, a generalized distance $\lambda$ satisfying \eqref{assump:Localization} (and additionally \eqref{assump:Localization:NormalDeriv} if desired) for some $k_\lambda \geq 2$. Then $\eta(\bx) = q(\lambda(\bx))$ belongs to $C^2(\Omega)$, but it is not guaranteed that its derivatives remain bounded near $\p \Omega$. In that case one can choose instead any function $\tilde{q}$ that satisfies \eqref{assump:MomentsOfNonlinLoc} for $N = 2$.
Then an application of Fa\`a di Bruno's formula shows that 
$|D^\alpha \eta[\lambda,\tilde{q}](\bx)| \leq C d_{\p \Omega}(\bx)^{2-|\alpha|}$ for all $\bx \in \Omega$ and for all $|\alpha| \leq 2$, where $C$ depends only on $\tilde{q}$, $\alpha$, and $\kappa_\alpha$, i.e. $\eta[\lambda,\tilde{q}]$ satisfies \eqref{assump:FullLocalization:C11}.

\subsection{The pointwise definition}\label{subsec:PointwiseOperator}

We establish a few lemmas first:

\begin{lemma}\label{lma:PointwiseOperator:Term1}
    Let $u \in C^2(\overline{\Omega})$. For $\bx \in \Omega$ define the function
    \begin{equation}
    \begin{split}
        \cD_{p,\delta}^{1,\veps}[u](\bx)
        := \mathds{1}_{ \Omega^{\veps;\lambda,q} }(\bx)  \int_\Omega I_{p,\delta}^{1,\veps}[u](\bx,\by) \, \rmd \by\,,
    \end{split}
    \end{equation}
    where
    \begin{equation}
    \begin{split}
        &I_{p,\delta}^{1,\veps}[u](\bx,\by) \\
        &:= \rho_{\delta,\beta-p}^\veps(\by,\bx) \Phi_p' \left( \grad u(\bx) \cdot \frac{\bx-\by}{|\bx-\by|} \right) \frac{ \eta_\delta(\bx) - \eta_\delta(\by) - \grad \eta_\delta(\by) \cdot (\bx-\by) }{ |\bx-\by| \eta_\delta(\bx) }\,.
    \end{split}
    \end{equation}
    Then for every $\bx \in \Omega$ and every $\veps >0$, $|I_{p,\delta}^{1,\veps}[u](\bx,\by)| \leq \rho_{\delta,\beta-p}(\by,\bx) |\grad u(\bx)|^{p-1} \frac{ 2 \kappa_1 }{\eta(\bx)}$, and consequently $\cD_{p,\delta}^{1,\veps}[u](\bx)$ converges almost everywhere in $\Omega$ to
    \begin{equation*}
        \cD_{p,\delta}^{1}[u](\bx) := \int_{\Omega} I_{p,\delta}^{1}[u](\bx,\by) \, \rmd \by\,,
    \end{equation*}
    where
    \begin{equation}
    \begin{split}
        &I_{p,\delta}^{1}[u](\bx,\by) \\
        &:= \rho_{\delta,\beta-p}(\by,\bx) \Phi_p' \left( \grad u(\bx) \cdot \frac{\bx-\by}{|\bx-\by|} \right) \frac{ \eta_\delta(\bx) - \eta_\delta(\by) - \grad \eta_\delta(\by) \cdot (\bx-\by) }{ |\bx-\by| \eta_\delta(\bx) }\,.
    \end{split}
    \end{equation}
  The function
    $\cD_{p,\delta}^{1}[u](\bx)$ belongs to $L^1_{loc}(\Omega)$ with $|\cD_{p,\delta}^{1}[u](\bx)| \leq \frac{ C |\grad u(\bx)|^{p-1} }{\eta(\bx)}$, where $C$ depends only on $\rho$, $\beta$, $p$, and $\kappa_1$.

    If additionally \eqref{assump:FullLocalization:C11} is satisfied, then
    \begin{equation}
        |I_{p,\delta}^{1}[u](\bx,\by)| \leq \delta \frac{ \vnorm{ \grad^2 \eta }_{L^\infty(\Omega)} }{2(1-\kappa_1 \delta)} \rho_{\delta,\beta-p}(\by,\bx) |\grad u(\bx)|^{p-1}\,,
    \end{equation}
    and consequently the integral defining $\cD_{p,\delta}^{1}[u](\bx)$ is absolutely convergent, with
    \begin{equation}
        |\cD_{p,\delta}^{1}[u](\bx)| \leq C(\beta,p,\rho,q,\lambda,\kappa_1) \delta |\grad u(\bx)|^{p-1}\,, \quad \forall \, \bx \in \Omega\,.
    \end{equation}
    Moreover, if \eqref{assump:FullLocalization:C11} is satisfied, then the convergence $\cD_{p,\delta}^{1,\veps}[u](\bx) \to \cD_{p,\delta}^{1}[u](\bx)$ as $\veps \to 0$ can be strengthened to the strong $L^r(\Omega)$ topology, for any $r \in [1,\infty)$, with
    \begin{equation*}
        \int_{\Omega} 
        |\cD_{p,\delta}^{1}[u](\bx)|^r \, \rmd \bx \leq C(\beta,p,\rho,q,\lambda,\kappa_1) \delta \int_{\Omega} |\grad u(\bx)|^{r(p-1)} \, \rmd \bx\,.
    \end{equation*}
\end{lemma}

\begin{proof}
    The results can be obtained in a straightforward way using the assumptions on $\eta$, the bound  \eqref{eq:KernelIntFunction:Bounds}, and the dominated convergence theorem.
\end{proof}

To state the next lemma, we introduce, for $u \in C^2(\overline{\Omega})$ and $\ba, \bfb\in \mathbb{R}^d$, a notation for the remainder term of the Taylor expansion
    \begin{equation*}
        R[u](\ba,\bfb) := - \int_0^1  \left( \grad^2 u(\ba + t \bfb) \frac{\bfb}{|\bfb|} \right) \cdot \frac{\bfb}{|\bfb|}  (1-t) \, \rmd t\,.
    \end{equation*}

\begin{lemma}\label{lma:PointwiseOperator:Term2}
    Let $u \in C^2(\overline{\Omega})$. For $\bx \in \Omega$ define the function
    \begin{equation}
    \begin{split}
        \cD_{p,\delta}^{2,\veps}[u](\bx)
        := \mathds{1}_{ \Omega^{\veps;\lambda,q} }(\bx) \int_\Omega I_{p,\delta}^{2,\veps}[u](\bx,\by) \, \rmd \by\,,
    \end{split}
    \end{equation}
    where
    \begin{equation*}
    \begin{split}
        I_{p,\delta}^{2,\veps}[u](\bx,\by) := \frac{ \hat{\rho}_{\delta,\beta-p}^\veps(\bx,\by) }{|\bx-\by|} \int_0^{ |\by-\bx| R[u](\bx,\by-\bx) } \Phi_p'' \left( \grad u(\bx) \cdot \frac{\bx-\by}{ |\bx-\by| } + \tau \right) \, \rmd \tau\,,
    \end{split}
    \end{equation*}
    Then for every $\bx \in \Omega$ and every $\veps >0$,
    \begin{equation}\label{eq:PointwiseOperator:BddIntegrand}
        \begin{split}
          &  |I_{p,\delta}^{2,\veps}[u](\bx,\by)| \leq C(p) |\grad u(\bx)|^{p-2} \hat{\rho}_{\delta,\beta-p}(\bx,\by) |R[u](\bx,\by-\bx)| \\
            &\qquad\quad + (|\by-\bx||R[u](\bx,\by-\bx)|)^{p-2} \hat{\rho}_{\delta,\beta-p}(\bx,\by) |R[u](\bx,\by-\bx)|\,.
        \end{split}
    \end{equation}
    and consequently by the dominated convergence theorem $\cD_{p,\delta}^{2,\veps}[u](\bx)$ converges almost everywhere in $\Omega$ to
    \begin{equation*}
        \cD_{p,\delta}^{2}[u](\bx) := \int_{\Omega} I_{p,\delta}^{2}[u](\bx,\by) \, \rmd \by\,, \quad \text{a.e. } \bx \in \Omega\,,
    \end{equation*}
    where
    \begin{equation}
    \begin{split}
        I_{p,\delta}^{2}[u](\bx,\by) := \frac{ \hat{\rho}_{\delta,\beta-p}(\bx,\by) }{|\bx-\by|} \int_0^{ |\by-\bx| R[u](\bx,\by-\bx) } \Phi_p'' \left( \grad u(\bx) \cdot \frac{\bx-\by}{ |\bx-\by| } + \tau \right) \, \rmd \tau\,.
    \end{split}
    \end{equation}
    The function $I_{p,\delta}^{2}[u](\bx,\by)$ satisfies the same bound \eqref{eq:PointwiseOperator:BddIntegrand}, and hence the integral defining $\cD_{p,\delta}^{2}[u](\bx)$ is absolutely convergent. Moreover, for any $r \in [1,\infty)$ there exists a constant $C$ depending only on $\rho$, $\beta$, $p$, and $r$ such that
    \begin{equation}\label{eq:PointwiseOperator:2:LrBdd}
    \vnorm{ \cD_{p,\delta}^2[u] }_{L^{r}(\Omega)}
        \leq C \big( \vnorm{\grad u}_{L^{r(p-1)}(\Omega)}^{p-2} + \Vnorm{ \grad^2 u }_{L^{r(p-1)}(\Omega)}^{p-2} \big) \Vnorm{ \grad^2 u }_{L^{r(p-1)}(\Omega)},
    \end{equation}
    and $\cD_{p,\delta}^{2}[u](\bx) \to \cL_{p,0}u(\bx)$ as $\delta \to 0$ 
    strongly in $L^r(\Omega)$.
\end{lemma}

\begin{proof}
    We will only prove \eqref{eq:PointwiseOperator:2:LrBdd} and the almost everywhere / $L^r(\Omega)$ convergence, since the other conclusions follow in a straightforward way from \eqref{eq:assump:Phi}. 
    Using \eqref{eq:PointwiseOperator:BddIntegrand} and H\"older's inequality, 
    \begin{equation*}
    \begin{split}
        &|\cD_{p,\delta}^{2}[u](\bx)|^r \\
        &\leq |\grad u(\bx)|^{r(p-2)} \left( \int_{\Omega} \hat{\rho}_{\delta,\beta-p}(\bx,\by) \, \rmd \by \right)^{r-1}
        \int_{\Omega} \hat{\rho}_{\delta,\beta-p}(\bx,\by) |R[u](\bx,\by-\bx)|^r \, \rmd \by
        \\
        &\; + \left( \int_{\Omega} \hat{\rho}_{\delta,\beta-p}(\bx,\by) |\bx-\by|^{\frac{r}{r-1}} \rmd \by \right)^{r-1}
        \int_{\Omega} \hat{\rho}_{\delta,\beta-p}(\bx,\by) |R[u](\bx,\by-\bx)|^{r(p-1)} \rmd \by \\
        &\leq C \int_{\Omega} \hat{\rho}_{\delta,\beta-p}(\bx,\by) \big( |\grad u(\bx)|^{r(p-2)} |R_2(u,\bx,\by-\bx)|^r + |R_2(u,\bx,\by-\bx)|^{r(p-1)} \big) \rmd \by,
    \end{split}
    \end{equation*}
    where we additionally used \eqref{eq:KernelIntFunction:Bounds}. Again by H\"older's inequality and \eqref{eq:KernelIntFunction:Bounds},
    \begin{equation*}
        \begin{split}
            &\int_\Omega |\cD_{p,\delta}^{2}[u](\bx)|^r \, \rmd \bx 
            \leq C \left( \int_\Omega \int_\Omega \hat{\rho}_{\delta,\beta-p}(\bx,\by) |\grad u(\bx)|^{r(p-1)} \, \rmd \by \, \rmd \bx \right)^{\frac{p-2}{p-1}} \\
            &\qquad \qquad\qquad  \left( \int_{\Omega} \int_{\Omega} \hat{\rho}_{\delta,\beta-p}(\bx,\by) |R[u](\bx,\by-\bx)|^{r(p-1)} \, \rmd \by \, \rmd \bx \right)^{1/(p-1)} \\
            &\qquad\qquad\qquad + C \int_{\Omega} \int_{\Omega} \hat{\rho}_{\delta,\beta-p}(\bx,\by) |R[u](\bx,\by-\bx)|^{r(p-1)} \, \rmd \by \, \rmd \bx  \\
            &\leq C \Vnorm{ \grad u }_{L^{r(p-1)}(\Omega)}^{r(p-2)} \left( \int_{\Omega} \int_{\Omega} \hat{\rho}_{\delta,\beta-p}(\bx,\by) |R[u](\bx,\by-\bx)|^{r(p-1)} \, \rmd \by \, \rmd \bx \right)^{1/(p-1)} \\
            &\;\; + C \int_{\Omega} \int_{\Omega} \hat{\rho}_{\delta,\beta-p}(\bx,\by) |R[u](\bx,\by-\bx)|^{r(p-1)} \, \rmd \by \, \rmd \bx \,.
        \end{split}
    \end{equation*}
    Now, by Tonelli's theorem and by \Cref{lma:CoordChange2}
    \begin{equation*}
    \begin{split}
        &\int_{\Omega} \int_{\Omega} \hat{\rho}_{\delta,\beta-p}(\bx,\by) |R[u](\bx,\by-\bx)|^{r(p-1)} \, \rmd \by \, \rmd \bx \\
        &= \int_{\Omega} \int_{B(0,1)} \frac{\rho(|\bz|)}{|\bz|^{\beta-p}} \left| \int_0^1 \left( \grad^2 u(\bx + t \eta_\delta(\bx) \bz) \frac{\bz}{|\bz|} \right) \cdot \frac{\bz}{|\bz|}   (1-t) \, \rmd t \right|^{r(p-1)} \, \rmd \bz \, \rmd \bx \\
        &\quad + \int_{\Omega} \int_{ B(0,1) } \frac{\rho(|\bz|)}{|\bz|^{\beta-p}} \left| \int_0^1 \left( \grad^2 u(\by + (1-t) \eta_\delta(\by) \bz ) \frac{\bz}{|\bz|} \right) \cdot  \frac{\bz}{|\bz|}  (1-t) \, \rmd t \right|^{r(p-1)} \, \rmd \bz \, \rmd \by \\
        &\leq 2 \int_0^1 \int_{\Omega} \int_{B(0,1)} \frac{\rho(|\bz|)}{|\bz|^{\beta-p}} |\grad^2 u( \bszeta_{\bz}^{t \delta}(\bx) )|^{r(p-1)} \, \rmd \bz \, \rmd \bx \, \rmd t
        \leq 3 \bar{\rho}_{p-\beta} \int_{\Omega} |\grad^2 u|^{r(p-1)} \, \rmd \bx\,,
    \end{split}
    \end{equation*}
    and so \eqref{eq:PointwiseOperator:2:LrBdd} is proved.
    The convergence result follows similar reasoning. Since $\Phi_p''$ is even, we have
    \begin{equation*}
        \begin{split}
            &\cD_{p,\delta}^{2}[u](\bx) \\
            &= \int_{B(0,1)} \frac{\rho(|\bz|)}{|\bz|^{\beta-p}} \frac{1}{|\bz| \eta_\delta(\bx)} \int_0^{ |\bz| \eta_\delta(\bx) R[u](\bx, \eta_\delta(\bx) \bz) } \Phi_p'' \left( \grad u(\bx) \cdot \frac{\bz}{|\bz|} - \tau \right) \, \rmd \tau \, \rmd \bz \\
            &\;\; + \int_{ \{ |\bsupsilon_{\bx}^\delta(\by)| < 1 \} } \frac{\rho(|\bsupsilon_{\bx}^\delta(\by)|)}{|\bsupsilon_{\bx}^\delta(\by)|^{\beta-p}} \frac{1}{|\bsupsilon_{\bx}^\delta(\by)| \eta_\delta(\by)} \\
            &\qquad \qquad \qquad \int_0^{ |\bsupsilon_{\bx}^\delta(\by)| \eta_\delta(\by) R[u](\bx, \eta_\delta(\by) \bsupsilon_{\bx}^\delta(\by) ) } \Phi_p'' \left( \grad u(\bx) \cdot \frac{\bsupsilon_{\bx}^\delta(\by)}{|\bsupsilon_{\bx}^\delta(\by)|} - \tau \right) \, \rmd \tau \, \rmd \by\,.
        \end{split}
    \end{equation*}
    By \eqref{eq:ComparabilityOfDistanceFxn2} and \Cref{lma:CoordChange1}
    \begin{equation*}
        \eta_\delta(\by) = \eta_\delta(\bx) + O(\delta) , \quad \text{ and }\; \eta_\delta(\by)^d  \det \grad \bsupsilon_{\bx}^\delta(\by) = 1 + O(\delta)\,,
    \end{equation*}
    where the comparison constants are independent of $\bx$ and $\by$, so
    \begin{equation*}
        \begin{split}
            &\cD_{p,\delta}^{2}[u](\bx) \\
            &= \int_{B(0,1)} \frac{\rho(|\bz|)}{|\bz|^{\beta-p}} \frac{1}{|\bz| \eta_\delta(\bx)} \int_0^{ |\bz| \eta_\delta(\bx) R[u](\bx, \eta_\delta(\bx) \bz) } \Phi_p'' \left( \grad u(\bx) \cdot \frac{\bz}{|\bz|} - \tau \right) \, \rmd \tau \, \rmd \bz \\
            &\quad + O( \delta |\cD_{p,\delta}^{2}[u](\bx)| ) + \int_{B(0,1)} \frac{\rho(|\bz|)}{|\bz|^{\beta-p}} \frac{1}{|\bz| (\eta_\delta(\bx)+O(\delta))} \\
            &\qquad \quad \int_0^{  (\eta_\delta(\bx)+O(\delta)) R[u](\bx, (\eta_\delta(\bx)+O(\delta)) \bz ) } \Phi_p'' \left( \grad u(\bx) \cdot \frac{\bz}{|\bz|} - \tau \right) \, \rmd \tau \, \rmd \bz \,.
        \end{split}
    \end{equation*}
    Changing variables in the $\tau$ integrals gives
    \begin{equation*}
        \begin{split}
            &\cD_{p,\delta}^{2}[u](\bx) \\
            &= \int_{B(0,1)} \frac{\rho(|\bz|)}{|\bz|^{\beta-p}} \int_0^{ R[u](\bx, \eta_\delta(\bx) \bz) } \Phi_p'' \left( \grad u(\bx) \cdot \frac{\bz}{|\bz|} - |\bz| \eta_\delta(\bx) \tau \right) \, \rmd \tau \, \rmd \bz \\
            &\quad + O( \delta |\cD_{p,\delta}^{2}[u](\bx)| ) + \int_{B(0,1)} \frac{\rho(|\bz|)}{|\bz|^{\beta-p}} \\
            &\qquad \quad \int_0^{R[u](\bx, (\eta_\delta(\bx)+O(\delta)) \bz ) } \Phi_p'' \left( \grad u(\bx) \cdot \frac{\bz}{|\bz|} - |\bz| (\eta_\delta(\bx)+O(\delta)) \tau \right) \, \rmd \tau \, \rmd \bz \,.
        \end{split}
    \end{equation*}
    Thanks to \eqref{eq:PointwiseOperator:2:LrBdd}, it is now clear that the dominated convergence theorem can be applied to obtain
    \begin{equation*}
    \begin{split}
        &\lim\limits_{\delta \to 0} \cD_{p,\delta}^{2}[u](\bx) \\
        =& 2 \int_{B(0,1)} \frac{\rho(|\bz|)}{|\bz|^{\beta-p}} \Phi_p'' \left( \grad u(\bx) \cdot \frac{\bz}{|\bz|} \right) \lim\limits_{\delta \to 0} \int_0^{R[u](\bx,\eta_\delta(\bx) \bz) } \, \rmd \tau \, \rmd \bz \\
        =& -2 \int_{B(0,1)} \frac{\rho(|\bz|)}{|\bz|^{\beta-p}} \Phi_p'' \left( \grad u(\bx) \cdot \frac{\bz}{|\bz|} \right) \left(  \grad^2 u(\bx) \frac{\bz}{|\bz|} \right) \cdot  \frac{\bz}{|\bz|} \, \rmd \bz \int_0^1 (1-t) \, \rmd t \,,
    \end{split}
    \end{equation*}
    which after integration using polar coordinates is exactly $\cL_{p,0}u(\bx)$.
    It is also clear that the dominated convergence theorem can be applied to $\vnorm{ \cD_{p,\delta}^2[u]}_{L^r(\Omega)}$ to obtain the convergence in $L^r(\Omega)$.
\end{proof}

The previous two lemmas allow us to summarize the conditional existence and convergence of a pointwise operator obtained from $\cL_{p,\delta}^\veps$.

\begin{corollary}\label{cor:PointwiseOperator:Summary}
    Let $u \in C^2(\overline{\Omega})$ and let $r \in [1,\infty)$. Then
    \begin{equation}
        \cL_{p,\delta}^\veps u(\bx) = \cD_{p,\delta}^{1,\veps}[u](\bx) + \cD_{p,\delta}^{2,\veps}[u](\bx)
    \end{equation}
    for all $\bx \in \Omega$, and $\cL_{p,\delta}^\veps u(\bx) \to \wt{\cL}_{p,\delta} u(\bx)$ as $\veps \to 0$ almost everywhere in $\Omega$, where
    \begin{equation*}
        \wt{\cL}_{p,\delta}u(\bx) := \cD_{p,\delta}^{1}[u](\bx) + \cD_{p,\delta}^{2}[u](\bx) \in L^1_{loc}(\Omega)\,.
    \end{equation*}
    If \eqref{assump:FullLocalization:C11} is additionally satisfied, then $\cL_{p,\delta}^\veps u(\bx) \to \wt{\cL}_{p,\delta} u(\bx)$ as $\veps \to 0$ strongly in $L^r(\Omega)$, there exists a constant $C = C(\rho,\beta,p,r) > 0$ such that
    $$
    \vnorm{\wt{\cL}_{p,\delta} u}_{L^r(\Omega)} \leq C \big(\vnorm{\grad u}_{L^{r(p-1)}(\Omega)}^{p-1} 
        + \vnorm{\grad^2 u}_{L^{r(p-1)}(\Omega)}^{p-1}
        \big)\,,
    $$
    and $\wt{\cL}_{p,\delta} u \to \cL_{p,0} u$ as $\delta \to 0$ 
    strongly in $L^r(\Omega)$.
\end{corollary}

\begin{proof}
    Using \eqref{eq:Taylor} and \eqref{eq:assump:Phi},
    \begin{equation*}
        \begin{split}
            &\Phi_p' \left( \frac{u(\bx)-u(\by)}{|\bx-\by|} \right) \\
            &= 
            \Phi_p' \left( \grad u(\bx) \cdot \frac{\bx-\by}{|\bx-\by|} \right) 
            + \int_{0}^{ |\bx-\by| R[u](\bx,\by-\bx)} \Phi_p'' \left(  \grad u(\bx) \cdot \frac{\bx-\by}{|\bx-\by|} + \tau \right) \, \rmd \tau\,,
        \end{split}
    \end{equation*}
    so by the conclusions of \Cref{lma:PointwiseOperator:Term1} and \Cref{lma:PointwiseOperator:Term2} we just need to show that
    \begin{equation}
        \begin{split}
            \mathds{1}_{ \Omega^{\veps;\lambda,q} }(\bx) \int_\Omega \hat{\rho}_{\delta,\beta-p}^\veps [\lambda,q](\bx,\by)  \frac{ \Phi_p' \left(  \grad u(\bx) \cdot \frac{\bx-\by}{|\bx-\by|}  \right) }{|\bx-\by|} \, \rmd \by = \cD_{p,\delta}^{1,\veps}[u](\bx)\,.
        \end{split}
    \end{equation}

    First, applying the coordinate change $\bz = \frac{\by-\bx}{\eta_\delta(\bx)}$ gives
    \begin{equation}\label{eq:PointwiseOperator:OddInt}
        \mathds{1}_{ \Omega^{\veps;\lambda,q} }(\bx) \int_\Omega \rho_{\delta,\beta-p}^\veps [\lambda,q](\bx,\by)  \frac{            \Phi_p' \left( \grad u(\bx) \cdot \frac{\bx-\by}{|\bx-\by|} \right) }{|\bx-\by|} \, \rmd \by = 0\,,
    \end{equation}
    since $\Phi_p'$ is odd. Next, using a suitable addition/subtraction and applying \Cref{lma:CoordChange1}
    \begin{equation*}
        \begin{split}
            &\mathds{1}_{ \Omega^{\veps;\lambda,q} }(\bx) \int_\Omega \rho_{\delta,\beta-p}^\veps [\lambda,q](\by,\bx)  \frac{            \Phi_p' \left( \grad u(\bx) \cdot \frac{\bx-\by}{|\bx-\by|} \right) }{|\bx-\by|} \, \rmd \by \\
            &= \frac{ - \mathds{1}_{ \Omega^{\veps;\lambda,q} }(\bx) }{ \eta_\delta(\bx) } \int_\Omega \mathds{1}_{ \{ \veps < |\bsupsilon_\bx^\delta(\by)| < 1 \} }  \frac{ \rho( |\bsupsilon_\bx^\delta(\by)| ) }{ |\bsupsilon_\bx^\delta(\by)|^{\beta-p} }  \frac{\Phi_p' \left( \grad u(\bx) \cdot \frac{\bsupsilon_\bx^\delta(\by)}{|\bsupsilon_\bx^\delta(\by)|} \right) }{|\bsupsilon_\bx^\delta(\by)|} \det \grad \bsupsilon_\bx^\delta(\by) \, \rmd \by \\
                &\quad - \frac{ \mathds{1}_{ \Omega^{\veps;\lambda,q} }(\bx) }{ \eta_\delta(\bx) } \int_\Omega \mathds{1}_{ \{ \veps < |\bsupsilon_\bx^\delta(\by)| < 1 \} }  \frac{ \rho( |\bsupsilon_\bx^\delta(\by)| ) }{ |\bsupsilon_\bx^\delta(\by)|^{\beta-p} }  \frac{\Phi_p' \left( \grad u(\bx) \cdot \frac{\bsupsilon_\bx^\delta(\by)}{|\bsupsilon_\bx^\delta(\by)|} \right) }{|\bsupsilon_\bx^\delta(\by)|} \\
                &\qquad \qquad \frac{ \eta_\delta(\bx) - \eta_\delta(\by) - \grad \eta_\delta(\by) \cdot (\bx-\by) }{ \eta_\delta(\by)^{d+1} } \, \rmd \by \\
            &= 0 - \mathds{1}_{ \Omega^{\veps;\lambda,q} }(\bx)  \int_\Omega \mathds{1}_{ \{ \veps < |\bsupsilon_\bx^\delta(\by)| < 1 \} }  \frac{ \rho( |\bsupsilon_\bx^\delta(\by)| ) }{ |\bsupsilon_\bx^\delta(\by)|^{\beta-p} }  \Phi_p' \left( \grad u(\bx) \cdot \frac{\bsupsilon_\bx^\delta(\by)}{|\bsupsilon_\bx^\delta(\by)|} \right) \\
                &\qquad \qquad \frac{ \eta_\delta(\bx) - \eta_\delta(\by) - \grad \eta_\delta(\by) \cdot (\bx-\by) }{ |\bx-\by| \eta_\delta(\bx) \eta_\delta(\by)^{d} } \, \rmd \by = \cD_{p,\delta}^{1,\veps}[u](\bx)\,,
        \end{split}
    \end{equation*}
    again using that $\Phi_p'$ is odd. 
\end{proof}

\begin{corollary}\label{cor:PointwiseOperator:AbsConv}
    Let $u \in C^2(\overline{\Omega})$. Assume \eqref{assump:FullLocalization:C11}, and assume further that $\beta < d + p - 1$. Then, for almost every $\bx \in \Omega$,
    \begin{equation*}
        \wt{\cL}_{p,\delta} u(\bx) = \cL_{p,\delta} u(\bx)\,,
    \end{equation*}
    where $\cL_{p,\delta} u(\bx) \in L^1_{loc}(\Omega)$ is defined via the absolutely convergent integral
    \begin{equation}\label{eq:PointwiseOperator:Defn}
        \cL_{p,\delta} u(\bx) := \int_{\Omega} \frac{\hat{\rho}_{\delta,\beta-p}(\bx,\by)
        }{|\bx-\by|}
        \Phi_p' \left( \frac{u(\bx)-u(\by)}{|\bx-\by|} \right)  \, \rmd \by\,, \quad \forall \, \bx \in \Omega\,.
    \end{equation}
    Moreover, if \eqref{assump:Localization:NormalDeriv} and \eqref{assump:MomentsOfNonlinLoc} are satisfied, then $\cL_{p,\delta} u = (\mathrm{d}\text{-}\cL)_{p,\delta} u$ in the space $[\mathfrak{W}^{\beta,p}[\delta;q](\Omega)]^*$, i.e.  \eqref{eq:GreensIdentity} holds for all $u \in C^2(\overline{\Omega})$ and all $v \in \mathfrak{W}^{\beta,p}[\delta;q](\Omega)$ with $(\mathrm{d}\text{-}\cL)_{p,\delta} u$ replaced with $\cL_{p,\delta} u$. 
\end{corollary}

\begin{proof}
    Clearly the integrand of $\cL_{p,\delta} u$ is bounded by $\frac{C}{ \eta_\delta(\bx)}\Vnorm{u}_{W^{1,\infty}(\Omega)}$.
    Now, for any fixed $\bx \in \Omega$, the integral in \eqref{eq:PointwiseOperator:OddInt} is absolutely convergent as $\veps \to 0$ since $\beta < d+p-1$; the integrand is bounded by $\frac{C}{\eta_\delta(\bx)}$. Thus all of the calculations in the previous proof can be redone after taking the limit $\veps \to 0$, showing the equality and concluding that $\cL_{p,\delta} u$ has all of the same properties as $\wt{\cL}_{p,\delta} u$.
\end{proof}

In fact, when the singularity of $\cL_{p,\delta}$ only occurs at the boundary, it is possible to show that $\cL_{p,\delta}$ maps $\mathfrak{W}^{\beta,p}[\delta;q](\Omega)$ to $L^{p'}_{loc}(\Omega) \cap [\mathfrak{W}^{\beta,p}_{0,\p \Omega}[\delta;q](\Omega)]^*$, analogous to the fact that $-\Delta : W^{1,p}(\Omega) \to [W^{1,p}_{0,\p \Omega}(\Omega)]^*$, the difference being that the nonlocal operator has no differentiability action in the interior of $\Omega$.

\begin{theorem}\label{thm:PointwiseOperator:AbsConv:Energy}
  Assume \eqref{assump:FullLocalization:C11}, and assume further that $\beta < d$. Then,     $\cL_{p,\delta} u(\bx) \in L^{p'}_{loc}(\Omega)$ is defined for  $u \in \mathfrak{W}^{\beta,p}[\delta;q](\Omega)$ via the absolutely convergent integral in \eqref{eq:PointwiseOperator:Defn}, and there exists a constant $C$ depending only on $p$, $\beta$, $\rho$, $\lambda$, $q$, $\kappa_0$ and $\kappa_1$ such that
    \begin{equation}\label{eq:PointwiseOperator:EnergyEst}
        \Vnorm{ \eta_\delta \cL_{p,\delta} u }_{L^{p'}(\Omega)} \leq C [u]_{ \mathfrak{W}^{\beta,p}[\delta;q](\Omega) }\,,\quad \forall u \in \mathfrak{W}^{\beta,p}[\delta;q](\Omega).
    \end{equation}
    Moreover, 
    if \eqref{assump:Localization:NormalDeriv} and \eqref{assump:MomentsOfNonlinLoc} are satisfied, then 
\begin{equation}\label{eq:GreensIdentity:StrongL}
        \int_{\Omega} v \cL_{p,\delta} u \, \rmd \bx = \cB_{p,\delta}(u,v)\,, \quad \forall 
        \,u \in \mathfrak{W}^{\beta,p}[\delta;q](\Omega) \text{ and } v \in \mathfrak{W}^{\beta,p}_{0,\p \Omega}[\delta;q](\Omega)\,,
    \end{equation}
    i.e. \eqref{eq:GreensIdentity} holds for all $v \in \mathfrak{W}^{\beta,p}_{0,\p \Omega}[\delta;q](\Omega)$ and all $u \in \mathfrak{W}^{\beta,p}[\delta;q](\Omega)$ with $(\mathrm{d}\text{-}\cL)_{p,\delta} u$ replaced by $\cL_{p,\delta} u$.
\end{theorem}

\begin{proof} 
    The integrand in $\cL_{p,\delta} u(\bx)$ is bounded by $\frac{C \hat{\gamma}_{\beta,0}[\delta;q](\bx,\by) }{\eta_\delta(\bx)^p} ( |u(\bx)|^p +|u(\by)|^p)$,
    hence the integral is absolutely convergent.
    Next, by \eqref{eq:ComparabilityOfDistanceFxn2},
    \begin{equation*}
        \begin{split}
            \Vnorm{ \eta_\delta \cL_{p,\delta} u }_{L^{p'}(\Omega)}^{p'} 
            &\leq C \int_{\Omega} \left( \int_{\Omega} \hat{\rho}_{\delta,\beta-p+1}(\bx,\by) \frac{|u(\bx)-u(\by)|^{p-1}}{ |\bx-\by|^{p-1} } \, \rmd \by \right)^{\frac{p}{p-1}} \, \rmd \bx \\
            &\leq C \int_{\Omega} \left( \int_{\Omega} \hat{\rho}_{\delta,\beta}(\bx,\by) \frac{|u(\bx)-u(\by)|^{p-1}}{ \eta_\delta(\bx)^{p-1} } \, \rmd \by \right)^{\frac{p}{p-1}} \, \rmd \bx\,.
        \end{split}
    \end{equation*}
    Then by H\"older's inequality
    \begin{equation*}
        \Vnorm{ \eta_\delta \cL_{p,\delta} u }_{L^{p'}(\Omega)}^{p'} \leq C \left( \int_{\Omega} \hat{\rho}_{\delta,\beta}(\bx,\by) \, \rmd \by \right)^{1/p} \int_{\Omega} \int_{\Omega} \hat{\rho}_{\delta,\beta}(\bx,\by) \frac{|u(\bx)-u(\by)|^{p}}{ \eta_\delta(\bx)^{p} } \, \rmd \by \, \rmd \bx\,,
    \end{equation*}
    and so \eqref{eq:PointwiseOperator:EnergyEst} follows by \eqref{thm:EnergySpaceIndepOfKernel}. 

    Last, by density we just need to establish \eqref{eq:GreensIdentity:StrongL} for $v \in C^\infty_c(\Omega)$. In this case, we can use nonlocal integration-by-parts; since $v$ is compactly supported and $\cL_{p,\delta} u \in L^1_{loc}(\Omega)$, the singularity of $\frac{1}{\eta_\delta(\bx)^p}$ is not seen in the integrals (see \Cref{lma:SupportOfConv} below for a similar result), and so the required linearity of integrals and Fubini-Tonelli theorems can be used.
\end{proof}

\subsection{The linear operator}

In the previous section we showed that the operator can be defined pointwise via \eqref{eq:Intro:Operator} for smooth functions, so long as the kernel has no singularity on the diagonal. It is desirable, in view of the Sobolev estimates of $\cL_{p,\delta}$, to extend its action via density to a wider class of measurable functions. 
We will be content to do so for the case that $\cL_{p,\delta}$ is linear, and leave the nonlinear operator for future works.
We show that the domain of the distributional form of the nonlocal operator can be extended to Sobolev functions, analogous to the definition of the classical weak derivative. This approach also has the advantage in that it works for a wider range of $\beta$, thus permitting singularity on the diagonal. Moreover, we can conclude a nonlocal Green's second identity.

For this subsection, we assume that $p = 2$ with $\Phi_2(t) = \frac{t^2}{2}$, along with the assumptions $\beta \in [0,d+1)$, \eqref{assump:Localization}, \eqref{assump:NonlinearLocalization}, \eqref{assump:VarProb:Kernel}, \eqref{assump:FullLocalization:C11}, and \eqref{eq:HorizonThreshold}.
Take $\bar{\rho}_{2-\beta} = \overline{C}_{d,2}$, and recall the abbreviations $\cL_{2,\delta}u = \cL_\delta u$ and $\cB_{2,\delta}(u,v) = \cB_{\delta}(u,v)$ from \Cref{subsec:Examples}.

Let $u \in L^1_{loc}(\Omega)$.
We say that the weak form of $\cL_{\delta} u$ exists if there exists a function $u_0 \in L^1_{loc}(\Omega)$ such that
\begin{equation*}
    \int_{\Omega} u_0 v \, \rmd \bx = \int_{\Omega} u \cL_{\delta} v \, \rmd \bx\,, \quad \forall \, v \in C^{\infty}_c(\Omega)\,,
\end{equation*}
where $\cL_{\delta}$ is defined in \eqref{eq:PointwiseOperator:Defn}. Clearly, if such a function $u_0$ exists, it is unique, and we denote it by $u_0 =
(\mathrm{w}\text{-}\cL)_{\delta}u$.

\begin{theorem}\label{thm:WeakOperator}
    Let $u \in W^{2,r}(\Omega)$ for some $r \in [1,\infty]$. Then $(\mathrm{w}\text{-}\cL)_{\delta} u$ exists and belongs to $L^r(\Omega)$, with $\vnorm{(\mathrm{w}\text{-}\cL)_{\delta} u}_{L^r(\Omega)} \leq C \vnorm{u}_{W^{2,r}(\Omega)}$ where $C = C(\rho,\beta,p,r)$.
    Moreover, if \eqref{assump:Localization:NormalDeriv} and \eqref{assump:MomentsOfNonlinLoc} are satisfied, then $(\mathrm{w}\text{-}\cL)_{\delta} u = (\mathrm{d}\text{-}\cL)_{2,\delta} u$ in the space $[\mathfrak{W}^{\beta,2}[\delta;q](\Omega)]^*$, i.e. \eqref{eq:GreensIdentity} holds for all $v \in \mathfrak{W}^{\beta,p}[\delta;q](\Omega)$ and all with $u \in W^{2,r}(\Omega)$ with  $(\mathrm{d}\text{-}\cL)_{2,\delta} u$ replaced with $(\mathrm{w}\text{-}\cL)_{\delta} u$. We further have
    \begin{equation*}
        \int_{\Omega} v (\mathrm{w}\text{-}\cL)_\delta u - u (\mathrm{w}\text{-}\cL)_{\delta} v \, \rmd \bx = A_\delta^N \int_{\p \Omega} u \frac{\p v}{\p \bsnu} - v \frac{\p u}{\p \bsnu} \, \rmd \sigma\,, \quad \forall u \in W^{2,r}(\Omega), \, v \in W^{2,r'}(\Omega).
    \end{equation*}
\end{theorem}

\begin{remark}
    If additionally $\beta < d$ it can be shown using the continuity estimates that $\cL_\delta u$ and $(\mathrm{w}\text{-}\cL)_\delta u$ coincide for any $u \in C^2(\overline{\Omega})$. This implies that, in this case, the domain of the operator $\cL_{\delta}$ can be extended to $\in W^{2,r}(\Omega)$ for any $r \in [1,\infty)$, and maps $W^{2,r}(\Omega)$ to $L^r(\Omega)$.
\end{remark}

\begin{proof}
    Let $\{u_n\}_n \subset C^{\infty}(\overline{\Omega})$ converge to $u$ in $W^{2,r}(\Omega)$. Then by \Cref{cor:PointwiseOperator:AbsConv} and \Cref{cor:PointwiseOperator:Summary}, 
    $        
    \Vnorm{ \cL_{\delta} u_n - \cL_{\delta} u_m }_{L^r(\Omega)} \leq C \Vnorm{ u_n - u_m }_{W^{2,r}(\Omega)} \to 0 \text{ as } n,m \to \infty\,,
    $
    since $\cL_\delta$ is linear.
    Therefore the sequence $\{ \cL_{\delta} u_n \}_n$ converges in $L^r(\Omega)$ to a function $u_0$. To see that $u_0 = (\mathrm{w}\text{-}\cL)_{\delta} u$, let $v \in C^{\infty}_c(\Omega)$ be arbitrary, and then
    \begin{equation*}
        \int_{\Omega} u_0 v  \, \rmd \bx 
        = \lim\limits_{n \to \infty} \int_\Omega \cL_{\delta} u_n \, v\, \rmd \bx 
        =\lim\limits_{n \to \infty} \int_\Omega u_n \cL_{\delta}  v \, \rmd \bx 
        = \int_{\Omega } u \cL_{\delta} v\, \rmd \bx\,,
    \end{equation*}
    where in the second equality we are guaranteed the validity of \eqref{eq:GreensIdentity} for $u_n \in C^{\infty}(\overline{\Omega})$ and $v \in C^{\infty}_c(\Omega)$ thanks to \Cref{cor:PointwiseOperator:AbsConv}.
    Finally, 
    \begin{align*}
  &  \vnorm{ (\mathrm{w}\text{-}\cL)_{\delta} u }_{L^r(\Omega)} 
    \leq \liminf_{n \to \infty} \vnorm{ (\mathrm{w}\text{-}\cL)_{\delta} u_n }_{L^r(\Omega)} \\
  &\qquad  \leq C \liminf_{n \to \infty} \vnorm{ u_n }_{W^{2,r}(\Omega)}
    = C \Vnorm{u}_{W^{2,r}(\Omega)}.
        \end{align*}
    Last, $(\mathrm{w}\text{-}\cL)_{\delta} u$ can replace $(\mathrm{d}\text{-}\cL)_{2,\delta} u$ in \eqref{eq:GreensIdentity} since the right-hand side holds for all $u_n$ in any smooth $W^{2,r}(\Omega)$-approximating sequence, and since both left-hand side terms are stable under strong $W^{2,r}(\Omega)$ convergence for any fixed $v \in \mathfrak{W}^{\beta,p}[\delta;q](\Omega)$.
\end{proof}

\subsection{The variational form}
Here, we take all of the assumptions of the previous subsection, along with \eqref{assump:Localization:NormalDeriv} and \eqref{assump:MomentsOfNonlinLoc}.
We first note that the pointwise definition of our nonlocal operator can be used to obtain a kind of {\it nonlocal}
weak convergence in the linear case.
This, along with the subsequent theorem, will be central in showing the consistency of solutions as the bulk horizon parameter $\delta$ approaches $0$.

\begin{theorem}\label{thm:BilinearFormLocalization}
    Suppose that $v \in W^{1,2}(\Omega)$, and suppose that $\{ u_{\delta} \}_\delta \subset W^{1,2}(\Omega)$ and $u \in W^{1,2}(\Omega)$ satisfy $u_{\delta} \rightharpoonup u$ weakly in $W^{1,2}(\Omega)$ as $\delta \to 0$. Then
	\begin{equation*}
		\lim\limits_{\delta \to 0} \cB_{\delta}(u_{\delta} - u,v) = 0\,.
	\end{equation*}
\end{theorem}

\begin{proof}
	Without loss of generality, we assume $u_{\delta} \rightharpoonup 0$ in $W^{1,2}(\Omega)$ as $\delta \to 0$.
    First assume that $v \in C^2(\overline{\Omega})$. Then, since $W^{1,2}(\Omega) \subset \mathfrak{W}^{\beta,p}[\delta;q](\Omega)$, we can use \Cref{cor:PointwiseOperator:AbsConv} and apply the nonlocal Green's identity \eqref{eq:GreensIdentity}:
    \begin{equation*}
         \cB_{\delta}(u_{\delta},v) = \int_{\Omega} \cL_{\delta} v(\bx) u_{\delta}(\bx) \, \rmd \bx  + \int_{\p \Omega} BF_{2,\delta}^N( \grad v, \bsnu) T u_\delta \, \rmd \sigma\,.
    \end{equation*}
    By \Cref{cor:PointwiseOperator:AbsConv} and \Cref{cor:PointwiseOperator:Summary}, and the choice of normalization constants, $\cL_{\delta} v \to -\Delta v$ strongly in $L^2(\Omega)$ as $\delta \to 0$, and by the compact embedding of $W^{1,2}(\Omega)$ into $L^2(\Omega)$ we have that $u_{\delta} \to 0$ strongly in $L^2(\Omega)$. So, 
    the first term on the right-hand side converges to $0$ as $\delta \to 0$. The second term on the right-hand side converges to $0$ as well; since $T : W^{1,2}(\Omega) \to W^{\frac{1}{2},2}(\p \Omega)$ is weakly continuous, we have that $T u_\delta \to 0 $ strongly in $L^2(\p \Omega)$ (note that the term $BF_{2,\delta}^N(\grad v,\bsnu)$ is bounded uniformly with respect to $\delta$).
    Therefore, $\cB_{\delta}(u_{\delta},v) \to 0$ as $\delta \to 0$ for all $v \in C^2(\overline{\Omega})$.
    Now let $v \in W^{1,2}(\Omega)$, and let $\{v_n \}_n \subset C^2(\overline{\Omega})$ converge to $v$ in $W^{1,2}(\Omega)$. Then by H\"older's inequality and \eqref{eq:Embedding}
    \begin{equation*}
    \begin{split}
        |\cB_{\delta}(u_{\delta},v)| &\leq |\cB_{\delta}(u_{\delta},v_n-v)| + |\cB_{\delta}(u_{\delta},v_n)| \\
        &\leq C\Vnorm{v_n -v}_{W^{1,2}(\Omega)} \Vnorm{u_{\delta}}_{W^{1,2}(\Omega)} + |\cB_{\delta}(u_{\delta},v_n)|\,.
    \end{split}
    \end{equation*}
    Since $\Vnorm{u_{\delta}}_{W^{1,2}(\Omega)}$ is bounded uniformly with respect to $\delta$, we can choose $n$ large enough so that the first term is arbitrarily small independent of $\delta$. Then we use the first part of the proof to let $\delta \to 0$ in the second term to get the conclusion.
\end{proof}

\begin{theorem}\label{thm:BilinearFormLocalization2}
    For all $u$ and $v \in W^{1,2}(\Omega)$, $\lim\limits_{\delta \to 0} \cB_{\delta}(u,v) = \cB_{0}(u,v)$.
\end{theorem}

\begin{proof} 
    The proof follows exactly the same steps as \cite[Theorem 1.1]{ponce2004new}, additionally using the polarization identity; the heterogeneous localization $\eta(\bx)$ gives no additional difficulty.
\end{proof}

\section{Boundary-localized convolutions and their adjoint operators}\label{sec:BdyLocalizedConv}
Our discussion in this section, unless indicated otherwise, is
under the assumptions \eqref{assump:beta}, 
\eqref{assump:NonlinearLocalization} for $k = k_q$, \eqref{assump:Localization} for $k = k_\lambda$,
\eqref{Assump:Kernel} for $k = k_\psi$, and \eqref{eq:HorizonThreshold}.
In addition, for any statement involving the  first order differentiation of the boundary-localized convolutions or their adjoint operators, we also assume that \eqref{Assump:Kernel} is satisfied for $k_\psi \geq 1$.

\subsection{Known results for the boundary-localized convolutions}
We first review properties of boundary-localized convolutions $K_\delta$ that were proved in \cite{scott2023nonlocal} that we will need.

\begin{theorem}[\cite{scott2023nonlocal}]\label{thm:KnownConvResults}
    The following hold:
    \begin{enumerate}
        \item If $u \in L^1_{loc}(\Omega)$, then $K_{\delta} u \in C^{k}(\Omega)$, where $k = \min \{ k_q, k_\lambda, k_\psi \}$.
        
        \item If $u \in C^0(\overline{\Omega})$,  then $K_{\delta} u \in  C^0(\overline{\Omega})$. Moreover, $K_{\delta} u(\bx) = u (\bx)$ for all $\bx \in \p \Omega$, and $K_{\delta} u \to u$ uniformly on $\overline{\Omega}$ as $\delta \to 0$.
        
        \item There exists a constant $C_0 = C_0(d,p,\psi,\kappa_1) > 0$ such that
        \begin{equation}\label{eq:ConvEst:Lp}
		\Vnorm{K_{\delta} u}_{L^p(\Omega)} \leq C_0 \Vnorm{u}_{L^p(\Omega)}\,, \quad \forall u \in L^p(\Omega) \text{ with } 1 \leq p \leq \infty
	\end{equation}
        and in fact
        \begin{equation}\label{eq:ConvergenceOfConv}
		\lim\limits_{\delta \to 0} \Vnorm{K_{\delta} u - u}_{L^p(\Omega)} = 0\,, \quad \forall u \in L^p(\Omega) \text{ with } 1 \leq p < \infty\,.
        \end{equation}

        \item Let $\wt{K}_\delta$ be the operator defined for any $\bv : \Omega \to \bbR^d$ and $\bx\in\Omega$ as
       \begin{equation}\label{eq:AuxOperatorDefn}
            \wt{K}_{\delta} \bv(\bx) := \int_{\Omega}   \psi_{\delta}(\bx,\by) \left[ \bI - \frac{(\bx-\by) \otimes \grad \eta_\delta(\bx)}{\eta_\delta(\bx)} \right] \bv(\by) \, \rmd \by\,.
        \end{equation}
        Then there exists a constant $C_1= C_1(d,p,\psi,\kappa_1) > 0$ such that
\begin{align}
\label{eq:ConvEst1:W1p:Pf1}
    &\grad K_\delta u(\bx) = \wt{K}_\delta[\grad u](\bx)\,, \quad  \forall \, u \in W^{1,p}(\Omega),\,  \bx \in \Omega\,,\\
    \label{eq:ConvEst:W1p}
    &\Vnorm{\grad K_{\delta} u}_{L^p(\Omega)} \leq C_1 \Vnorm{\grad u}_{L^{p}(\Omega)}\,, \;\; \forall u \in W^{1,p}(\Omega) \text{ with } 1 \leq p \leq \infty\,,\\
    \label{eq:ConvergenceOfConv:W1p}	  &\lim\limits_{\delta \to 0} \Vnorm{K_{\delta} u - u}_{W^{1,p}(\Omega)} = 0\,, \quad \forall u \in W^{1,p}(\Omega) \text{ with } 1 \leq p < \infty\,.
\end{align}
        \item For $1<p<\infty$, there 
        exists
        a constant $C_2
        = C_2(d,p,\beta,\psi,q,\kappa_0,\kappa_1)$
        such that   
     \begin{equation}\label{eq:Intro:ConvEst:Deriv}
		\begin{split}
			\Vnorm{ \grad K_{\delta} u }_{L^{p}(\Omega)} 
			\leq C_2 [u]_{\mathfrak{W}^{\beta,p}[\delta;q](\Omega)},
            \;\; \forall u \in \mathfrak{W}^{\beta,p}[\delta;q](\Omega).
		\end{split}
        \end{equation}
   \item     Additionally, let $1 < p < \infty$ and denote the trace operator $T : \mathfrak{W}^{\beta,p}[\delta;q](\Omega) \to W^{1-1/p,p}(\p \Omega)$. Then $T K_{\delta} u = T u$ for all $u \in W^{1,p}(\Omega)$.
    \end{enumerate}
\end{theorem}

As a consequence of the above, we
note that the boundary-localized convolution can be used to construct a generalized distance function satisfying \eqref{assump:Localization:NormalDeriv}.

\begin{lemma}\label{lma:BdyLocalizedDistFxn}
    For any Lipschitz domain $\Omega \subset \bbR^d$, there exists a generalized distance function $\bar{\lambda}$ that satisfies \eqref{assump:Localization} with $k_\lambda = \infty$ and \eqref{assump:Localization:NormalDeriv}.
\end{lemma}

\begin{proof}
    It is straightforward to verify, using the techniques in this section and in \cite{scott2023nonlocal}, that the lemma is satisfied by $\bar{\lambda}(\bx) = K_{\veps}[\lambda,q,\psi] d_{\p \Omega}(\bx)$, where $\lambda$ is a generalized distance function of the type constructed in \cite{Stein}.
\end{proof}

\subsection{Properties of adjoint operators and estimates in function spaces}

\begin{theorem}\label{thm:ConvProp:Adj:UniformCase}
    Let $u \in C^0(\Omega)$. Then $K_{\delta}^* u \in C^0(\Omega)$. Moreover, 
    $K_{\delta}^* u \to u$ in $L^\infty(\Omega)$  as $\delta \to 0$.
\end{theorem}

\begin{proof}
	Clearly $K_\delta^* u \in C^0(\Omega)$.
	For the $L^\infty(\Omega)$-convergence, we have for any $\bx \in \Omega$
	\begin{equation*}
		\begin{split}
			|K_{\delta}^*u(\bx) - u(\bx)|
			&= \left| \int_{\Omega} \psi_{\delta}(\by,\bx) (u(\by)-u(\bx)) \, \rmd \by \right| + \left| \Psi_{\delta}(\bx) - 1 \right| |u(\bx)|\\
			&\leq \sup_{ \bz \in B(\bx, \frac{\eta_\delta(\bx)}{1-\kappa_1 \delta}  ) } |u(\bz) - u(\bx)| \int_{\Omega} \psi_{\delta}(\by,\bx) \, \rmd \by + \Vnorm{ \Psi_\delta - 1}_{L^\infty(\Omega)} |u(\bx)| \\
			&= \Psi_{\delta}(\bx) \sup_{ \bz \in B(\bx, \frac{3}{2} \eta_\delta(\bx)) } |u(\bz) - u(\bx)| + \Vnorm{ \Psi_\delta - 1}_{L^\infty(\Omega)} |u(\bx)|\,.
		\end{split}
	\end{equation*}
	Both quantities converge to $0$ as $\delta \to 0$ independently of $\bx \in \Omega$; the first since $u$ is uniformly continuous on $\overline{\Omega}$, and the second by the estimate \eqref{eq:KernelIntFunction:Bounds}.
 \end{proof}

\begin{lemma}\label{lma:SupportOfConv}
    Suppose that $v \in C^0(\overline{\Omega})$ has compact support in $\Omega$, i.e. there exists $c_v > 0$ such that
		$\supp v \subset \Omega^{c_v;\eta,q}\,.$
	Then both $K_{\delta} v$ and $K_\delta^* v$ have compact support with
    $\supp K_{\delta} v \subset \Omega^{ \frac{1}{1+ \kappa_1 \delta} c_{v}; \eta,q}$ and $\supp K_{\delta}^* v \subset \Omega^{ (1 - \kappa_1 \delta) c_{v}; \eta,q}$.
\end{lemma}

\begin{proof}
    The result for $K_\delta$ is in \cite{scott2023nonlocal}. By \eqref{eq:ComparabilityOfDistanceFxn2}, whenever $h(\bx) < ( 1- \kappa_1 \delta  ) c_v $
	we have $  \{ \by \, : |\bx-\by| \leq \eta_\delta(\by) \} \subset \Omega_{c_v;\eta,q} $.
    Therefore, since the domain of integration in $K_{\delta}^*v(\bx)$ and $\supp v$ are disjoint, we have $K_\delta^* v(\bx) = 0$. 
\end{proof}

\begin{theorem}\label{thm:ConvEst1:Adj}
    There exists a constant $C_0
   =C_0(d, p, \psi,\kappa_1)>0$ such that   
\begin{eqnarray}
& \label{eq:ConvEst:Adj:Lp}
		\Vnorm{K_{\delta}^* u}_{L^p(\Omega)} \leq C_0 \Vnorm{u}_{L^p(\Omega)} &\quad \forall u \in L^p(\Omega) \text{ with } 1 \leq p \leq \infty\,.\\
\label{eq:ConvergenceOfConvAdj}
&\lim\limits_{\delta \to 0} \Vnorm{K_{\delta}^* u - u}_{L^p(\Omega)} = 0\,, &\quad \forall u \in L^p(\Omega) \text{ with } 1 \leq p < \infty\,.
\end{eqnarray}
\end{theorem}

\begin{proof}
	First we prove \eqref{eq:ConvEst:Adj:Lp} for $1 \leq p < \infty$. 
	By H\"older's inequality, \eqref{eq:KernelIntFunction:Bounds}, and Tonelli's theorem, we have
	\begin{equation*}
		\Vnorm{ K_\delta^* u }_{L^p(\Omega)}^p \leq \int_{\Omega} \left( \int_{\Omega} \psi_\delta(\bz,\bx) \, \rmd \bz \right)^{p-1} \int_{\Omega} \psi_\delta(\by,\bx) |u(\by)|^p \, \rmd \by \, \rmd \bx \leq C \Vnorm{u}_{L^p(\Omega)}^p\,.
	\end{equation*}
    The inequality when $p = \infty$ is trivial.
    The result \eqref{eq:ConvergenceOfConvAdj} follows from the density of $C^\infty(\overline{\Omega})$ in $L^p(\Omega)$, from the $L^p(\Omega)$-continuity of $K_\delta$ contained in \eqref{eq:ConvEst:Adj:Lp}, and from the $L^\infty(\Omega)$-convergence of $K_\delta u$ to $u$ in \Cref{thm:ConvProp:Adj:UniformCase}.
\end{proof}

For $1 < p < \infty$, the action of $K_\delta^*$ on $f \in [W^{1,p}(\Omega)]^*$ is defined as
\begin{equation*}
    \vint{K_\delta^* f, v} := \vint{f, K_\delta v}\,, \quad \forall v \in W^{1,p}(\Omega)\,,
\end{equation*}
with the same definition holding for $f \in [W^{1,p}_{0,\p \Omega_D}(\Omega)]^*$ and $v \in W^{1,p}_{0,\p \Omega_D}(\Omega)$.
This definition is consistent with the identity $\int_{\Omega} K_\delta^* f v \, \rmd \bx = \int_{\Omega} f K_\delta v \, \rmd \bx$ for any $f \in L^{p'}(\Omega)$.

\begin{theorem}
    Let $1 < p < \infty$.
    There exist positive constants  $C_1(d,p,\psi,\kappa_1)$ and $C_2(d,\beta,p,\psi,q,\kappa_0,\kappa_1)$ such that
   \begin{eqnarray}
       &&\qquad
		\Vnorm{K_{\delta}^* f}_{[W^{1,p}(\Omega)]^*} \leq C_1 \Vnorm{f}_{[W^{1,p}(\Omega)]^*}, \quad \forall f \in [W^{1,p}(\Omega)]^*,\label{eq:ConvEst:Adj:WMinus1p}   \\    
   &&  \qquad   \Vnorm{K_{\delta}^* f}_{[\mathfrak{W}^{\beta,p}[\delta;q](\Omega)]^*} \leq C_2 \Vnorm{f}_{[W^{1,p}(\Omega)]^*}, 
        \quad \forall f \in [W^{1,p}(\Omega)]^*,
        \label{eq:ConvEst:Adj:Nonlocal}        \\
 &&    \qquad   \lim\limits_{\delta \to 0} \vint{K_{\delta}^* f - f, v} = 0\,, \quad \forall \, v \in W^{1,p}(\Omega),
        \; f \in [W^{1,p}(\Omega)]^*. \label{eq:ConvergenceOfConv:Adj:Wminus1p}
         \end{eqnarray}
     The same results hold for $[W^{1,p}(\Omega)]^*$ replaced with the space $[W^{1,p}_{0,\p \Omega_D}(\Omega)]^*$.
\end{theorem}

\begin{proof}
        \eqref{eq:ConvEst:Adj:WMinus1p} and \eqref{eq:ConvEst:Adj:Nonlocal} follow from their respective dual statements \eqref{eq:ConvEst:W1p} and \eqref{eq:Intro:ConvEst:Deriv}, and 
        \eqref{eq:ConvergenceOfConv:Adj:Wminus1p} follows from its dual statement \eqref{eq:ConvergenceOfConv:W1p}.
\end{proof}

\subsection{Improved Sobolev estimates of the adjoint operators}

We recall that any element $f \in [W^{1,p}(\Omega)]^*$ is represented uniquely by a pair $(f_0,\bff_1) \in L^{p'}(\Omega) \times L^{p'}(\Omega;\bbR^d)$ such that
\begin{equation}\label{eq:DualSpace:Rep}
    \vint{f,v} = \int_{\Omega} f_0(\bx) v(\bx) \, \rmd \bx + \int_{\Omega} \bff_1(\bx) \cdot \grad v(\bx) \, \rmd \bx\,, \quad \forall v \in W^{1,p}(\Omega)\,.
\end{equation}
See \cite[chapter 3]{A75}.

\begin{theorem} 
    There exists a constant $C_1
    =C_1(d, p, \psi, \kappa_1)
    >0$
    such that    \begin{equation}\label{eq:ConvEst2:Adj:0:0}
        \Vnorm{\eta_\delta \grad K_{\delta}^* u}_{L^p(\Omega)} \leq C_1 \Vnorm{u}_{L^{p}(\Omega)} \qquad \forall u \in L^{p}(\Omega)\,, \quad 1 \leq p \leq \infty\,.
    \end{equation}
\end{theorem}

\begin{proof}
We apply H\"older's inequality, and then use 
\eqref{eq:KernelIntegralDerivativeEstimate:Adj},
\eqref{eq:KernelDerivativeEstimate:Adj}, \eqref{eq:ComparabilityOfDistanceFxn2}, and Fubini's theorem:
    \begin{equation*}
    \begin{aligned}
    &\Vnorm{\eta_\delta \grad K_{\delta}^* u}_{L^p(\Omega)}^p \leq \int_{\Omega} 
    \eta_\delta(\bx)^p \left( \int_{\Omega} |\grad_{\bx}  \psi_{\delta}(\by,\bx) |\, \rmd \by \right)^{p-1}
    \int_{\Omega} |\grad_{\bx} \psi_{\delta}(\by,\bx)| \, |u(\by)|^p \, \rmd \by 
    \rmd \bx  \\
    &\qquad \leq C \int_{\Omega}
   \int_{\Omega} \frac{ \eta_\delta(\bx) (|\psi'|)_{\delta}(\by,\bx) }{ \eta_\delta(\by) } \, |u(\by)|^p \, \rmd \by 
   \rmd \bx
   \leq C \Vnorm{u}_{L^p(\Omega)}^p\,.
		\end{aligned}\
	\end{equation*}	
 
\end{proof}

Before we further estimate $K_\delta^*$, we prove some results on auxiliary convolution-type operators.

\begin{lemma}
    For $1 < p < \infty$ and $\bu \in L^p(\Omega;\bbR^d)$, define the operator $\wt{K}_\delta \bu$ as in \eqref{eq:AuxOperatorDefn}, and define its adjoint
    \begin{equation*}
        \begin{split}
            \wt{K}_{\delta}^* \bu(\bx) &:= \int_{\Omega} \psi_{\delta}(\by,\bx) \left[ \bI - \frac{(\by-\bx) \otimes \grad \eta_\delta(\by)}{\eta_\delta(\by)} \right] \bu(\by) \, \rmd \by\,.
        \end{split}
    \end{equation*}
    Then $\vint{ \wt{K}_{\delta} \bu, \bv }_{L^2(\Omega)} = \vint{ \wt{K}_{\delta}^* \bv, \bu }_{L^2(\Omega)}$ for any $\bu \in L^p(\Omega;\bbR^d)$ and $\bv \in L^{p'}(\Omega;\bbR^d)$, and there exists a constant $C_0$ depending only on $d$, $p$, $\psi$ and $\kappa_1$ such that 
    \begin{equation}\label{eq:UtilityConv:Est:Lp}
        \vnorm{ \wt{K}_{\delta} \bu  }_{L^p(\Omega)} + \vnorm{ \wt{K}_{\delta}^* \bu  }_{L^p(\Omega)} \leq C_0 \Vnorm{ \bu }_{L^p(\Omega)} \qquad \forall \bu \in L^p(\Omega;\bbR^d)\,.
    \end{equation}
    In addition, there exists a constant $C_1$ depending only on $d$, $p$, $\psi$ and $\kappa_1$ such that 
    \begin{equation}\label{eq:UtilityConv:Est:W1p}
        \vnorm{ \eta_\delta \grad \wt{K}_{\delta}^* \bu  }_{L^p(\Omega)} \leq C_0 \Vnorm{ \bu }_{L^p(\Omega)} \qquad \forall \bu \in L^p(\Omega;\bbR^d)\,.
    \end{equation}
    Moreover, if
    \eqref{assump:Localization} and \eqref{assump:NonlinearLocalization}
    are satisfied for $k \geq 2$, 
    then there exists a constant $C_1$ depending only on $d$, $p$, $\psi$ and $\kappa_1$ such that 
    \begin{equation}\label{eq:UtilityConv:Est:W1p:2}
        \vnorm{ \eta_\delta \grad \wt{K}_{\delta} \bu  }_{L^p(\Omega)} 
        \leq C_0 \Vnorm{ \bu }_{L^p(\Omega)} \qquad \forall \bu \in L^p(\Omega;\bbR^d)\,.
    \end{equation}
\end{lemma}

\begin{proof}
The estimate \eqref{eq:UtilityConv:Est:Lp} follows from \eqref{eq:ConvEst:Lp} and \eqref{eq:ConvEst:Adj:Lp}, and the bounds
\begin{equation*}
    |\wt{K}_\delta \bu(\bx)| \leq (1+\kappa_1) \left( K_\delta |\bu|(\bx) + K_\delta^* |\bu|(\bx) \right)\,.
\end{equation*}

We only prove \eqref{eq:UtilityConv:Est:W1p}; the proof for $\wt{K}_\delta$ is similar, and in any case only the estimate for $\wt{K}_\delta^*$ will be referenced elsewhere.
Directly,
\begin{equation*}
    \begin{split}
        \grad \wt{K}_{\delta}^* \bu(\bx) &= \int_{\Omega} \frac{ (\psi')_\delta(\by,\bx) }{ \eta_\delta(\by) } \frac{\bx-\by}{|\bx-\by|} \otimes \left( \bu(\by) + \frac{\grad \eta_\delta(\by) \cdot \bu(\by) }{ \eta_\delta(\by) } (\bx-\by) \right) \\
        &\qquad + \psi_\delta(\by,\bx) \frac{\grad \eta_\delta(\by) \cdot \bu(\by) }{ \eta_\delta(\by) } \bI \, \rmd \by\,,
    \end{split}
\end{equation*}
and so by \eqref{eq:ComparabilityOfDistanceFxn2}
\begin{equation*}
    \begin{split}
        \eta_\delta(\bx) |\grad \wt{K}_{\delta}^* \bu(\bx)| &\leq (1+\kappa_1 \delta)^2 \int_{\Omega}  \Big( (|\psi'|)_\delta(\by,\bx) + \psi_\delta(\by,\bx) \Big) |\bu(\by)| \, \rmd \by\,.
    \end{split}
\end{equation*}
The result then follows from the method used to prove \eqref{eq:ConvEst:Adj:Lp}.
\end{proof}

\begin{theorem}
    Let $1 < p < \infty$. 
    Then there exists a positive constant $C_0$ depending only on $d$, $p$, $\psi$ and $\kappa_1$ such that
    \begin{equation}\label{eq:ConvEst2:Adj:0:neg1}
		\Vnorm{\eta_\delta K_{\delta}^*f }_{L^{p}(\Omega)} \leq C_0 \Vnorm{f}_{[W^{1,p'}(\Omega)]^*}\,, \qquad \forall \, f \in [W^{1,p'}(\Omega)]^*\,.
    \end{equation}
    If additionally \eqref{Assump:Kernel}
    is satisfied for $k_{\psi} \geq 2$, then there exists a positive constant $C_1$ depending only on $d$, $p$, $\psi$ and $\kappa_1$ such that
    \begin{equation}\label{eq:ConvEst2:Adj:1:neg1}
        \Vnorm{\eta_\delta^2 K_{\delta}^* f}_{W^{1,p}(\Omega)} \leq C_1 \Vnorm{f}_{[W^{1,p'}(\Omega)]^*} \,, \qquad \forall \, f \in [W^{1,p'}(\Omega)]^*\,.
    \end{equation}
    The same results hold for $[W^{1,p'}(\Omega)]^*$ replaced with $[W^{1,p'}_{0,\p \Omega_D}(\Omega)]^*$
\end{theorem}

\begin{proof}
    We first prove \eqref{eq:ConvEst2:Adj:0:neg1}.
    Let $f$ be identified with the pair $(f_0,\bff_1)$ as in \eqref{eq:DualSpace:Rep}; we get for any $v \in W^{1,p'}(\Omega)$
    \begin{equation*}
    \begin{split}
        \vint{\eta_\delta K_\delta^*f,v} &=  \vint{ f, K_\delta (\eta_\delta v) } 
        = \int_{\Omega} f_0(\bx) K_\delta (\eta_\delta v)(\bx) \, \rmd \bx + \int_{\Omega} \bff_1(\bx) \cdot \grad [ K_\delta (\eta_\delta v) ](\bx) \, \rmd \bx\,.
    \end{split}
    \end{equation*}
    Then $\grad[K_\delta(\eta_\delta v)] = \wt{K}_\delta [ \grad (\eta_\delta v) ]$ by the formula \eqref{eq:ConvEst1:W1p:Pf1}, so
    \begin{equation*}
    \begin{split}
        \vint{\eta_\delta K_\delta^*f,v} 
        &= \int_{\Omega} K_\delta^* f_0(\bx) \eta_\delta(\bx) v(\bx) \, \rmd \bx + \int_{\Omega} \wt{K}_\delta^* \bff_1(\bx) \cdot \grad [ \eta_\delta v ](\bx) \, \rmd \bx \\
        &= \int_{\Omega} K_\delta^* f_0(\bx) \eta_\delta(\bx) v(\bx) \, \rmd \bx + \int_{\Omega} \wt{K}_\delta^* \bff_1(\bx) \cdot \grad \eta_\delta(\bx)  v(\bx) \, \rmd \bx \\
        &\qquad + \int_{\Omega} \eta_\delta(\bx) \wt{K}_\delta^* \bff_1(\bx) \cdot \grad v(\bx) \, \rmd \bx\,.
    \end{split}
    \end{equation*}
    To integrate by parts in the last term, we note that $\eta_\delta \wt{K}_\delta^* \bff_1 \in W^{1,p}(\Omega)$ by \eqref{eq:UtilityConv:Est:W1p}, and so by the divergence theorem and trace theorem
    \begin{equation*}
        \int_{\Omega} \eta_\delta \wt{K}_\delta^* \bff_1\cdot \grad v \, \rmd \bx = -\int_{\Omega} \div[ \eta_\delta \wt{K}_\delta^* \bff_1 ] v \, \rmd \bx + \int_{\p \Omega} T v \cdot T[\eta_\delta \wt{K}_\delta^* \bff_1] \cdot \bsnu \, \rmd \sigma\,.
    \end{equation*}
    Now, suppose $\bsvarphi \in C^1(\overline{\Omega};\bbR^d)$. Then an application of \eqref{eq:UtilityConv:Est:Lp} with $p = \infty$ shows that $\eta_\delta \wt{K}_\delta^* \bsvarphi(\bx) = {\bf 0}$ for all $\bx \in \p \Omega$.
    Now let $\bsvarphi_n \in C^1(\overline{\Omega};\bbR^d)$ be a sequence such that $\bsvarphi_n \to \bff_1$ in $L^{p}(\Omega)$. Then by  \eqref{eq:UtilityConv:Est:W1p}
	\begin{equation*}
		\begin{split}
			\vnorm{ T [\eta_\delta \wt{K}_\delta^* \bff_1]  }_{W^{1-1/p,p}(\p \Omega)} 
			& \leq \vnorm{ T [\eta_\delta \wt{K}_\delta^* \bff_1] - T [\eta_\delta \wt{K}_\delta^* \bsvarphi_n]  }_{W^{1-1/p,p}(\p \Omega)} \\
            &\leq C \vnorm{  \eta_\delta \wt{K}_\delta^* \bff_1 - \eta_\delta \wt{K}_\delta^* \bsvarphi_n  }_{W^{1,p}(\Omega)}\\
            &
		  \leq C \Vnorm{ \bff_1 - \bsvarphi_n }_{L^{p}(\Omega)} \to 0 \text{ as } n \to \infty\,.
		\end{split}
	\end{equation*}
	Therefore
		$\eta_\delta \wt{K}_\delta^* \bff_1 \in W^{1,p}_0(\Omega)$, and therefore the boundary term is zero. So we can represent $\eta_\delta K_\delta^* f$ as
        \begin{equation*}
    \begin{split}
        \vint{\eta_\delta K_\delta^*f,v} 
        &= \int_{\Omega} \big( \eta_\delta K_\delta^* f_0 + \grad \eta_\delta \cdot \wt{K}_\delta^* \bff_1 - \div [ \eta_\delta \wt{K}_\delta^* \bff_1 ] \big)   v \, \rmd \bx \\
        &= \int_{\Omega} \big( \eta_\delta K_\delta^* f_0 - \eta_\delta \div \wt{K}_\delta^* \bff_1 \big)  v \, \rmd \bx\,, \quad \forall \, v \in W^{1,p'}(\Omega)\,.
    \end{split}
    \end{equation*}
    This action of $\eta_\delta K_\delta^*f$ can be extended by density to $v \in L^{p'}(\Omega)$; we have already noted throughout the proof that each term acting on $v$ in the integrand belongs to $L^p(\Omega)$. Therefore $\eta_\delta K_\delta^* f $ can be identified with this function, and \eqref{eq:ConvEst2:Adj:0:neg1} follows from \eqref{eq:UtilityConv:Est:Lp}, \eqref{eq:UtilityConv:Est:W1p}, and the equivalence $\vnorm{f}_{ [W^{1,p'}(\Omega)]^* } \approx \vnorm{f_0}_{L^p(\Omega)} + \vnorm{\bff_1}_{L^p(\Omega)}$.

    To show \eqref{eq:ConvEst2:Adj:1:neg1}, we use that
    \begin{equation*}
         (\eta_\delta K_\delta^* f)(\bx) = \eta_\delta(\bx) K_\delta^* f_0(\bx) - \eta_\delta(\bx) \div \wt{K}_\delta^* \bff_1(\bx)\,,
    \end{equation*}
    and so clearly $\eta_\delta^2  K_\delta^* f \in L^p(\Omega)$. Directly,
    \begin{equation*}
        \grad [\eta_\delta^2 K_\delta^* f] = 2 \eta_\delta \grad \eta_\delta K_\delta^* f_0 + \eta_\delta^2 \grad K_\delta^* f_0 - 2 \eta_\delta \grad \eta_\delta \div \wt{K}_\delta^* \bff_1 - \eta_\delta^2 \grad[ \div \wt{K}_\delta^* \bff_1 ]\,.
    \end{equation*}
    All of these terms except for the last can be estimated using \eqref{eq:ConvEst:Adj:Lp}, \eqref{eq:ConvEst2:Adj:0:0}, and \eqref{eq:UtilityConv:Est:W1p}. 
    Moreover, the last term can be treated with an argument similar to that used to prove \eqref{eq:UtilityConv:Est:W1p}; taking an additional partial derivative does not complicate the estimation. 
    \end{proof}

\subsection{More general boundary-localized convolutions}
We need a few more estimates on boundary-localized convolution-type operators emerging naturally from the equation $\cL_{p,\delta} u = f$.

For $0 \leq \alpha < d-1$, we define the boundary-localized convolution operator
\begin{equation*}
    J_{\delta,\alpha}[\lambda,q,\psi] u(\bx) := \frac{1}{P_{\delta,\alpha}[\lambda,q,\psi](\bx)}\int_{\Omega} 
    \left( \frac{ \psi_{\delta,\alpha}(\bx,\by) }{ \eta_\delta(\bx)^2 } + \frac{ \psi_{\delta,\alpha}(\by,\bx) }{ \eta_\delta(\by)^2 } \right) u(\by) \, \rmd \by\,,
\end{equation*}
where $\psi_{\delta,\alpha}$ is defined using the abbreviation introduced in \eqref{eq:abbrev} and $P_{\delta,\alpha}[\lambda,q,\psi](\bx)$ is the normalization
\begin{equation*}
    P_{\delta,\alpha}[\lambda,q,\psi](\bx) := \int_{\Omega} 
    \left( \frac{ \psi_{\delta,\alpha}(\bx,\by) }{ \eta_\delta(\bx)^2 } + \frac{ \psi_{\delta,\alpha}(\by,\bx) }{ \eta_\delta(\by)^2 } \right) \, \rmd \by\,.
\end{equation*}
Whenever the context is clear, we adopt the convention \eqref{eq:abbrev} for the abbreviations $J_{\delta,\alpha}[\lambda,q,\psi] = J_{\delta,\alpha}$ and $P_{\delta,\alpha}[\lambda,q,\psi] = P_{\delta,\alpha}$.
\begin{theorem} Let $1 \leq p \leq \infty$.
    There exists a constant $C_0 > 0$ depending only on $d$, $p$,  $\beta$, $\psi$, $\lambda$, $\kappa_1$ and $\alpha$ such that
    \begin{equation}\label{eq:OperatorConv:Lp}
        \vnorm{ J_{\delta,\alpha} u }_{L^p(\Omega)} \leq C_0 \Vnorm{ u }_{L^p(\Omega)}\,, \quad \forall u \in L^p(\Omega)\,.
    \end{equation}
    There additionally exists a constant $C_1 > 0$ depending only on $d$, $\beta$, $p$, $\psi$, $q$, $\alpha$, $\kappa_0$, and $\kappa_1$ such that for all $\delta$ satisfying \eqref{eq:HorizonThreshold}
    \begin{equation}\label{eq:OperatorConv:Energy}
        \vnorm{ \grad J_{\delta,\alpha} u }_{L^{p}(\Omega)} \leq C_1 [u]_{ \mathfrak{W}^{\beta,p}[\delta;q](\Omega) }\,, \quad \forall u \in  \mathfrak{W}^{\beta,p}[\delta;q](\Omega)\,.
    \end{equation}
\end{theorem}

\begin{proof}
    By \eqref{eq:ComparabilityOfDistanceFxn2} and \eqref{eq:KernelIntFunction:Bounds}, 
    \begin{equation}\label{eq:OperatorConv:Pf1}
        1 + \frac{(1-\kappa_1 \delta)^2}{1+\kappa_1 \delta} \leq \frac{ \eta_\delta(\bx)^2 P_{\delta,\alpha}(\bx) }{ \int_{B(0,1)} \frac{ \psi(|\bz|)}{ |\bz|^\alpha } \, \rmd \bz } \leq 1 + \frac{ (1+\kappa_1 \delta)^2}{ 1-\kappa_1 \delta} \,,
    \end{equation}
    and so again by \eqref{eq:ComparabilityOfDistanceFxn2}
    \begin{equation*}
        |J_{\delta,\alpha} u(\bx)| \leq C K_\delta[ \lambda,q, |\cdot|^{-\alpha} \psi ](|u|)(\bx) + C   K_\delta^*[ \lambda,q, |\cdot|^{-\alpha} \psi ](|u|)(\bx)\,.
    \end{equation*}
    The estimate \eqref{eq:OperatorConv:Lp} then follows from \eqref{eq:ConvEst:Lp} and \eqref{eq:ConvEst:Adj:Lp}.
    
    Now, similarly to the proof of \eqref{eq:Intro:ConvEst:Deriv} in \cite{scott2023nonlocal}, we have
    \begin{equation*}
        \begin{split}
            \grad J_{\delta,\alpha} u(\bx) &= \frac{1}{P_{\delta,\alpha}(\bx)}\int_{\Omega} 
            \grad \left( \frac{ \psi_{\delta,\alpha}(\bx,\by) }{ \eta_\delta(\bx)^2 } + \frac{ \psi_{\delta,\alpha}(\by,\bx) }{ \eta_\delta(\by)^2 } \right) (u(\by)-u(\bx)) \, \rmd \by \\
            &\quad -\frac{\grad P_{\delta,\alpha}(\bx)}{P_{\delta,\alpha}(\bx)^2}\int_{\Omega} 
            \left( \frac{ \psi_{\delta,\alpha}(\bx,\by) }{ \eta_\delta(\bx)^2 } + \frac{ \psi_{\delta,\alpha}(\by,\bx) }{ \eta_\delta(\by)^2 } \right) (u(\by)-u(\bx)) \, \rmd \by\,,
        \end{split}
    \end{equation*}
    where we added and subtracted $\frac{\grad P_{\delta,\alpha}(\bx)}{P_{\delta,\alpha}(\bx)} u(\bx)$.
    By a direct computation, we get analogously to  \eqref{eq:KernelDerivativeEstimate} and \eqref{eq:KernelDerivativeEstimate:Adj}
    \begin{equation*}
        \begin{split}
            \left| \grad_\bx \left[ \frac{\psi_{\delta,\alpha}(\bx,\by)}{\eta_\delta(\bx)^2} + \frac{\psi_{\delta,\alpha}(\by,\bx)}{\eta_\delta(\by)^2} \right] \right| \leq C \frac{ \hat{\psi}_{\delta,\alpha+1}(\bx,\by) + \wh{(|\psi|')}_{\delta,\alpha}(\bx,\by) }{ \eta_\delta(\bx)^3 }\,,
        \end{split}
    \end{equation*}
    where we additionally used \eqref{eq:ComparabilityOfDistanceFxn2}. Therefore $|\grad P_{\delta,\alpha}(\bx)| \leq \frac{C}{\eta_\delta(\bx)^3}$, and so 
    \begin{equation*}
        \begin{split}
            |\grad J_{\delta,\alpha} u(\bx)| &\leq C \int_{\Omega}
            \frac{ \hat{\psi}_{\delta,\alpha+1}(\bx,\by) + \wh{(|\psi|')}_{\delta,\alpha}(\bx,\by) }{ \eta_\delta(\bx) } |u(\by)-u(\bx)| \, \rmd \by\,,
        \end{split}
    \end{equation*}
    repeatedly applying \eqref{eq:OperatorConv:Pf1} and \eqref{eq:ComparabilityOfDistanceFxn2} as necessary. Finally we apply H\"older's inequality (noting that $\alpha + 1 < d$) to get
    \begin{equation*}
        \int_{\Omega} 
        |\grad J_{\delta,\alpha} u(\bx)|^p \, \rmd \bx \leq C \int_\Omega \int_\Omega \frac{ \hat{\psi}_{\delta,\alpha+1}(\bx,\by) + \wh{(|\psi|')}_{\delta,\alpha}(\bx,\by) }{ \eta_\delta(\bx)^p } |u(\by)-u(\bx)|^p \, \rmd \by \, \rmd \bx\,,
    \end{equation*}
    and the result follows from the upper bound in \eqref{thm:EnergySpaceIndepOfKernel}.
\end{proof}

We have the following as a consequence of the estimates in the proof:

\begin{corollary}\label{cor:OperatorConv:NormalizedBds}
    There exists constants $C>0$ and $c>0$ depending only on $d$, $\alpha$, $\psi$, and $\kappa_1$ such that
    \begin{equation*}
        c \eta_\delta(\bx)^2 \leq \frac{1}{P_{\delta,\alpha}(\bx)} \leq C \eta_{\delta}(\bx)^2 \quad \text{ and } \quad \left| \frac{\grad P_{\delta,\alpha}(\bx)}{P_{\delta,\alpha}(\bx)^2} \right| \leq C \eta_\delta(\bx)\,, \quad \text{ for a.e. } \, \bx \in \Omega\,.
    \end{equation*}
\end{corollary}

\section{Regularity for the nonlocal problem}\label{sec:Regularity}

For the rest of the paper, we assume that $p = 2$ and that $0 \leq \beta < d-1$, and take $\Phi_2(t) = \frac{t^2}{2}$. 
We assume that $\lambda$ satisfies \eqref{assump:Localization} and \eqref{assump:Localization:NormalDeriv},
$q$ satisfies \eqref{assump:NonlinearLocalization} and \eqref{assump:MomentsOfNonlinLoc}, the resulting function $\eta = q(\lambda)$ satisfies \eqref{assump:FullLocalization:C11}, and $\rho$ satisfies \eqref{assump:VarProb:Kernel} with the  additional conditions $\rho \in C^1(\bbR)$ and $\bar{\rho}_{2-\beta} = \overline{C}_{d,2}$. 
Furthermore, we assume that $\bar{\lambda}$ satisfies \eqref{assump:Localization} and  $\psi$ satisfies \eqref{Assump:Kernel} for $k_\psi \geq 2$, and we assume \eqref{eq:HorizonThreshold}.

We recall the convention $\cL_{2,\delta} = \cL_\delta$ introduced in \Cref{subsec:Examples}
\begin{equation*}
    \cL_{\delta} u(\bx) = \int_{\Omega} 
    \left( \frac{ \rho_{\delta,\beta}(\bx,\by) }{ \eta_\delta(\bx)^2 } + \frac{ \rho_{\delta,\beta}(\by,\bx) }{ \eta_\delta(\by)^2 } \right) (u(\bx) - u(\by)) \, \rmd \by\,,
\end{equation*}
with the kernel $\rho_{\delta,\beta}$ following the convention in \eqref{eq:abbrev}.
By the results from the previous sections, $\cL_\delta u(\bx)$ is an absolutely convergent integral for any $u \in L^1_{loc}(\Omega)$, with $|\cL_\delta u(\bx)| \leq \frac{C}{ \eta_\delta(\bx)^2 } ( |u(\bx)| + K_\delta(|u|)(\bx) +  K_\delta^*(|u|)(\bx) )$, where $K_\delta = K_\delta[\lambda,q,|\cdot|^{-\beta}\rho(\cdot)]$ and $K_\delta^* = K_\delta^*[\lambda,q,|\cdot|^{-\beta}\rho(\cdot)]$.
We also recall $\cB_{2,\delta}(u,v) = \cB_\delta(u,v)$.

\begin{theorem}
    Suppose that $u \in \mathfrak{W}^{\beta,2}[\delta;q](\Omega)$ satisfies
    \begin{equation}\label{eq:Regularity:ELeqn}
        \cB_{\delta}(u,v) = \int_{\Omega} G_\delta[u](\bx) v(\bx) \, \rmd \bx\,, \quad \forall \, v \in C^\infty_c(\Omega)\,,
    \end{equation}
    with a given operator $G_\delta[u]=G_\delta[u](\bx)  \in W^{1,2}_{loc}(\Omega)$. 
    Then $\cL_\delta u(\bx) = G_\delta[u](\bx)$ almost everywhere in $\Omega$, i.e.
    \begin{equation}\label{eq:Regularity:KeyEquality}
        u(\bx) = J_{\delta,\beta}[\lambda,q,\rho] u(\bx) + \frac{ G_\delta[u](\bx) }{ P_{\delta,\beta}[\lambda,q,\rho](\bx) } \text{ almost everywhere in } \Omega\,,
    \end{equation}
    and there exists a constant $C$ depending only on $d$, $\beta$, $\rho$, $p$, $\lambda$, $q$, and $\kappa_1$ such that
    \begin{equation}\label{eq:Regularity:Estimate}
        \Vnorm{\grad u}_{L^{2}(\Omega)} \leq C \left( [u]_{\mathfrak{W}^{\beta,2}[\delta;q](\Omega)} + \vnorm{ \eta_\delta G_\delta[u] }_{L^2(\Omega)} + \vnorm{ \eta_\delta^2 \grad G_\delta[u] }_{L^2(\Omega)} \right)\,.
    \end{equation}
    Consequently,
    $u \in W^{1,2}(\Omega)$ whenever the right-hand side of \eqref{eq:Regularity:Estimate} is finite; in particular, if $G_\delta[u](\bx)$ satisfies the appropriate weighted Sobolev estimates.
\end{theorem}

\begin{proof}
    By \Cref{thm:PointwiseOperator:AbsConv:Energy}, the identity \eqref{eq:GreensIdentity:StrongL} can be applied to get
    \begin{equation*}
        \int_{\Omega} v(\bx) \cL_{\delta}u(\bx) \, \rmd \bx = \int_{\Omega} v(\bx) G_\delta[u](\bx) \, \rmd \bx\,, \quad \forall \, v \in C^\infty_c(\Omega)\,.
    \end{equation*}
    Thus \eqref{eq:Regularity:KeyEquality} holds. Differentiating this equation gives
    \begin{equation*}
        \grad u(\bx) = \grad J_{\delta,\beta}u(\bx) + \frac{\grad P_{\delta,\beta}(\bx)}{P_{\delta,\beta}(\bx)^2} G_\delta[u](\bx) + \frac{\grad G_\delta[u](\bx)}{ P_{\delta,\beta}(\bx) }\,.
    \end{equation*}
    The result then follows from \Cref{cor:OperatorConv:NormalizedBds},  \eqref{eq:OperatorConv:Energy} and the assumption on $G_\delta[u]$.
\end{proof}

\section{Convergence of solutions in stronger norms}\label{sec:Convergence}

In this section, with all of the assumptions of \Cref{sec:Regularity}, we revisit the $\delta \to 0$ regime for the nonlocal problems in the linear case.
Note that while an element of $[W^{1,2}(\Omega)]^*$ can serve as the Poisson data $f$ for the local problems \eqref{eq:BVP:Dirichlet:Local}, \eqref{eq:BVP:Neumann:Local}, and \eqref{eq:BVP:Robin:Local},  a properly regularized version, denoted by $f_\delta$, has to be used for a well-posed nonlocal problem.
We first show that the sequence of solutions to the nonlocal problem with the regularized $f_\delta$ converges as $\delta \to 0$ to the unique variational solution of 
\eqref{eq:BVP:Dirichlet:Local} (resp. \eqref{eq:BVP:Neumann:Local} and \eqref{eq:BVP:Robin:Local}) with the Poisson data $f$, all subject to corresponding inhomogeneous boundary conditions. We then illustrate how the regularization can be obtained. Further, we present the consistency of the outward normal derivative on $\p \Omega$.

First, we note the behavior of the lower-order term characterized by the function $\ell$ that satisfies \eqref{eq:LowerOrderTerm}; Then for any $u$, $v$, and $w \in L^m(\Omega)$ there exists a $C$ depending only on $\ell$ such that
\begin{equation}\label{lma:LowerOrderTerm}
    \begin{split}
        \int_{\Omega} & |\ell'(u)-\ell'(v)| |w| \, \rmd \bx \\ & \leq 
        \begin{cases}
            C (\Vnorm{u}_{L^m(\Omega)}^{m-2} + \Vnorm{v}_{L^m(\Omega)}^{m-2}) \Vnorm{ u-v }_{L^m(\Omega)} \Vnorm{ w }_{L^m(\Omega)}, & \text{ if } m \geq 2\,, \\
            C \Vnorm{ u - v }_{L^m(\Omega)}^{m-1} \Vnorm{ w }_{L^m(\Omega)}\,, & \text{ if } 1 < m \leq 2\,.
        \end{cases}
    \end{split}
    \end{equation}
Indeed, the case $1 < m \leq 2$ follows from \eqref{eq:LowerOrderTerm} and H\"older's inequality with exponents $\frac{m}{m-1}$ and $m$. The case $m \geq 2$ follows from \eqref{eq:LowerOrderTerm} and the generalized H\"older's inequality with exponents $\frac{m}{m-2}$, $\frac{m}{2}$, and $\frac{m}{2}$.

\begin{theorem}
    Let $f \in [W^{1,2}(\Omega)]^*$. Suppose that for each $\delta < \underline{\delta}_0$ there exists $f_\delta \in [\mathfrak{W}^{\beta,2}[\delta;q](\Omega)]^*$ that satisfies
    \begin{equation}\label{eq:H1Convergence:RHSConvergence}
        \lim\limits_{\delta \to 0} \vint{f_\delta -f, v} = 0 \,, \quad \forall \, v \in W^{1,2}(\Omega)\,,
    \end{equation}
    and that there exists $C_0 > 0$ independent of $\delta$ such that
    \begin{equation}\label{eq:H1Convergence:RHSEstimate}
        \vnorm{f_\delta}_{ [\mathfrak{W}^{\beta,2}[\delta;q](\Omega)]^* } + \vnorm{ \eta_\delta f_\delta}_{L^2(\Omega)} + \vnorm{\eta_\delta^2 \grad f_\delta}_{L^2(\Omega)} \leq C_0 \vnorm{f}_{ [W^{1,2}(\Omega)]^* }\,.
    \end{equation}
    Let $u_\delta$ be the solution in the respective admissible set of either \eqref{eq:BVP:Dirichlet:Weak}, \eqref{eq:BVP:Neumann:Weak}, or \eqref{eq:BVP:Robin:Weak}, with Poisson data $f_\delta$ and boundary data $g$. Then there exists $C$ independent of $u_\delta$ and $\delta$ such that
    \begin{equation}\label{eq:H1BoundOnSolutions}
        \vnorm{u_\delta}_{W^{1,2}(\Omega)} \leq C (\vnorm{f}_{ [W^{1,2}(\Omega)]^* } + \vnorm{g} ) + C (\vnorm{f}_{ [W^{1,2}(\Omega)]^* } + \vnorm{g} )^{m-1} \,,
    \end{equation}
    where $\vnorm{g}$ denotes $\vnorm{g}_{W^{1/2,2}(\p \Omega_D)}$ in the case \eqref{eq:BVP:Dirichlet:Weak} and $\vnorm{g}_{[W^{1/2,2}(\p \Omega)]^*}$ in the cases \eqref{eq:BVP:Neumann:Weak} and \eqref{eq:BVP:Robin:Weak}.

    Moreover, in the case of \eqref{eq:BVP:Dirichlet:Weak} with $g = 0$, the above result additionally holds with the same assumptions, 
    but with the spaces $[W^{1,2}(\Omega)]^*$ and $[\mathfrak{W}^{\beta,2}[\delta;q](\Omega)]^*$ replaced by the spaces $[W^{1,2}_{0,\p \Omega_D}(\Omega)]^*$ and $[\mathfrak{W}^{\beta,2}_{0, \p \Omega_D} [\delta;q](\Omega)]^*$ respectively.
\end{theorem}

\begin{proof}
    First, in the case that $u_\delta$ satisfies \eqref{eq:BVP:Dirichlet:Weak}, since $C^\infty_c(\Omega) \subset \mathfrak{W}^{\beta,2}_{0,\p \Omega_D}[\delta;q](\Omega)$ we have
    \begin{equation}\label{eq:H1Convergence:ELeqn:Pf}
        \cB_{\delta}(u_\delta,v) = \vint{ f_\delta - \mu K_\delta^* \big( \ell'(K_\delta u_\delta) \big), v}\,, \quad \forall v \in C^\infty_c(\Omega)\,.
    \end{equation}
    Therefore \eqref{eq:Regularity:ELeqn}, \eqref{eq:Regularity:KeyEquality}, and \eqref{eq:Regularity:Estimate} hold with $G_\delta[u](\bx) = f_\delta - \mu K_\delta^* \big( \ell'(K_\delta u_\delta) \big)$.
    Now, in \eqref{eq:Regularity:Estimate} we apply the energy estimate \eqref{eq:EnergyEstimate:Dirichlet}, and then apply \eqref{eq:H1Convergence:RHSEstimate}. At the same time, we apply \eqref{eq:ConvEst:Adj:Nonlocal}, \eqref{eq:ConvEst2:Adj:0:neg1}, and \eqref{eq:ConvEst2:Adj:1:neg1} to the $u$-term, and get from \eqref{eq:Regularity:Estimate} the inequality
    \begin{equation*}
        \vnorm{u_\delta}_{W^{1,2}(\Omega)} \leq C \big( \vnorm{f}_{[W^{1,2}(\Omega)]^*} + \vnorm{g}_{ W^{1/2,2}(\p \Omega_D) } +  \vnorm{\ell'(K_\delta u_\delta)}_{[W^{1,2}(\Omega)]^*} \big)\,.
    \end{equation*}
    Recalling \eqref{eq:SobolevExponent}, for any $\varphi \in W^{1,2}(\Omega)$ we apply \eqref{lma:LowerOrderTerm} and the classical Sobolev embedding $W^{1,2}(\Omega) \hookrightarrow L^m(\Omega)$ to get
    \begin{equation*}
        |\vint{ \ell'(K_\delta u_\delta), \varphi }| 
        \leq C \vnorm{ K_\delta u_\delta }_{W^{1,2}(\Omega)}^{m-1} \vnorm{\varphi}_{W^{1,2}(\Omega)}\,.
    \end{equation*}
    Thus the estimates \eqref{eq:Intro:ConvEst:Deriv} and \eqref{eq:EnergyEstimate:Dirichlet} give
    \begin{equation*}
        \vnorm{\ell'(K_\delta u_\delta)}_{[W^{1,2}(\Omega)]^*} 
        \leq C \vnorm{ K_\delta u_\delta }_{W^{1,2}(\Omega)}^{m-1} 
        \leq C ( \vnorm{f}_{[W^{1,2}(\Omega)]^*} + \vnorm{g}_{ W^{1/2,2}(\p \Omega_D) }  )^{m-1}\,.
    \end{equation*}
    So \eqref{eq:H1BoundOnSolutions} holds in the case that $u_\delta$ satisfies \eqref{eq:BVP:Dirichlet:Weak}. The proof in the case that $g =0$ on $\p \Omega_D$ and the function spaces replaced with the homogeneous spaces is exactly the same, using the corresponding estimates for those spaces.

    Next, in the case that $u_\delta$ satisfies \eqref{eq:BVP:Neumann:Weak}, we need to show suffices to show \eqref{eq:H1Convergence:ELeqn:Pf} (in this case $\mu = 0$).
    Given $v \in C^\infty_c(\Omega)$, define $w = v - (v)_\Omega$. Then from \eqref{eq:BVP:Neumann:Weak} and the compatibility condition we have
    \begin{equation*}
    \begin{split}
        \cB_{\delta}(u_\delta,v) = \cB_{\delta}(u_\delta,w) 
        &= \vint{f,w} + \vint{g,w} \\
        &= \vint{f,v} + \vint{g,v} - (\vint{f,1} + \vint{g,1}) (v)_\Omega = \vint{f,v} + \vint{g,v}\,.
    \end{split}
    \end{equation*}
    Thus \eqref{eq:H1Convergence:ELeqn:Pf} holds, and then the proof proceeds exactly the same way as in the Dirichlet case, with \eqref{eq:EnergyEstimate:Neumann} in place of \eqref{eq:EnergyEstimate:Dirichlet}.
    Finally, the case that $u_\delta$ satisfies \eqref{eq:BVP:Robin:Weak} is exactly the same as the Dirichlet case (with \eqref{eq:EnergyEstimate:Robin} in place of \eqref{eq:EnergyEstimate:Dirichlet}), since $C^\infty_c(\Omega)$ is a subset of the test functions.  
\end{proof}

\begin{theorem}
    Let $f \in [W^{1,2}(\Omega)]^*$. Suppose that $\{ f_\delta \}_{\delta} \subset [\mathfrak{W}^{\beta,2}[\delta;q](\Omega)]^*$ is a sequence that satisfies \eqref{eq:H1Convergence:RHSConvergence}-\eqref{eq:H1Convergence:RHSEstimate}. Let $\{u_\delta\}_\delta$ be the sequence of solutions in the respective admissible set to either \eqref{eq:BVP:Dirichlet:Weak}, \eqref{eq:BVP:Neumann:Weak}, or \eqref{eq:BVP:Robin:Weak} with Poisson data $f_\delta$ and boundary data $g$. Then there exists $u \in W^{1,2}(\Omega)$ such that $u_\delta \rightharpoonup u$ weakly in $W^{1,2}(\Omega)$, and $u$ is the unique admissible solution to either \eqref{eq:BVP:Dirichlet:Local}, \eqref{eq:BVP:Neumann:Local}, or \eqref{eq:BVP:Robin:Local}, respectively, with Possion data $f$ and boundary data $g$.

    The same result holds with data in the wider class when $g = 0$ in the case of \eqref{eq:BVP:Dirichlet:Weak}.
\end{theorem}

\begin{remark}
    Note that the convergence results in \Cref{thm:WellPosedness:Dirichlet}, \Cref{thm:WellPosedness:Neumann}, and \Cref{thm:WellPosedness:Robin} hold only in the strong topology on $L^p(\Omega)$, and hold only for Poisson data $f_\delta = K_\delta^* f$.
    The convergence result here both strengthens the topology and generalizes the class of approximating Poisson data in the case of the linear problem.
\end{remark}

\begin{proof}
    Since the solutions satisfy the uniform $W^{1,2}(\Omega)$ bound \eqref{eq:H1BoundOnSolutions}, it follows that they converge weakly in $W^{1,2}(\Omega)$ to a function $u$. 
    
    In all three cases, $\vint{f_\delta, v} \to \vint{f,v}$ for all test functions $v$ by assumption. Moreover,    the class of test functions for the local problem is a subspace of the class of test functions for the nonlocal problem. 
    Therefore $\cB_{\delta}(u_\delta,v) \to \cB_0(u,v)$ for any test function $v$ for the local problem by \Cref{thm:BilinearFormLocalization} and \Cref{thm:BilinearFormLocalization2}. 
    
    Since the embedding $W^{1,2}(\Omega) \hookrightarrow L^m(\Omega)$ is compact, the estimate of \eqref{lma:LowerOrderTerm} and the continuity estimates \eqref{eq:ConvEst:Lp} and \eqref{eq:ConvEst:Adj:Lp} show that the lower-order term converges strongly as $\delta \to 0$ in both the Dirichlet and Robin cases.

    Finally, in the Dirichlet case, since $T u_\delta = g \in W^{\frac{1}{2},2}(\p \Omega)$ for all $\delta$, we have by $W^{1,2}(\Omega)$-weak continuity of the traces that $T u = g$. 
\end{proof}

The next theorem gives an explicit construction of the mollified sequence $\{ f_{\delta} \}_\delta$ that satisfies \eqref{eq:H1Convergence:RHSConvergence}-\eqref{eq:H1Convergence:RHSEstimate}.

\begin{theorem}\label{thm:DataMollification}
The distribution $f_{\delta} := K_{\delta}^*[\bar{\lambda},q,\psi] f$ satisfies \eqref{eq:H1Convergence:RHSConvergence}-\eqref{eq:H1Convergence:RHSEstimate}.
\end{theorem}

\begin{proof}
    The result follows from \eqref{eq:ConvEst:Adj:Nonlocal}, \eqref{eq:ConvEst2:Adj:0:neg1}, and \eqref{eq:ConvEst2:Adj:1:neg1}.
\end{proof}

\section{Convergence of normal derivatives}\label{sec:ConvOfNormals}

Remaining under the assumptions of \Cref{sec:Regularity},
we show that the ``approximate'' normal derivatives of $u_\delta$ converge to the corresponding normal form of $u$.

\begin{theorem}\label{thm:NormalDerivativeDef:Dirichlet}
    Let $f \in L^2(\Omega)$, and assume that $\{f_\delta\}_\delta \subset L^2(\Omega)$ satisfies \eqref{eq:H1Convergence:RHSConvergence}-\eqref{eq:H1Convergence:RHSEstimate}.
    Let $u_\delta$ in the admissible set solve \eqref{eq:BVP:Dirichlet:Weak} with Poisson data $f_\delta$ and boundary data $g \in W^{1/2,2}(\p \Omega)$.
    Define the distribution $Z_\delta$ as
    \begin{equation}\label{eq:zdelta}
        \vint{ Z_\delta, v } := \cB_{\delta}(u_\delta, \bar{v}) + \mu \int_{\Omega} \ell'(K_\delta u_\delta) K_\delta \bar{v} \, \rmd \bx - \int_{\Omega} f_\delta \bar{v} \, \rmd \bx\,, \quad v \in W^{1/2,2}(\p \Omega)\,,
    \end{equation}
    where $\bar{v}$ is any $W^{1,2}(\Omega)$-extension of $v$ to $\Omega$.
    Then for any $\delta$, $Z_\delta$ defined via \eqref{eq:zdelta} is independent of the choice of extension, and
    $Z_\delta \in [W^{1/2,2}(\p \Omega)]^*$, with 
    \begin{equation}\label{eq:NonlocalNormalDeriv:Estimate}
        \Vnorm{ Z_\delta }_{ [W^{1/2,2}(\p \Omega)]^* } \leq C \Vnorm{ f }_{L^2(\Omega)}\,.
    \end{equation}
\end{theorem}

\begin{proof}
    Let $\bar{v}_1$ and $\bar{v}_2$ be two extensions of $v \in W^{1/2,2}(\p \Omega)$. Then $T(\bar{v}_1 -\bar{v}_2) = 0$, and so $\bar{v}_1 - \bar{v}_2 \in \mathfrak{W}^{\beta,p}_{0,\p \Omega}[\delta;q](\Omega)$ (see \Cref{thm:FxnSpaceProp}). Thus, it can be used as a test function in \eqref{eq:BVP:Dirichlet:Weak}, and gives
    \begin{equation*}
        \cB_{\delta}(u_\delta,\bar{v}_1 - \bar{v}_2) = - \mu \int_{\Omega} \ell'(K_\delta u_\delta) K_\delta [\bar{v}_1 - \bar{v}_2] \, \rmd \bx + \int_{\Omega} f_\delta (\bar{v}_1 - \bar{v}_2) \, \rmd \bx\,.
    \end{equation*}
    Therefore $\vint{Z_\delta, T(\bar{v}_1)} = \vint{Z_\delta, T(\bar{v}_2)}$. Now, to see the estimate \eqref{eq:NonlocalNormalDeriv:Estimate}, use H\"older's inequality, \eqref{thm:EnergySpaceIndepOfKernel}, the growth condition on $\ell'$, and the Sobolev embedding theorem with \eqref{eq:Intro:ConvEst:Deriv}, 
    all on the right-hand side of \eqref{eq:zdelta} to get
    \begin{equation*}
        |\vint{Z_\delta,v}| \leq C \Vnorm{ u_\delta }_{\mathfrak{W}^{\delta,2}[\delta;q](\Omega)} \Vnorm{ \bar{v} }_{\mathfrak{W}^{\delta,2}[\delta;q](\Omega)} + \Vnorm{f_\delta}_{[\mathfrak{W}^{\delta,2}[\delta;q](\Omega)]^*} \Vnorm{\bar{v}}_{\mathfrak{W}^{\delta,2}[\delta;q](\Omega)}\,.
    \end{equation*}
    Then by the energy estimate \eqref{eq:EnergyEstimate:Dirichlet}, \eqref{eq:Embedding}, and \eqref{eq:H1Convergence:RHSEstimate},
    \begin{equation*}
        |\vint{Z_\delta,v}| \leq C \Vnorm{f_\delta}_{[\mathfrak{W}^{\delta,2}[\delta;q](\Omega)]^*} \Vnorm{\bar{v}}_{W^{1,2}(\Omega)} \leq C \Vnorm{f}_{L^2(\Omega)} \Vnorm{v}_{W^{1/2,2}(\p \Omega)}\,.
    \end{equation*}
\end{proof}

\begin{theorem}\label{thm:NormalDerivativeConv:Dirichlet}
    With all the assumptions of \Cref{thm:NormalDerivativeDef:Dirichlet}, assume additionally that the Dirichlet data $g \in W^{3/2,2}(\p \Omega)$. Let $u$ be the weak solution of \eqref{eq:BVP:Dirichlet:Local} in the admissible set, which is the $W^{1,2}(\Omega)$-weak limit of the weak solutions $\{u_\delta\}_\delta$ of \eqref{eq:BVP:Dirichlet:Weak} in the respective admissible set. Then 
    the sequence $\{ Z_\delta \}_\delta$ defined via \eqref{eq:zdelta} converges to  $\frac{\p u}{\p \bsnu}$ weakly in $[W^{1/2,2}(\p \Omega)]^*$.
\end{theorem}

\begin{proof}
    Thanks to the regularity of $f$ and $g$, and since $u$ is the unique weak solution of \eqref{eq:BVP:Dirichlet:Local}, it follows that $u \in W^{2,2}(\Omega)$, with $\frac{\p u}{\p \bsnu} \in W^{1/2,2}(\p \Omega)$; see for instance \cite[Theorem 4.14]{Giaquinta}.
    Then $(\mathrm{w}\text{-}\cL)_{\delta} u$ exists, and so by \eqref{eq:GreensIdentity} in the context of \Cref{thm:WeakOperator}
    \begin{equation}\label{eq:NormalDerivConv:D:Pf1}
        \int_{\Omega} (\mathrm{w}\text{-}\cL)_\delta u \, v \, \rmd \bx = \cB_{\delta}(u,v) 
        - \int_{\p \Omega} BF_{2,\delta}^N(\grad u,\bsnu) v \, \rmd \sigma\,, \quad \forall v \in W^{1,2}(\Omega)\,.
    \end{equation}
    Observe that, in this case, $BF_{2,\delta}^N(\grad u,\bsnu) \to \frac{\p u}{\p \bsnu}
    $ strongly in $L^2(\p \Omega)$ by the dominated convergence theorem (in fact equality holds if $N > 1$).
    So by \Cref{thm:BilinearFormLocalization2} and the classical Green's identity,
    \begin{equation*}
        \begin{split}
          &  \lim\limits_{\delta \to 0} \cB_{\delta}(u,v) - \int_{\p \Omega} BF_{2,\delta}^N(\grad u,\bsnu) v \, \rmd \sigma = \cB_{0}(u,v) - 
            \int_{\p \Omega} \frac{\p u}{\p \bsnu} v \, \rmd \bx \\
            &\qquad = 
            \int_{\Omega} \grad u \cdot \grad v \, \rmd \bx - \int_{\p \Omega} \frac{\p u}{\p \bsnu} v \, \rmd \bx
            =\int_{\Omega} (- \Delta u) v \, \rmd \bx\,.
        \end{split}
    \end{equation*}
    Thus $(\mathrm{w}\text{-}\cL)_{\delta} u \rightharpoonup -\Delta u$ weakly in $[W^{1,2}(\Omega)]^*$.
    At the same time, $u \in W^{2,2}(\Omega)$ is a strong solution of \eqref{eq:BVP:Dirichlet:Local}, i.e.
    \begin{equation}\label{eq:NormalDerivConv:D:Pf2}
        -\Delta u(\bx) + \mu \ell'(u(\bx)) = f(\bx)\,, \text{ almost everywhere in } \Omega\,.
    \end{equation}
    
    Now, for any $v \in W^{1/2,2}(\p \Omega)$ and any $W^{1,2}(\Omega)$ extension $\bar{v}$ of $v$, we use \eqref{eq:NormalDerivConv:D:Pf1} 
    and \eqref{eq:NormalDerivConv:D:Pf2}
    to get
    \begin{equation*}
        \begin{split}
            \Vint{ Z_\delta - 
            \frac{\p u}{\p \bsnu} , v }
            &= \cB_{\delta}(u_\delta - u, \bar{v}) + \mu \int_{\Omega} \ell'(K_\delta u_\delta) K_\delta \bar{v} - \ell'(u) \bar{v} \, \rmd \bx \\
                &\quad + \int_{\Omega} (f-f_\delta) \bar{v} \, \rmd \bx + \int_{\Omega} \big( (\mathrm{w}\text{-}\cL)_\delta u - (-\Delta u) \big) \bar{v} \, \rmd \bx \\
                &\quad + \int_{\p \Omega} \left(  BF_{2,\delta}^N(\grad u,\bsnu) - 
                \frac{\p u}{\p \bsnu} \right) v \, \rmd \sigma\,.
        \end{split}
    \end{equation*}
    As $\delta \to 0$, the last two terms vanish on account of the discussion above, the third term vanishes by the assumption \eqref{eq:H1Convergence:RHSConvergence} on $f_\delta$ and $f$, and the first term vanishes by \Cref{thm:BilinearFormLocalization}. To see that the second term vanishes, use the compact embedding of $W^{1,2}(\Omega)$ into $L^m(\Omega)$, along with \eqref{lma:LowerOrderTerm}, \eqref{eq:ConvEst:Lp} and \eqref{eq:ConvEst:Adj:Lp}.
\end{proof}

\begin{remark}
    The assumption $g \in W^{3/2,2}(\p \Omega)$ was made in order to easily obtain the strong $L^2$ convergence of $(\mathrm{w}\text{-}\cL)_{\delta} u $ to $-\Delta u$. It would be natural to consider the convergence of normal derivatives in the weak topology of $W^{1/2,2}(\p \Omega)$, but such a result requires additional regularity beyond $W^{1,2}$ on the nonlocal solutions $u_\delta$; this will be investigated in future works.
\end{remark}

In the case of Robin boundary conditions, we can improve the convergence to $W^{1/2,2}(\Omega)$ in a special case.

\begin{theorem}\label{thm:NormalDerivative:Robin}
        Let $f \in L^2(\Omega)$, and assume that $\{f_\delta\}_\delta \subset L^2(\Omega)$ satisfies \eqref{eq:H1Convergence:RHSConvergence}-\eqref{eq:H1Convergence:RHSEstimate}.
    Let $u_\delta$ solve \eqref{eq:BVP:Robin:Weak} under the assumptions for well-posedness of \Cref{thm:WellPosedness:Robin}, with Poisson data $f_\delta$ and boundary data $g \in [W^{1/2,2}(\p \Omega)]^*$.
    Define the distribution $Z_\delta$ as in \eqref{eq:zdelta}.
    Then for any $\delta$, $Z_\delta$ defined via \eqref{eq:zdelta} is independent of the choice of extension, and
    $Z_\delta \in [W^{1/2,2}(\p \Omega)]^*$ with
    \begin{equation}\label{eq:NonlocalNormalDeriv:Estimate:Robin}
        \Vnorm{ Z_\delta }_{ [W^{1/2,2}(\p \Omega)]^* } \leq C ( \Vnorm{ f }_{L^2(\Omega)} + \Vnorm{ g }_{[W^{1/2,2}(\p \Omega)]^*} )\,.
    \end{equation}

    Under the additional assumptions that $g \in W^{1/2,2}(\p \Omega)$ and that $b \equiv b_0$ is a nonnegative constant on all of $\p \Omega$, let $u$ be the weak solution of \eqref{eq:BVP:Robin:Local}, which is the $W^{1,2}(\Omega)$-weak limit of the solutions $\{u_\delta\}_\delta$ of \eqref{eq:BVP:Robin:Weak}.
    Then the sequence $\{ Z_\delta \}_\delta$ defined via \eqref{eq:zdelta} converges to $\frac{\p u}{\p \bsnu}$ weakly in $W^{1/2,2}(\p \Omega)$.
\end{theorem}

\begin{proof}
    The results of the previous theorems follow using almost the same arguments, and the differences offer no additional difficulty.
    
    To see the convergence result, it can first be shown using the definition \eqref{eq:zdelta} and the equation \eqref{eq:BVP:Robin:Weak} that if $g \in W^{1/2,2}(\p \Omega)$ then $Z_\delta = g - b_0 u_\delta \in W^{1/2,2}(\p \Omega)$ for $\scH^{d-1}$-almost everywhere in $\p \Omega$. In the same way, $\frac{\p u}{\p \bsnu} = g - b_0 u \in W^{1/2,2}(\p \Omega)$. If $b_0 = 0$ then there is nothing to show, and if $b_0 > 0$ then $Z_\delta - \frac{\p u}{\p \bsnu} = b_0 T(u_\delta - u)$, and $ T u_\delta \rightharpoonup T u$ weakly in $W^{1/2,2}(\p \Omega)$ since $u_\delta \rightharpoonup u$ weakly in $W^{1,2}(\Omega)$.
\end{proof}

\section{Conclusion}

We have presented a study of nonlocal operators with heterogeneous localization and the associated boundary value problems. 
A key component is the proof of a nonlocal Green's identity, which was then used to place boundary-value problems into correspondence with weak forms,
while the well-posedness of these weak formulations and their convergence in the nonlocal-to-local limit is obtained via the analysis in \cite{scott2023nonlocal}. 
The scaling of the kernels, and the range of $\beta$, have allowed us to treat simultaneously both fractional and convolution-type problems, with the same class of boundary information.
Through the additional estimates on the adjoint operators and generalizations presented here, studies initiated in \cite{scott2023nonlocal} on boundary-localized convolutions bear more fruit in the analysis of nonlocal problems and the connections of the latter to classical problems.
In special cases we obtain regularity and improved convergence of solutions in classical Sobolev spaces.

We remark on the situation in the regime that $\beta \to d + p$. One can determine, through careful reading of the proofs in \cite{scott2023nonlocal} and this work, that many of the nonlocal-to-local convergence results remain true -- with correspondingly different statements -- in this so-called ``fractional-to-classical'' limit. An exception is the analysis from the last sections in this work, which are valid only for $\beta < d-1$.

While the analytical studies given in \cite{scott2023nonlocal} and here, as well as the earlier works \cite{tian2017trace,du2022fractional}, have started to lay a rigorous foundation for the use of local boundary conditions in nonlocal problems,
the ``nonlocal calculus'' for the operators treated in this work is, as of now, somewhat incomplete. A further understanding of the nonlocal operators in more general function spaces is needed, so that the validity of identities such as \eqref{eq:GreensIdentity} can be established systematically rather than developed case by case. 
Establishing suitable regularity properties for the models in this work would make significant progress in such a program, in addition to their importance for mathematical theory and physical consistency.

We note that operators with similar boundary-localizing properties to $K_\delta$ and $K_\delta^*$ were first -- to our knowledge -- studied in \cite{burenkov1998sobolev,burenkov1982mollifying}, and later in \cite{hintermuller2020variable}. It would be relevant to carry out studies of models similar to ours but with ``variable-step'' localization; this would, for instance, broaden the scope of implementing such models in coupled local-nonlocal problems \cite{tao2019nonlocal}.
At the same time, establishing convergence rates of various quantities of interest -- in both those models and the models in this work -- is relevant for their use in implementation.

\section*{Acknowledgements}

The authors thank Solveig Hepp and Moritz Kassmann for helpful and inspiring discussions.

\bibliographystyle{siamplain}
\bibliography{References2}

\end{document}